\documentclass[12pt]{article}

\usepackage[utf8]{inputenc}
\usepackage{xcolor}
\usepackage{amsfonts}
\usepackage{amsmath}
\usepackage{amsthm}
\usepackage{thm-restate}
\usepackage{enumitem}
\usepackage{comment}
\usepackage{hyperref}
\usepackage{float}
\usepackage[parfill]{parskip}
\usepackage{geometry}

\declaretheorem[name = Theorem, numberwithin = section]{theorem}
\declaretheorem[name =  Lemma, numberwithin = section]{lemma}
\declaretheorem[name = Proposition, numberwithin = section]{proposition}

\declaretheorem[name = Remark, numberwithin = section, style = remark]{remark}
\declaretheorem[name =  Definition, numberwithin = section, style = definition]{definition}
\numberwithin{equation}{section}
\newcommand{\Secref}[1]{Section~\ref{#1}}

\DeclareMathOperator{\tr}{tr}
\let\div\relax
\DeclareMathOperator{\div}{div}
\DeclareMathOperator{\Ric}{Ric}
\DeclareMathOperator{\Scal}{Scal}
\DeclareMathOperator{\Hess}{Hess}
\DeclareMathOperator{\M}{\mathcal{M}}

\newcommand{\LB}{[\![}
\newcommand{\RB}{]\!]}
\newcommand{\ssubset}{\subset\joinrel\subset}
\newcommand{\mres}{\mathbin{\vrule height 1.6ex depth 0pt width
0.13ex\vrule height 0.13ex depth 0pt width 1.3ex}}

\title{On the existence and properties of solutions of the generalized Jang equation with respect to asymptotically anti-de Sitter initial data}
\author{Benjamin Meco \\ Uppsala University \\ \href{mailto:benjamin.meco@math.uu.se}{benjamin.meco@math.uu.se}}
\date{}

\begin{document}

\maketitle

\begin{abstract}
    We provide a rigorous analysis of the generalized Jang equation in the asymptotically anti-de Sitter setting modelled on constant time slices of anti-de Sitter spacetimes in dimensions $3\leq n \leq 7$ for a very general class of asymptotics. Potential applications to spacetime positive mass theorems for asymptotically anti-de Sitter initial data sets are discussed.
\end{abstract}




\section{Introduction}
\label{sec_Introduction}
In Mathematical General Relativity, initial data sets are triples $(M,g,K)$ where $(M,g)$ is a Riemannian manifold and $K$ is a symmetric $2$-tensor. An initial data set $(M,g,K)$ models a slice of a spacetime $(N,h)$ such that the induced metric is $g$ and the second fundamental is $K$. Roughly speaking, an initial data set can be described as
\begin{itemize}
    \item \textit{asymptotically Euclidean} if it is asymptotic to the $\{t = 0\}$ slice of Minkowski spacetime which may be seen as the initial data set $(\mathbb R^n,\delta, 0)$,
    \item \textit{asymptotically hyperbolic} or \textit{hyperboloidal} if it is asymptotic to the upper unit hyperboloid in Minkowski spacetime which may be seen as the initial data set $(\mathbb H^n,b,b)$, where $\mathbb H^n$ is hyperbolic space and $b$ is the standard hyperbolic metric,
    \item \textit{asymptotically anti-de Sitter} if it is asymptotic to the $\{t = 0\}$ slice of anti-de Sitter spacetime which may be seen as the initial data set $(\mathbb H^n,b,0)$. 
\end{itemize}
The mass of an initial data set $(M,g,K)$ is defined by comparing it with a model initial data set $(M_0,g_0,K_0)$ seen as a slice of some model spacetime $(N_0,h_0)$. The expressions that have been derived using Hamiltonian formalism usually take form of a flux integral at infinity that is well-behaved under isometries of $(N_0,h_0)$. In all physically reasonable matter models, the so called \emph{dominant energy condition} is satisfied and so it is standard to assume it when studying initial data sets. A general positive mass theorem may be stated as follows: "If an initial data set $(M,g,K)$ has well-defined mass and satisfies the dominant energy condition, then its mass is non-negative and is equal to zero if and only if $M$ embeds in the model spacetime as a spacelike slice with the induced metric $g$ and the second fundamental form of the embedding equal to $K$".

The first result in this direction was obtained by Schoen and Yau \cite{SchoenYau1} for asymptotically Euclidean $3$-dimensional initial data sets with $K \equiv 0$,  the so called Riemannian case. Subsequently, this result was extended by Schoen and Yau  in \cite{SchoenYau2} to the case of general $3$-dimensional asymptotically Euclidean initial data sets. Around the same time, Witten \cite{Witten} proved a positive mass theorem for asymptotically Euclidean initial data sets in all dimensions under the assumption that the manifold be spin, which imposes a topological restriction in dimensions $n>3$. Many new results have recently been proven in the asymptotically Euclidean setting. In particular, the result of \cite{SchoenYau2} has been extended to hold in dimensions $3 < n \leq 7$, see Eichmair \cite{EichmairReduction} and Eichmair, Huang, Lee and Schoen \cite{EichmairHuangLeeSchoen}, using in particular methods from geometric measure theory. In addition, a positive mass theorem has been shown to hold for a more general class of asymptotically Euclidean initial data sets which may, in addition to the distinguished asymptotically Euclidean end, have other asymptotic ends, see Lesourd, Unger and Yau \cite{LesourdUngerYau}. There are also recent proofs by so-called level set methods. See for instance Bray, Khuri, Kazaras, and Stern \cite{BrayKhuriKazarasStern} and Hirch, Kazaras, and Khuri \cite{HirschKazarasKhuri} using level sets of linearly growing harmonic functions and generalizations thereof. See also Agostiniani, Mazzieri and Oronzio \cite{AgostinianiMazzieriOronzio} for a proof in the Riemannian case using level sets of the Green function of the
manifold $(M,g)$.

The first definitions of mass in the asymptotically hyperbolic setting were given by Wang \cite{Wang}, for a class of conformally compact Riemannian manifolds satisfying rather stringent asymptotic conditions, and by Chru\'sciel and Nagy \cite{ChruscielNagy} for asymptotically anti-de Sitter initial data sets. Some key properties of these definitions, in particular coordinate invariance, were established  by Chru\'sciel and Herzlich \cite{ChruscielHerzlich}. Proofs of related positive mass theorems were obtained, under spin assumption, by Wang \cite{Wang}, Chru\'sciel and Herzlich \cite{ChruscielHerzlich}, and Zhang \cite{Zhang}. The first non-spinor proof in the asymptotically hyperbolic setting  was given by Andersson, Cai, and Galloway \cite{AnderssonCaiGalloway} for asymptotically hyperbolic manifolds with asymptotics as in  Wang \cite{Wang} satisfying an additional assumption that the so called \emph{mass aspect function} has sign. This assumption has been recently removed by Chru\'sciel,  Galloway, Nguyen, and Paetz in \cite{ChruscielGallowayNguyenPaetz}, yielding a non-spinor proof of the positive mass theorem for asymptotically hyperbolic manifolds in dimensions less than $8$. We also note a reduction argument of Chru\'sciel and Delay \cite{ChruscielDelayHyp} that allows one to deduce the positive mass theorem for asymptotically hyperbolic manifolds from the positive mass theorem for asymptotically Euclidean initial data sets in any dimension, provided that the later holds.

Positivity of mass of hyperboloidal initial data sets has also been actively studied. The definition of mass in this setting and the first spinor proof of the positive mass theorem can be found in Chru\'{s}ciel, Jezierski, and \L\c{e}ski \cite{ChruscielJezierskiLeski}. A non-spinor proof adapting the original argument of Schoen and Yau \cite{SchoenYau2} to the hyperboloidal case was given in dimension $3$ by Sakovich \cite{SakovichJang}. This proof was extended to higher dimensions $3 < n < 8$ by Lundberg \cite{Lundberg} using  methods similar to those in Eichmair \cite{EichmairReduction}. 

The proofs of the positive mass theorem for asymptotically Euclidean initial data sets in \cite{SchoenYau2} and \cite{EichmairReduction} and for hyperboloidal initial data sets in \cite{SakovichJang} and \cite{Lundberg} are based on the so-called \emph{Jang equation reduction} argument. The Jang equation is a prescribed mean curvature equation for a graph in a product Riemannian manifold that first appeared in the work of Jang \cite{Jang}. Jang was the first to notice the relevance of this equation in the context of positive mass theorems, especially analysing the case of having zero mass, pointing out that the equation can be used to characterize initial data sets arising as slices of Minkowski spacetime. The Jang equation was rediscovered by Schoen and Yau in \cite{SchoenYau2}, where it was observed that it can be used to deform initial data sets satisfying the dominant energy condition, to initial data sets with the same mass and "almost nonnegative" scalar curvature. Further, it turns out that performing a conformal change to zero scalar curvature decreases the mass and yields a Riemannian asymptotically Euclidean manifold to which the Riemannian case of the positive mass theorem for asymptotically Euclidean manifolds can be applied, see \cite{SchoenYau1} and \cite{SchoenYau2}, yielding positivity of mass for the initial data that may be chosen to be asymptotically Euclidean as in \cite{SchoenYau2} and \cite{EichmairReduction} or asymptotically hyperbolic as in \cite{SakovichJang} and \cite{Lundberg}. When the mass is zero, the constructed solution to the Jang equation provides the desired embedding into the Minkowski spacetime. 
 
It has been discussed (see e.g. Malec and O'Murchadha \cite{MalecMurchadha2004}) whether a Jang equation reduction argument can be used to prove the so-called spacetime Penrose inequality, a refinement of the positive mass theorem for spacetimes with black holes. The version of this conjectured result for asymptotically Euclidean initial data sets argues that there is a certain lower bound for their mass that becomes an equality only for spacelike slices of Schwarzschild spacetime. As pointed out by Bray and Khuri in \cite{BrayKhuri}, Schwarzschild spacetime is a warped product and so it appears more natural to modify the Jang equation so that the reduction is carried out in a warped product setting. One of the proposals of Bray and Khuri reduces the proof of the spacetime Penrose inequality to solving a coupled system consisting of the so-called generalized Jang equation and an equation for the warping factor  designed so that initial data sets satisfying the dominant energy condition are deformed into asymptotically Euclidean manifolds with ``almost nonnegative'' scalar curvature. We note that in spherical symmetry the coupled system reduces to a single ODE yielding a proof of the spacetime Penrose inequality in the spherically symmetric case, see \cite{BrayKhuri} for details.

Similar issues arise when attempting to apply the Jang equation reduction to anti-de Sitter initial data sets. In the case of zero mass the initial data set is expected to be embedable as a slice in anti-de Sitter spacetime, which is a warped product. Thus, it appears natural that reduction arguments allowing to prove the positive mass theorem for this class of initial data should involve a generalized Jang equation as well. One such argument, proposed by Cha and Khuri \cite{ChaKhuri18} in dimension 3, involves deforming the asymptotically anti-de Sitter initial data set to an asymptotically Euclidean initial data set in a several step process using both a coupled system involving a generalized Jang equation as in \cite{BrayKhuri} and a classical Jang equation as studied in \cite{SakovichJang}. Cha and Khuri \cite{ChaKhuri18} also study the behaviour of solutions of the generalized Jang equation at infinity in dimension $3$, for a fixed warping factor when the asymptotics of initial data are similar to those in \cite{Wang}. 

This article may be seen as a complement to the work of Cha and Khuri \cite{ChaKhuri18}. Our main result is the proof of the existence of solutions to the generalized Jang equation for asymptotically anti-de Sitter initial data in dimensions $3 \leq n \leq 7$ for very general asymptotics, given a warping factor satisfying a general asymptotic condition. We make use of geometric measure theory tools as in Eichmair \cite{EichmairReduction} and a barrier method, different from the one used in Cha and Khuri \cite{ChaKhuri18}, based on the observation that the metric is "almost conformal" to the Euclidean metric. We also propose an alternative reduction argument that can be used to deduce a positive mass theorem for asymptotically anti-de Sitter initial data from the positive mass theorem for asymptotically hyperbolic manifolds provided that a coupled system involving the same generalized Jang equation as in \cite{ChaKhuri18} and \cite{BrayKhuri} but a different equation for a warping factor, has a solution. The reduction thereby proceeds in one step and the proposed coupled system appears to be simpler and arguably more natural from a geometric point of view.

The paper is structured as follows. \Secref{sec_Preliminaries} contains necessary background and definitions that we will use. In \Secref{sec_Barriers} we construct barriers for the generalized Jang equation for asymptotically anti-de Sitter initial data sets with the most general possible asymptotics and in any dimension $n \geq 3$.  In \Secref{sec_RegularBVP} we apply standard methods (similar to those used in \cite[Section $2$]{EichmairReduction} and \cite[Section $4$]{SakovichJang}) to solve a regularized boundary value problem associated to the generalized Jang equation. In \Secref{sec_GeometricSolution} we apply the results of Sections \ref{sec_Barriers} and \ref{sec_RegularBVP} and methods of geometric measure theory developed in \cite{EichmairPlateau} to construct a so-called \emph{geometric} solution to the generalized Jang equation. This is the only part of our work where the dimension restrictions $3\leq n \leq 7$ are required. In \Secref{sec_MassChange} we present our proposed coupled system and discuss its relevance in the context of positive mass theorems. We prove \autoref{ChangeOfMassFinal}, that the positive mass theorem is valid for asymptotically anti-de Sitter initial data sets, assuming the coupled system has a solution with certain properties. 

\newpage
\textbf{Acknowledgements.} 
This work would not be possible without the help of my advisor Anna Sakovich. I would like to thank her for suggesting that I work on this topic and for patiently answering all my questions during our many discussions.  I would like to thank the Erwin Schr\"odinger International Institute for Mathematics and Physics (ESI) for their hospitality during the programs "Non-regular Spacetime Geometry" and "Mathematical Relativity: Past, Present, Future" in 2023 where parts of this paper were completed.


\section{Preliminaries}
\label{sec_Preliminaries}

\label{subsec_notation}
We will repeatedly work with manifolds embedded in each other. The following conventions will be employed

\begin{itemize}
    \item $(M,g)$ is an $n$-dimensional Riemannian manifold which is  complete unless otherwise stated.
    \item $(\widetilde M,\widetilde g)$ is the $(n+1)$-dimensional manifold $M\times \mathbb{R}$ equipped with the metric $\widetilde g := g + u^2dt^2$, where $t$ is the coordinate on $\mathbb{R}$ and $u: M \to \mathbb{R}$ is a positive function. The resulting Riemannian manifold $(\widetilde M, \widetilde g)$ is denoted by $M\times_u \mathbb{R}$ and is called the warped product space with the warping factor $u$. 
    \item $(\Gamma(f),\overline g)$ is the $n$-dimensional graph $\Gamma(f) := \{(x,f(x)) : x \in U\}\subset \widetilde M$, where $U \subset M$ is an open set, $f: U \to \mathbb{R}$ is a function and $\overline g := g + u^2df^2$ is the induced metric. The covariant derivative of $\overline g$ will be denoted by $\overline \nabla$.
    \item The second fundamental form $A$ of $\Sigma^{n-1} \subset M^n$ with respect to the unit normal $\nu$ is
    \begin{align*}
        A(X,Y) & := \langle \nabla_X\nu, Y \rangle.
    \end{align*}
\end{itemize}

Let $\{\partial_i\}_{i = 1}^n$ denote a local coordinate frame on $M$ so that $\{\overline \partial_i = \partial_i + f_i\partial_t\}_{i = 1}^{n}$ is a local coordinate frame on $\Gamma(f)$ and $\{\partial_t\} \cup \{\partial_i\}_{i = 1}^n$ is a local coordinate frame on $\widetilde M$. Using these frames, we can express the components of $\widetilde g$ and $\overline g$ in terms of those of $g$ as follows
\begin{align}
    \label{eq:MetricComponents}
    \begin{split}
    \widetilde g_{ij} & = g_{ij}, \quad \widetilde g_{it} = g_{it} = 0, \quad \widetilde g_{tt} = u^2, \quad \overline g_{ij} = g_{ij} + u^2f_if_j, \\
    \widetilde g^{ij} & = g^{ij}, \quad \widetilde g^{it} = g^{it} = 0, \quad \widetilde g^{tt} = u^{-2}, \quad \overline g^{ij} = g^{ij} - \frac{u^2f^if^j}{1 + u^2|df|_g^2}.
    \end{split}
\end{align}
We also have
\begin{equation}
    \label{eq:WarpedChristoffels}
    \begin{aligned}
    \widetilde \Gamma_{ij}^k & = \Gamma_{ij}^k, & 
    \widetilde \Gamma_{it}^k & = 0, & 
    \widetilde \Gamma_{tt}^k & = -uu^k, &
    \widetilde \Gamma_{ij}^t & = 0, & 
    \widetilde \Gamma_{it}^t & = \frac{u_i}{u}, & 
    \widetilde \Gamma_{tt}^t & = 0.
    \end{aligned}
\end{equation}
We now collect some standard results for the graphs $\Gamma(f) \subset M \times_u \mathbb{R}$, see e.g. \cite{BrayKhuri}.\\

\begin{restatable}[Properties of graphs in warped product spaces]{proposition}{PropertiesGraphs}
    \label{prop:PropertiesGraphs}
    Suppose that the warped product space $(\widetilde M = M \times \mathbb R, \widetilde g = g + u^2dt^2)$ satisfies $g \in C^2_\text{loc}(M)$ and $u \in C^2_\text{loc}(M)$, so that $\widetilde g \in C^2_\text{loc}(\widetilde M)$. We then have the following.
    
    The downward pointing unit normal of $\Gamma(f)$ is given by
    \begin{equation}
        \label{eq:UnitNormalGraph}
        \nu = \frac{1}{u}\frac{u^2f^i \partial_i - \partial_t}{(1 + u^2|df|_g^2)^{1/2}}. 
    \end{equation}
    The second fundamental form of $\Gamma(f)\subset M\times_u\mathbb{R}$ is given by
    \begin{equation}
        \label{eq:SecFundFormGraph}
        A = \frac{u\Hess(f) + du\otimes df + df \otimes du + u^2\langle du,df \rangle df\otimes df}{\sqrt{1 + u^2|df|_g^2}}.
    \end{equation}
    The mean curvature of $\Gamma(f)\subset M\times_u\mathbb{R}$ is given by
    \begin{equation}
        \label{eq:MeanCurvatureGraph}
        H_{\Gamma(f)} := \tr_{\overline g}(A) = \left(g^{i j} - \frac{u^2 f^i f^j}{1 + u^2|df|_g^2}\right)\frac{u\Hess_{i j}(f) + u_i f_j + f_i u_j + u^2 \langle du,df\rangle_g f_i f_j}{\sqrt{1+u^2 |df|_g^2}}. 
    \end{equation}
    The Christoffel symbols of $M$ and of the graph $\Gamma(f)$ are related by
    \begin{equation*}
        \overline \Gamma_{ij}^k - \Gamma_{ij}^k = - uu^kf_i f_j + A_{ij}\frac{uf^k}{\sqrt{1 + u^2|df|_g^2}}.
    \end{equation*}
    For $v,f: M \to \mathbb R$, the Laplacian $\Delta^{\overline g} v$ is given by
    \begin{align}
        \label{eq:LaplacianDifferenceGraph}
        \Delta^{\overline g} v 
        & = \left(g^{ij} - \frac{u^2f^if^j}{1 + u^2|df|_g^2}\right)\Hess_{ij}v + \frac{u|df|_g^2\langle \nabla u, \nabla v\rangle_g}{1 + u^2|df|_g^2} -\frac{u\langle\nabla f,\nabla v\rangle_g H_{\Gamma(f)}}{\sqrt{1 + u^2|df|_g^2}}.
    \end{align}
\end{restatable}
We will now introduce the class of asymptotically hyperbolic initial data sets that we will work with. Our model for $(M,g)$ is the hyperbolic space $\mathbb H^n$ for which we will use the Poincaré ball model. In other words, $\mathbb{H}^n$ will be represented as the unit ball $ B_1 := \{x \in \mathbb{R}^n : |x| < 1\}$ equipped with the metric
\begin{equation*}
    b := \frac{4}{(1-|x|^2)^2}\delta,
\end{equation*}
where $\delta$ is the standard Euclidean metric on $B_1$. Using $\rho := \frac{1-|x|^2}{2}$ we can rewrite $b$ as
\begin{align}
    \label{eq:AltDefMetric}
    \begin{split}
        b & = \rho^{-2}\delta = \rho^{-2}\left(\frac{d\rho^2}{1-2\rho} + (1-2\rho)\sigma\right),
    \end{split}
\end{align}
where $\sigma$ is the round metric on $S^{n-1}$.\\
\begin{definition}
    \label{def:AsympHypManifolds}
    Suppose that $k\geq 0$ is an integer, that $\alpha \in [0,1)$ and $\tau > 0$ are real numbers and that $(M,g)$ is a Riemannian manifold with $g \in C^{k,\alpha}_\text{loc}(M)$, possibly with boundary. We then say that $(M,g)$ is a $C^{k,\alpha}_\tau$ \emph{asymptotically hyperbolic manifold}, if there is a compact set $K_0 \subset M$, a radius $R_0 > 0$ and a diffeomorphism
    \begin{equation*}
        \Psi: M\setminus K_0 \to \mathbb{H}^n \setminus \overline {B_{R_0}},
    \end{equation*}
    called a \emph{chart at infinity} such that $\Psi_*g - b \in C^{k,\alpha}_{\text{loc}}(\mathbb{H}^n; S^2\mathbb{H}^n)$ and
    \begin{equation*}
        ||\Psi_* g - b||_{C^{k,\alpha}_\tau(\mathbb{H}^n\setminus B_{R_0};S^2\mathbb{H}^n)} := \sup_{x \in \mathbb{H}^n \setminus \overline B_{R_0 + 1}}\bigg(\rho(x)^{-\tau} ||\Psi_*g -b||_{C^{k,\alpha}(B_1(x);S^2\mathbb{H}^n)}\bigg) < \infty.
    \end{equation*}
\end{definition}
Next we define the notion of a $C^{k,\alpha}$ initial data set.\\ 
\begin{definition}
    \label{def:InitialDataSets}
    Suppose that $k \geq 1$ is an integer, that $\alpha \in [0,1)$ is a real number, that $(M,g)$ is a Riemannian manifold with $g \in C^{k,\alpha}_\text{loc}(M)$, possibly with boundary and that $K \in C_\text{loc}^{k-1,\alpha}(M)$ symmetric $2$-tensor on $M$. We then say that the triple $(M,g,K)$ is a \emph{$C^{k,\alpha}$ initial data set}. 
\end{definition}
\hfill
\begin{definition}
    \label{def:aads}
    Suppose that $k \geq 1$ is an integer, that $\alpha \in [0,1)$ and $\tau > 0$ are real numbers and that $(M,g,K)$ is a $C^{k,\alpha}$ initial data set. We then say that $(M,g,K)$ is a \emph{$C^{k,\alpha}_\tau$ asymptotically anti-de Sitter} initial data set if $(M,g)$ is a $C^{k,\alpha}_\tau$ asymptotically hyperbolic manifold and the symmetric $2$-tensor $K$ satisfies 
    \begin{equation*}
        \label{eq:DecayDef_2}
        ||\Psi_* K||_{C^{k-1,\alpha}_\tau(\mathbb{H}^n\setminus \overline B_{R_0};S^2\mathbb{H}^n)} < \infty,
    \end{equation*}
    where $\Psi: M \setminus K_0 \to \mathbb H^n \setminus \overline{B_{R_0}}$ is the chart at infinity as in \autoref{def:InitialDataSets}.
\end{definition}
 \hfill
\begin{definition}
    \label{def:DecayFunctions}
    Suppose that $k,k'\geq 0$ are integers, that $\alpha,\alpha' \in [0,1)$ and $\tau,\tau' > 0$ are real numbers, 
    that $(M,g)$ is a $C^{k,\alpha}_\tau$ asymptotically hyperbolic manifold and that $\Psi: M\setminus K_0 \to \mathbb H^n \setminus B_{R_0}$ is the associated chart at infinity. We then define \emph{the weighted H\"older space $C^{k',\alpha'}_{\tau'}(M\setminus K_0)$} as the collection of those $f \in C^{k',\alpha'}_\text{loc}(M\setminus K_0)$ such that
    \begin{equation*}
        ||f||_{C^{k',\alpha'}_{\tau'}(M\setminus K_0)} := \sup_{x \in \mathbb H^n\setminus \overline {B_{R_0 + 1}}}\left(\rho(x)^{-\tau}||f\circ \Psi^{-1}||_{C^{k,\alpha}(B_1(\Phi(x)))}\right) < \infty.
    \end{equation*}
    Since $K_0$ is compact it can be covered by finitely many geodesic balls $B_{r_1}(p_1), \dots, B_{r_k}(p_k)$ of radius less than $r_0$, where $r_0 > 0$ is the injectivity radius of $(M,g)$. Letting $\phi_i: B_{r_i}(p_i) \to \{x \in \mathbb R^n : |x| < r_i\}$ denote geodesic normal coordinates in each $B_{r_i}(p_i)$, we define \emph{the weighted H\"older space $C^{k',\alpha'}_{\tau'}(M)$} as the collection of those $f \in C^{k',\alpha'}_\text{loc}(M)$ such that
    \begin{equation*}
        ||f||_{C^{k',\alpha'}_{\tau'}(M)} := ||f||_{C^{k',\alpha'}_{\tau'}(M\setminus K_0)} + \sum_{i = 1}^k ||f\circ \phi_i^{-1}||_{C^{k',\alpha'}(B_{r_i}(0))} < \infty.
    \end{equation*}
\end{definition}
Different finite open covers of the compact set $K_0$ give different but equivalent norms on the space $C^{k',\alpha'}_\text{loc}(M)$. 

We remind the reader that asymptotically anti-de Sitter initial data sets model slices of spacetimes satisfying the Einstein equations with negative cosmological constant. Consequently, the \emph{dominant energy condition} takes the following form.\\
\begin{definition}
    \label{def:DEC}
    Suppose that $k \geq 2$, we then say that a $C^{k,\alpha}$ initial data set $(M,g,K)$ satisfies \emph{the dominant energy condition} if $\mu \geq |J|_g$, where $\mu$ and $J$ are given by
    \begin{align}
        \label{eq:LocalEnergy}
        2\mu & = \Scal^g + n(n-1) + \tr_g(K)^2 - |K|_g^2, \\
        \label{eq:LocalCurrent}
        J & = \div_g K - d\tr_gK.
    \end{align}
\end{definition}
We will now introduce the generalized Jang equation. Let $(M,g,K)$ be a $C^{k,\alpha}_\tau$ asymptotically anti-de Sitter initial data set. We extend the tensor $K$ to $M \times_u \mathbb{R}$ as follows
\begin{equation}
    \label{eq:ExtensionK}
    \overline K(\partial_i, \partial_t) := \overline K(\partial_t,\partial_i) := 0 , \quad \overline K(\partial_t, \partial_t) := \frac{u^2 \langle df,du\rangle}{\sqrt{1 + u^2|df|_g^2}}.
\end{equation}
In this case
\begin{align}
    \label{eq:TraceK}
    \begin{split}
    \tr_{\Gamma(f)} \overline K
    & = \tr_{\Gamma(f)} K + \frac{u^2|df|_g^2\langle df,du\rangle}{(1 + u^2|df|_g^2)^{3/2}},
    \end{split}
\end{align}
where the second term in the left hand side of \eqref{eq:TraceK} can be expressed using \eqref{eq:UnitNormalGraph} as follows
\begin{equation*}
    \frac{u^2|df|_g^2\langle df,du\rangle}{(1 + u^2|df|_g^2)^{3/2}} = \frac{\langle \nu, \nabla u\rangle}{u}\left(1 - \langle -\nu, u^{-1}\partial_t \rangle^2\right).
\end{equation*}
Hence the trace of the extended tensor $\overline K$ is
\begin{equation}
    \label{eq:PresecribedMeanCurvature}
    \tr_{\Gamma(f)} \overline K = \tr_{\Gamma(f)}K + \frac{\langle \nu, \nabla u\rangle}{u}\left(1 - \langle -\nu, u^{-1}\partial_t \rangle^2\right). 
\end{equation}
The identity \eqref{eq:MeanCurvatureGraph} now implies
\begin{equation}
    \label{eq:JangOperator}
    \mathcal{J}(f) := H_{\Gamma(f)} - \tr_{\Gamma(f)}(\overline K) = \left(g^{i j} - \frac{u^2f^i f^j}{1 + u^2 |df|_g^2}\right)\left(\frac{u\Hess_{i j}f + u_i f_j + f_i u_j}{\sqrt{1 + u^2|df|_g^2}} - K_{i j}\right). 
\end{equation}
The equation
\begin{align}
    \label{eq:JangEquation}
    \left(g^{i j} - \frac{u^2f^i f^j}{1 + u^2 |df|_g^2}\right)\left(\frac{u\Hess_{i j}f + u_i f_j + f_i u_j}{\sqrt{1 + u^2|df|_g^2}} - K_{i j}\right) = 0, 
\end{align}
will be referred to as the \emph{generalized Jang equation} and abbreviated as $\mathcal J(f) = 0$. This equation was originally introduced by Bray and Khuri in \cite{BrayKhuri}, where one can find an extensive discussion and motivation. In particular, the extension \eqref{eq:ExtensionK} is chosen so that solutions to the generalized Jang equation blow up respectively down at marginally outer trapped respectively inner trapped surfaces of the initial data set $(M,g,K)$, see Bray and Khuri \cite[Section $2$]{BrayKhuri} or the proof of \autoref{thm:JangGeometric} below. 

In this article, we will use the following straightforward modification of a result by Bray and Khuri \cite[Theorem 1]{BrayKhuri}.\\

\begin{proposition}[Generalized Schoen-Yau identity]
    \label{prop:GeneralizedSchoenYauIdentity}
    Suppose that $k \geq 2$, that $U \subset M$ is open, that $(M,g,K)$ is a $C^{k,\alpha}$ initial data set and that $f: U \to \mathbb R$ satisfies the generalized Jang equation \eqref{eq:JangEquation}. Then the scalar curvature $\Scal^{\overline g}$ of the graph $ \Sigma = \Gamma(f)$ is given by
    \begin{equation}
        \label{eq:SchoenYauIdentity}
        \Scal^{\overline g} = -n(n-1) + 2(\mu - J(w)) + |A - \overline K|_{\overline g}^2 + 2|q|_{\overline g}^2 - \frac{2}{u}\div_{\overline g}(uq),
        \end{equation}
        where
        \begin{equation}
            \label{eq:DefSchoenYau}
            q_i := \frac{uf^j(A^\Sigma_{ij} - (\overline K|_\Sigma)_{ij})}{\sqrt{1 + u^2|df|_g^2}}, \quad w:= \frac{uf^i\partial_i}{\sqrt{1 + u^2|df|_g^2}}, \quad f^i = f_jg^{ij}.
        \end{equation}
\end{proposition}
The difference between \eqref{eq:SchoenYauIdentity} and \cite[Equation $2.3$]{BrayKhuri} is that the definition of $\mu$ in \eqref{eq:LocalEnergy} is different, to account for the cosmological constant. 

We will often assume that the warping factor $u$ is strictly positive and that $u - \rho^{-1} \in C^{2,0}_0(M\setminus K_0)$. In this case we have the following estimate.\\
\begin{lemma}
    \label{lem:WarpingBounds}
    Suppose that $k \geq 1$, that $(M,g,K)$ is a $C^{k,\alpha}_{\tau}$ asymptotically anti-de Sitter initial data set and that $(u - \rho^{-1}) \in C^{2,0}_0(M)$. Then there is a constant $C > 0$ depending only on $(M,g,K)$, $|u - \rho^{-1}|_{C^{2,0}_0(M)}$ and $\min_M u > 0$ such that
    \begin{align*}
        |\nabla u|_g + |\nabla^2 u|_g \leq Cu.
    \end{align*}
\end{lemma}
\begin{proof}
    Note that
    \begin{align*}
        \left|\nabla^b \rho^{-1}\right|_b = \frac{1+O(\rho)}{\rho}, \quad |(\nabla^b)^2 \rho^{-1}|_b = \frac{2+O(\rho)}{\rho}.
    \end{align*}
    The claim now follows from the assumptions on $(M,g,K)$, $u$ and $u - \rho^{-1}$.
\end{proof}


\section{Barriers at infinity for the Jang equation}
\label{sec_Barriers}
In this section, $(M,g,K)$ is a $C^{1,\alpha}_\tau$ asymptotically anti-de Sitter initial data set for some $\alpha \in (0,1)$ and $\tau > n/2$. The barriers are defined as follows (cf. Schoen and Yau \cite{SchoenYau2}).\\
\begin{definition}
    \label{def:Barrier}
    A function $f_+: \{\rho < \rho_0\} \to \mathbb R$ is an upper barrier for the Jang equation $\mathcal{J}(f) = 0$, if $f_+ \in C^2(\{\rho < \rho_0\}) \cap C^0(\{\rho \leq \rho_0\})$ and
    \begin{align}
        \label{eq:BarrierConditionUpper}
        \lim_{\rho \to \rho_0-}\partial_\rho f_+ = \infty, \quad \mathcal{J}(f_+) < 0.
    \end{align}
    Similarly, a function $f_-:\{\rho < \rho_0\} \to \mathbb R$ is a lower barrier for the Jang equation if $f_- \in C^2(\{\rho < \rho_0\}) \cap C^0(\{\rho \leq \rho_0\})$ and
    \begin{align}
        \label{eq:BarrierConditionLower}
        \lim_{\rho \to \rho_0-}\partial_\rho f_- = -\infty, \quad \mathcal{J}(f_-) > 0.
    \end{align}
\end{definition}
The following comparison principle is well known (cf. Schoen and Yau \cite{SchoenYau2} and Sakovich \cite[Proposition $4.4$]{SakovichJang}).\\
\begin{lemma}
    \label{lem:ComparisonJang}
    Let $0 < \rho_1 < \rho_0$ be constants, suppose that $f_+$ is an upper barrier and that $f_-$ is a lower barrier for the Jang equation on $\{\rho \leq \rho_0\}$ such that
    \begin{equation*}
        f_-|_{\{\rho = \rho_1\}} \leq \phi|_{\{\rho = \rho_1\}} \leq f_+|_{\{\rho = \rho_1\}}
    \end{equation*}
    and let $\epsilon > 0$ be such that $\mathcal{J}(f_+) \leq \epsilon f_+$ and $\mathcal{J}(f_-) \geq \epsilon f_-$ on $\{\rho_1 \leq \rho \leq \rho_0\}$. Further suppose that $f$ is a solution to the Dirichlet problem
    \begin{equation*}
        \begin{cases}
        \mathcal{J}(f) = \epsilon f, & \text{on } \{\rho \geq \rho_1\} \\
        f = \phi , & \text{on } \{\rho = \rho_1\} 
        \end{cases}.
    \end{equation*}
    Then
    \begin{equation*}
        f_- \leq f \leq f_+ \text{ on } \{\rho_1 \leq \rho \leq \rho_0\}.
    \end{equation*}
\end{lemma} 
\begin{proof}
    Since $\{\rho_1 \leq \rho \leq \rho_0\}$ is compact it follows that $f_+ - f$ attains a minimum at some point $p \in \{\rho_1 \leq \rho \leq \rho_0\}$. If $p \in \{\rho = \rho_1\}$ then since $f_+ \geq \phi$ on $\{\rho = \rho_1\}$ we find
    \begin{align*}
        \inf_{\rho_1 \leq \rho \leq \rho_0} (f_+ - f) = f_+(p) - f(p) = f_+(p) - \phi(p) \geq 0.
    \end{align*}
    If $p \in \{\rho_1 < \rho < \rho_0\}$ then $|d(f - f_+)|_g(p) = 0$ and $\Hess (f_+ - f)$ is positive definite at $p$. Combining this with \eqref{eq:JangOperator} we find that at $p$
    \begin{align*}
        \epsilon(f_+ - f) \geq \mathcal J(f_+) - \mathcal J(f) = u\Delta^g(f_+ - f) \geq 0.
    \end{align*}
    Finally if $p \in \{\rho = \rho_0\}$, then since $\lim_{\rho \to \rho_0^-} \partial_\rho f_+ = \infty$, it follows that $\lim_{\rho \to \rho_0^-}  \partial_\rho(f_+ - f) = \infty$ and hence there is some point $p'$ with $\rho_1 < \rho(p') < \rho_0$ such that $(f_+ - f)(p') < (f_+ - f)(p)$, which contradicts $p$ being a global minimum. Thus $f_+ - f \geq 0$ on all of $\{\rho_1 \leq \rho \leq \rho_0\}$. A similar argument shows that $f - f_- \geq 0$ on $\{\rho_1 \leq \rho \leq \rho_0\}$.
\end{proof}

\begin{proposition}[Existence of barriers for the Jang equation]
    \label{prop:ExistenceBarriers}
    Suppose that $\alpha \in (0,1)$, $\tau \in (n/2,n)$ and that $(M,g,K)$ is a $C^{1,\alpha}_\tau$ asymptotically anti-de Sitter initial data set and let $U := \{\rho < \rho_0\}$. If the warping factor $u$ satisfies
    \begin{equation}
        \label{eq:DefWarp}
        u - \rho^{-1} \in C^{1,0}_0(U),
    \end{equation}
    then there exists a positive function $f_+ \in C^{2,\alpha}_{\tau + 1}(U)$ such that $f_+$ is an upper barrier in the sense of \autoref{def:Barrier} and $f_- := -f_+$ is a lower barrier in the sense of \autoref{def:Barrier}. 
\end{proposition}

\begin{proof}
    Throughout the proof, we let $C$ denote a constant that may differ from line to line but that may depend only on the initial data $(M,g,K)$ and $u$. The bound $|K|_g \leq C\rho^\tau$ gives
    \begin{equation*}
        \mathcal{J}(f) = \left(g^{ij} - \frac{u^2f^if^j}{1 + u^2|df|_g^2}\right)\frac{u\Hess_{ij}f + f_iu_j + u_if_j}{(1 + u^2|df|_g^2)^{1/2}} + O(\rho^\tau).
    \end{equation*}
    Note that when $f = f(\rho)$, we have $f^\alpha = g^{\alpha\rho}f_\rho$, $f^\rho = g^{\rho \rho} f_\rho$, $|df|_g^2 = f^\rho f_\rho,$ and hence
    \begin{align*}
        \begin{split}
        g^{\rho \rho} - \frac{u^2 f^\rho f^\rho}{1 + u^2|df|_g^2} 
        & = \frac{g^{\rho \rho}}{1 + u^2 |df|_g^2}, \\
        g^{\rho \alpha} - \frac{u^2 f^\rho f^\alpha}{1 + u^2|df|_g^2} 
        & = \frac{g^{\rho \alpha}}{1 + u^2 |df|_g^2}, \\
        g^{\alpha \beta} - \frac{u^2 f^\alpha f^\beta}{1 + u^2|df|_g^2} 
        & = g^{\alpha \beta} - g^{\rho \alpha}g^{\rho \beta}(g^{\rho \rho})^{-1}\frac{u^2 |df|_g^2}{1 + u^2|df|_g^2}.
        \end{split}
    \end{align*}
    Similarly, for $f = f(\rho)$ we find that
    \begin{align*}
        \begin{split}
        u\Hess_{\rho \rho}f + f_\rho u_\rho + u_\rho f_\rho & = u\left(f'' - \Gamma^{\rho}_{\rho \rho} f' + \frac{2u_\rho}{u}f'\right), \\
        u\Hess_{\rho \alpha}f + f_\rho u_\alpha + u_\rho f_\alpha & = u\left(-\Gamma^{\rho}_{\rho \alpha} + \frac{u_\alpha}{u}\right)f', \\
        u\Hess_{\alpha \beta}f + f_\alpha u_\beta + u_\alpha f_\beta & = -u\Gamma^{\rho}_{\alpha \beta}f'.
        \end{split}
    \end{align*}
    In summary, in the case when $f = f(\rho)$ we obtain
    \begin{align}
        \label{eq:J2'}
        \begin{split}
            \mathcal{J}(f) 
            & = \frac{ug^{\rho \rho}}{(1 + u^2 |df|_g^2)^{3/2}}\left(f'' - \Gamma^{\rho}_{\rho \rho} f' + \frac{2u_\rho}{u}f'\right)
            + \frac{2ug^{\rho \alpha}}{(1 + u^2 |df|_g^2)^{3/2}}\left(-\Gamma^{\rho}_{\rho \alpha} + \frac{u_\alpha}{u}\right)f'\\
            & \quad - \frac{u\Gamma^{\rho}_{\alpha \beta}f'}{(1 + u^2|df|_g^2)^{1/2}}\left(g^{\alpha \beta} - g^{\rho \alpha}g^{\rho \beta}(g^{\rho \rho})^{-1}\frac{u^2 |df|_g^2}{1 + u^2|df|_g^2}\right) + O(\rho^\tau).
        \end{split}
    \end{align}
    As observed by Malec and Murchadha \cite{MalecMurchadha2004}, in the presence of spherical symmetry there is a substitution that transforms the Jang equation into an ordinary differential equation. We employ this fact here and set
    \begin{equation}
        \label{eq:Ansatzf'}
        f_+ := f_+(\rho) = \int_0^\rho \frac{\xi_+(s)}{\sqrt{1 - \xi_+(s)^2}}ds,
    \end{equation}
    for some $\xi_+: [0,\rho_0] \to [0,1]$ to be determined. We require that $\xi_+(\rho_0) = 1$, so that
    \begin{equation}
        \label{eq:fdrho'}
        \partial_\rho f_+(\rho) = f_+'(\rho) = \frac{\xi_+(\rho)}{\sqrt{1 - \xi_+(\rho)^2}},
    \end{equation}
    implies that the first condition of \eqref{eq:BarrierConditionUpper} is fulfilled by $f_+$. From \eqref{eq:Ansatzf'} it follows that
    \begin{equation}
        \label{eq:fd2rho'}
        f_+'' = \frac{\xi_+'}{(1-\xi_+^2)^{3/2}}.
    \end{equation}
    In the view of the identities \eqref{eq:fdrho'}, and \eqref{eq:fd2rho'} we obtain from \eqref{eq:J2'}
    \begin{align}
        \label{eq:J3'}
        \begin{split}
            & \frac{\mathcal{J}(f_+)(1 + u^2|df_+|_g^2)^{3/2}(1-\xi_+^2)^{3/2}}{ug^{\rho\rho}} \\
            & < \xi_+' - \Gamma^{\rho}_{\rho \rho} \xi_+(1 - \xi_+^2) + \frac{2u_\rho}{u}\xi_+(1 - \xi_+^2) + \frac{2g^{\alpha\rho}}{g^{\rho \rho}}\left(-\Gamma^{\rho}_{\rho \alpha} + \frac{u_\alpha}{u}\right)\xi_+(1 - \xi_+^2)\\
            & \quad - \frac{\Gamma^{\rho}_{\alpha \beta}(1 + u^2|df_+|_g^2)(1 - \xi_+^2)\xi_+}{g^{\rho \rho}}\left(g^{\alpha \beta} - g^{\rho \alpha}g^{\rho \beta}(g^{\rho \rho})^{-1}\frac{u^2 |df_+|_g^2}{1 + u^2|df_+|_g^2}\right) \\
            & \quad + \frac{C\rho^\tau (1 + u^2|df_+|_g^2)^{3/2}(1-\xi_+^2)^{3/2}}{ug^{\rho \rho}}.
        \end{split}
    \end{align}
    Furthermore, the fact that $(M,g,K)$ is $C^{1,\alpha}_\tau$ asymptotically anti-de Sitter, together with equations \eqref{eq:AltDefMetric} and \eqref{eq:DefWarp} implies the estimates
    \begin{equation}
        \label{eq:TonsOfEstimates}
        \begin{aligned}
            g^{\rho \rho} & = \rho^2(1 - 2\rho) + O(\rho^{\tau + 2}),
            & g^{\rho \alpha} & = O(\rho^{\tau + 2}),
            & g^{\alpha \beta} & = \frac{\rho^2\sigma^{\alpha \beta} }{1-2\rho}+ O(\rho^{\tau + 2}), \\
            \Gamma_{\rho\rho}^\rho & = \frac{-1}{\rho}\frac{1 - 3\rho}{1 - 2\rho} + O(\rho^{\tau - 1}),
            & \Gamma_{\rho\alpha}^\rho & = O(\rho^{\tau - 1}),
            & \Gamma_{\alpha \beta}^\rho & = \sigma_{\alpha \beta}(\rho^{-1} + O(1)), \\
            u &= \frac{1}{\rho} + O(1), 
            & \frac{u_\rho}{u} & = \frac{-1}{\rho} + O(1),
            & \frac{u_\alpha}{u} & = O(\rho).
        \end{aligned}
    \end{equation}
    As a consequence of \eqref{eq:TonsOfEstimates} and \eqref{eq:fdrho'} we have
    \begin{align}
        \label{eq:SphericalMetric}
        g^{\alpha \beta} - g^{\rho \alpha}g^{\rho \beta}(g^{\rho \rho})^{-1}\frac{u^2 |df_+|_g^2}{1 + u^2|df_+|_g^2} 
        & = \frac{\rho^2\sigma^{\alpha \beta}}{1 - 2\rho} + O(\rho^{\tau + 2})
    \end{align}
    and
    \begin{align}
        \label{eq:KeyCancellation}
        \begin{split}
            (1 + u^2|df_+|_g^2)(1 - \xi_+^2) 
            & = \left(1 + \frac{u^2g^{\rho \rho}\xi_+^2}{1 - \xi_+^2}\right)(1-\xi_+^2) \\
            & = 1 - \xi_+^2 + \left(\rho^{-1} + O(1)\right)^2\left(\rho^2(1-2\rho) + O(\rho^{\tau + 2})\right)\xi_+^2 \\
            & = 1 - \xi_+^2 + (1 + O(\rho))\xi_+^2 \\
            & = 1 + O(\rho).
        \end{split}
    \end{align}
    Combining \eqref{eq:J3'} with the estimates \eqref{eq:TonsOfEstimates}, \eqref{eq:SphericalMetric} and \eqref{eq:KeyCancellation} it follows that
    \begin{align*}
        & \frac{\mathcal{J}(f_+)(1 + u^2|df_+|_g^2)^{3/2}(1-\xi_+^2)^{3/2}}{ug^{\rho\rho}} \\
        & < \xi_+' + \frac{1 + O(\rho)}{\rho}\xi_+(1 - \xi_+^2) - \frac{2 + O(\rho)}{\rho}\xi_+(1 - \xi_+^2) + 2\left(O(\rho^{\tau - 1}) + O(\rho)\right)\xi_+(1 - \xi_+^2)\\
        & \quad - \frac{\sigma_{\alpha \beta}(\rho^{-1} + O(1))(1 + O(\rho))\xi_+}{\rho^2(1 + O(\rho))}\bigg(\frac{\rho^2\sigma^{\alpha \beta}}{1 + O(\rho)} + O(\rho^{\tau + 2})\bigg) + O\left(\rho^{\tau-1}\right) \\
        & = \xi_+' - \frac{1 + O(\rho)}{\rho}\xi_+(1 - \xi_+^2)
        - \frac{(1 + O(\rho))\xi_+(n-1 + O(\rho))}{\rho} + O(\rho^{\tau - 1}),
    \end{align*}
    which in turn yields the inequality
    \begin{equation*}
        \frac{\mathcal{J}(f_+)(1 + u^2|df_+|_g^2)^{3/2}(1-\xi_+^2)^{3/2}}{ug^{\rho\rho}} < \xi_+' - \frac{\xi_+}{\rho}(n - \xi_+^2 - C_0\rho) + C_0\rho^{\tau - 1},
    \end{equation*}
    for some fixed constant $C_0 > 0$ that only depends on $(M,g,K)$ and $u$. 

    We now seek a function $\xi_+: [0,\rho_0] \to [0,1]$ such that $\xi_+(\rho_0) = 1$ and
    \begin{equation}
        \label{eq:ConditionsXi}
        \xi_+' - \frac{\xi_+}{\rho}(n - \xi_+^2 - C_0\rho) + C_0\rho^{\tau - 1} < 0,
    \end{equation}
    so that $f_+$ defined by \eqref{eq:Ansatzf'} is an upper barrier in the sense of \autoref{def:Barrier}. For some $\lambda \in (0,1)$ to be specified later, we will look for $\xi_+$ in the form
    \begin{equation}
        \label{eq:Ansatzk}
        \xi_+(\rho) = \frac{1}{(1-\lambda)(\rho_0/\rho)^{n/2} + \lambda (\rho_0/\rho)^{\tau}}.
    \end{equation} 

    Since $\rho_0/\rho \geq 1$ and $n/2 < \tau$ we have $(\rho_0/\rho)^{n/2} \leq (\rho_0/\rho)^\tau$ and hence for all $\lambda \in [0,1]$ we have $(\rho_0/\rho)^{n/2} \leq (1-\lambda)(\rho_0/\rho)^{n/2} + \lambda (\rho_0/\rho)^\tau \leq (\rho_0/\rho)^\tau$. Consequently
    \begin{equation}
        \label{eq:UpperLowerk}
        (\rho/\rho_0)^\tau \leq \xi_+ \leq (\rho/\rho_0)^{n/2}.
    \end{equation}
    Defining 
    \begin{align}
        \label{eq:Gamma}
        \gamma & := \frac{n}{2} + \frac{\lambda(\tau - n/2)}{\lambda + (1-\lambda)(\rho/\rho_0)^{\tau - n/2}},
    \end{align}
    we obtain $\xi_+' = \gamma\xi_+\rho^{-1}$. We  conclude that for $\xi_+$ as in \eqref{eq:Ansatzk}
    \begin{equation}
        \label{eq:dummy}
        \xi_+' - \frac{\xi_+}{\rho}(n - \xi_+^2 - C_0\rho) + C_0\rho^{\tau - 1} < -(n-\gamma - (\rho/\rho_0)^n - C_0\rho)\frac{\xi_+}{\rho} + C_0\rho^{\tau - 1}.
    \end{equation}
    Next let $\rho_0 > 0$ be such that $C_0\rho_0 + C_0\rho_0^{\tau} < (n-\tau)/4$, then we have
    \begin{align}
        \label{eq:BarrierKey'}
        n-\gamma - (\rho/\rho_0)^n - C_0\rho_0 - C_0\rho_0^{\tau} > \frac{3n + \tau}{4} - \gamma - (\rho/\rho_0)^n.
    \end{align}
    Let $F:[0,1]\times (0,1) \to \mathbb{R}$ be given by
    \begin{equation}
        \label{eq:DefF}
        F(x,\lambda) = \frac{\lambda (\tau - n/2)}{\lambda + (1-\lambda)x^{\tau-n/2}}, 
    \end{equation}
    so that $0 \leq F(x,\lambda) \leq \tau - n/2$. Let $x_0 := (3(n-\tau)/8)^{1/n} > 0$ and note that $\lim_{\lambda \to 0^+} F(x_0,\lambda) = 0$. Thus we may choose $\lambda \in (0,1)$ such that $F(x_0,\lambda) < 1/8$. Now for any $x \in [0,x_0]$
    \begin{align*}
        \frac{n + \tau}{4} - F(x,\lambda) - x^n
        & \geq \frac{n + \tau}{4} - (\tau - n/2) - x_0^n = \frac{3(n - \tau)}{4} - \frac{3(n-\tau)}{8} > 0.
    \end{align*}
    Since $\tau > n/2$, the function $F(x,\lambda)$ is decreasing in $x$, hence for $x \in [x_0,1]$ we have
    \begin{align*}
        F(x,\lambda) + x^n \leq F(x_0,\lambda) + 1 < \frac{9}{8} = \frac{3 + 3/2}{4} \leq \frac{n + \tau}{4}.
    \end{align*}
    In summary, for this choice of $\lambda \in (0,1)$, we have
    \begin{equation}
        \label{eq:DesiredF}
        F(x,\lambda) + x^n < \frac{n + \tau}{4},
    \end{equation}
    for all $x \in [0,1]$. Moreover in view of \eqref{eq:Gamma} and \eqref{eq:DefF} we find $\gamma = n/2 + F(\rho/\rho_0,\lambda)$, which combined with \eqref{eq:BarrierKey'} and \eqref{eq:DesiredF} implies
    \begin{equation}
        \label{eq:BarrierCrucial}
        \begin{aligned}
            n - \gamma - (\rho/\rho_0)^n - C_0\rho_0 - C_0\rho_0^{\tau} 
            & > \frac{3n + \tau}{4} - \frac{n}{2} - \bigg(F(\rho/\rho_0,\lambda) + (\rho/\rho_0)^n\bigg) \\
            & > \frac{3n + \tau}{4} - \frac{n}{2} - \frac{n + \tau}{4} = 0.
        \end{aligned}
    \end{equation}
    Equation \eqref{eq:BarrierCrucial} implies $n - \gamma - (\rho/\rho_0)^n - C_0\rho > C_0\rho_0^\tau$ which combined with \eqref{eq:UpperLowerk} and \eqref{eq:dummy} yields
    \begin{align*}
        \xi_+' - \frac{\xi_+}{\rho}(n - \xi_+^2 - C_0\rho) + C_0\rho^{\tau - 1}
        & < -(n-\gamma - (\rho/\rho_0)^n - C_0\rho)\frac{\rho^{\tau - 1}}{\rho_0^\tau} + C_0\rho^{\tau - 1} \\
        & < C_0\rho^{\tau - 1} - C_0\rho^{\tau - 1}
        = 0,
    \end{align*}
    which we wanted to show. In conclusion, the function $f_+(\rho) = \int_0^\rho \frac{\xi_+(s)}{\sqrt{1 - \xi_+(s)^2}}ds$ satisfies
    \begin{align*}
        \lim_{\rho \to \rho_0^-} \partial_\rho f_+ = \infty, \quad \mathcal{J}(f_+) < 0, \quad f_+ > 0,
    \end{align*}
    and due to the bound
    \begin{align*}
        \int_{\rho_0/2}^{\rho_0}|f_+'(s)|ds < \int_{\rho_0/2}^{\rho_0} \frac{1}{\sqrt{1 - \xi_+(s)^2}} ds < C\int_0^C x^{-1/2}dx = 2C^{3/2} < \infty,
    \end{align*} 
    it follows that $f_+$ extends continuously to the boundary $\{\rho = \rho_0\}$ of $\{\rho \leq \rho_0\}$. Hence $f_+$ is indeed an upper barrier of the generalized Jang equation in the sense of \autoref{def:Barrier}. Finally a computation shows that $\xi_+ \in C^{2,\alpha}_\tau$ and hence $f_+ \in C^{2,\alpha}_{\tau + 1}$ .

    If we now let $\xi_- := -\xi_+$, then $f_- = -f_+$. By construction, $f_- < 0$, it satisfies the first condition of \eqref{eq:BarrierConditionLower} and we have $f_- \in C^{2,\alpha}_{\tau + 1}$. The second condition of \eqref{eq:BarrierConditionLower} is also satisfied since
    \begin{align*}
        \frac{\mathcal{J}(f_-)(1 + u^2|df_-|_g^2)^{3/2}(1-\xi_-^2)^{3/2}}{ug^{\rho\rho}}
        & > \xi_-' - \frac{\xi_-}{\rho}(n - \xi_-^2 - \epsilon) - C_0\rho^{\tau - 1} \\
        & = -\left(\xi_+' - \frac{\xi_+}{\rho}(n - \xi_+^2 - \epsilon) + C_0\rho^{\tau - 1}\right)
        > 0,
    \end{align*}
    where in the last line we have used that $\xi_+$ satisfies \eqref{eq:ConditionsXi}.
\end{proof}


\section{A regularized boundary value problem}
\label{sec_RegularBVP}
In this section we will prove the following proposition.\\
\begin{proposition}
    \label{RegularizedBVP}
    Suppose that $\alpha \in (0,1)$ and $\tau > 0$ are real numbers, that $(M,g,K)$ is a $C^{2,\alpha}_\tau$ asymptotically anti-de Sitter initial data set and that $u\in C^{2,\alpha}_\text{loc}(M)$ satisfies
    \begin{align*}
        u - \rho^{-1} \in C^{2,0}_0(M).
    \end{align*}
    Suppose furthermore that $\Omega \subset M$ is a bounded domain such that $\partial \Omega$ is $C^{3,\alpha}$ and that the mean curvature $H_{\partial \Omega}$ of $\partial \Omega$ computed as the tangential divergence of the outward pointing unit normal satisfies
    \begin{equation*}
        H_{\partial \Omega} - |\tr_{\partial \Omega} K| > 0.
    \end{equation*}
    Lastly suppose that $\epsilon > 0$ and $\phi \in C^{3,\alpha}(\overline \Omega)$, if either $\epsilon > 0$ is sufficiently small or $\phi \equiv 0$ then there exists $f \in C^{3,\alpha}(\overline \Omega)$ such that
    \begin{align}
        \label{RegularJang}
        \underbrace{\left(g^{i j} - \frac{u^2f^i f^j}{1 + u^2 |df|_g^2}\right)\left(\frac{u\Hess_{i j}(f) + u_i f_j + f_i u_j}{\sqrt{1 + u^2|df|_g^2}} - K_{i j}\right)}_{ = \mathcal J(f)} & = \epsilon f, & x \in \Omega, \\ 
        \label{RegularJangBoundary}
        f & = \phi, & x \in \partial \Omega.
    \end{align}
\end{proposition}
We note that similar results have been obtained in Eichmair \cite[Lemma $2.2$]{EichmairPlateau}, Eichmair \cite[Proposition $5$]{EichmairReduction}, Andersson, Eichmair and Metzger \cite[Theorem $3.1$]{AnderssonEichmairMetzger} and Sakovich \cite[Lemma $4.2$]{SakovichJang}, see also references therein. To prove Proposition \ref{RegularizedBVP} we will apply the continuity method to the related boundary value problem
\begin{align}
    \label{RegularJangS}
    \left(g^{i j} - \frac{u^2f^i f^j}{1 + u^2 |df|_g^2}\right)\left(\frac{u\Hess_{i j}(f) + u_i f_j + f_i u_j}{\sqrt{1 + u^2|df|_g^2}} - sK_{i j}\right) & = \epsilon f, & x \in \Omega, \\ 
    \label{RegularJangSBoundary}
    f & = s\phi, & x \in \partial \Omega,
\end{align}
for $s \in [0,1]$.

\subsection{A-priori estimates}
\label{sec_Apriori}
\begin{proposition}
    \label{prop:AprioriEstimates}
    Suppose that $f_s \in C^{3,\alpha}(\overline \Omega)$ satisfies \eqref{RegularJangS}-\eqref{RegularJangSBoundary}. Then there is a constant $C$ depending on $(M, g, K)$, $\epsilon > 0$, $\Omega$, $\phi$, $|u|_{C^2(\overline \Omega)}$, $\inf_M u$ and $\alpha$, but independent of $s \in [0,1]$, such that
    \begin{equation*}
        |f_s|_{C^{2,\alpha}(\overline \Omega)} \leq C.
    \end{equation*}
\end{proposition}

\begin{proof}
    Throughout this proof, the constant $C > 0$ may change from line to line but depends only on $(M, g, K)$, $\epsilon > 0$, $\Omega$, $\phi$, $|u|_{C^2(\overline \Omega)}$, $\inf_M u$ and $\alpha$, and is independent of $s \in [0,1]$. For simplicity of notation, we will write $f$ instead of $f_s$. The proof proceeds in the following steps.

    \emph{Global $C^0$ estimate.} Suppose that $f$ attains its maximum at an interior point $p \in \Omega$. Then $\nabla f(p) = 0$ and $\Hess f(p)$ is negative semi-definite. Hence \eqref{RegularJangS} becomes
    \begin{align*}
        \epsilon f(p) = g^{ij}\left(u\Hess_{i j}(f) - sK_{i j}\right) = u\Delta^g f - s\tr_gK \leq |\tr_g K|.
    \end{align*}
    We get a similar inequality at an interior minimum point. Consequently, 
    \begin{align}
        \label{C0Bound}
        |f_s| \leq \max\left\{\epsilon^{-1}\sup_{\Omega} |\tr_g(K)|,\sup_{\partial \Omega} |\phi| \right\} \leq C \quad \text{on $\overline \Omega$},
    \end{align}
    where $C$ depends only on $(M,g,K)$, $\epsilon > 0$ and $\phi$.

    \emph{Interior $C^1$ estimate.} We set $v:= |df|_g^2$. Let $p \in \Omega$ be an interior maximum point for $v$. Then
    \begin{align}
        \label{vMax}
        v_i = 2f^j\Hess_{ij}f = 0,
    \end{align}
    at $p$ for $i = 1,\dots, n$. Further, we note that
    \begin{align}
        \label{MeanCurvatureFormula}
        \nabla_i \left(\frac{uf^i}{\sqrt{1 + u^2|df|_g^2}}\right) = \left(g^{i j} - \frac{u^2f^i f^j}{1 + u^2|df|_g^2}\right)\frac{u\Hess_{i j}(f) + u_i f_j + f_i u_j}{(1 + u^2|df|_g^2)^{1/2}} - \frac{\langle du,df\rangle_{g}}{(1 + u^2|df|_g^2)^{3/2}}. 
    \end{align}
    Combining \eqref{MeanCurvatureFormula} with \eqref{RegularJangS} we obtain
    \begin{align*}
        \epsilon f & = \nabla_i \left(\frac{uf^i}{\sqrt{1 + u^2|df|_g^2}}\right) + \frac{\langle du,df\rangle_{g}}{(1 + u^2|df|_g^2)^{3/2}} - \left(g^{ij} - \frac{u^2f^i f^j}{1 + u^2|df|_g^2}\right)sK_{ij}.
    \end{align*}
    Following Schoen and Yau \cite[Section $4$]{SchoenYau2} (see also Han and Khuri \cite[Section $2$]{HanKhuri}) we apply $\nabla_k$ to both sides and use the commutation relation $V^l \Ric_{j l} = \nabla_i \nabla_j V^i - \nabla_j \nabla_i V^i$, valid for any vector field $V$. After this we multiply the obtained relation by $f^k$ and sum over $k = 1,\dots, n$. Working at $p$, where \eqref{vMax} holds, we find
    \begin{align}
        \label{STAR}
        \begin{split}
            \epsilon |df|_g^2 
            & = f^k\nabla_i\nabla_k \left(\frac{uf^i}{\sqrt{1 + u^2|df|_g^2}}\right) + \frac{uf^if^k\Ric_{ik}}{\sqrt{1 + u^2|df|_g^2}} + \frac{(\nabla^2u)(\nabla f,\nabla f)}{(1 + u^2|df|_g^2)^{3/2}} \\
            & - \frac{3u\langle du,df\rangle^2|df|_g^2}{(1 + u^2|df|_g^2)^{5/2}} + \frac{2u\langle du,df\rangle f^i f^j}{(1 + u^2|df|_g^2)^2}sK_{ij} - \left(g^{ij} - \frac{u^2f^i f^j}{1 + u^2|df|_g^2}\right)sf^k\nabla_kK_{ij}.
        \end{split}
    \end{align}
    We now estimate the terms in the right hand side of \eqref{STAR} using \autoref{lem:WarpingBounds} as follows
    \begin{align*}
        \left|\frac{uf^i f^k\Ric_{i k}}{\sqrt{1 + u^2|df|_g^2}}\right| 
        & \leq Cv^{1/2}, &
        \left|\frac{(\nabla^2u)(\nabla f,\nabla f)}{(1 + u^2|df|_g^2)^{3/2}}\right| 
        & \leq C, \\
        \left|\frac{3u\langle du,df\rangle^2|df|_g^2}{(1+u^2|df|_g^2)^{5/2}}\right| 
        & \leq C, &
        \left|\frac{2u \langle du,df\rangle f^i f^j}{(1 + u^2|df|_g^2)^2}K_{i j}\right| 
        & \leq C, \\
        \left|\left(g^{i j} -\frac{u^2 f^i f^j}{1 + u^2|df|_g^2}\right)sf^k (\nabla_k K_{i j})\right| 
        & \leq C v^{1/2},
    \end{align*}
    with the constant $C$ depending on $(M, g, K)$, $|u|_{C^2(\overline \Omega)}$ and $\inf_M u$. Equation \eqref{STAR} now yields
    \begin{align}
        \label{STAR2}
        \epsilon v & \leq f^k\nabla_i \nabla_k \left(\frac{uf^i}{\sqrt{1 + u^2|df|_g^2}}\right) + C(1 + v^{1/2}).
    \end{align}

    We now turn to the first term in the right hand side of \eqref{STAR2}. Expanding, we obtain
    \begin{align}
        \label{STAR3}
        \begin{split}
        & f^k\nabla_i \nabla_k \left(\frac{uf^i}{\sqrt{1 + u^2|df|_g^2}}\right)
        = \nabla_i \left[\left(g^{i j} - \frac{u^2f^i f^j}{1 + u^2|df|_g^2}\right)\frac{uf^k \Hess_{k j} (f)}{\sqrt{1 + u^2|df|_g^2}}\right] \\
        & + \frac{f^k\nabla_i f^i u_k + f^kf^i \nabla_i u_k}{(1 + u^2|df|_g^2)^{3/2}}  - \frac{3f^k f^i u_ku u_i|df|_g^2}{(1+u^2|df|_g^2)^{5/2}} - \frac{g^{ij}u\Hess_{kj}(f)g^{k l}\Hess_{l i} (f)}{\sqrt{1 + u^2|df|_g^2}}.
        \end{split}
    \end{align}
    Recalling that $\nabla v = 0$ and that $\Hess v$ is negative semi-definite at $p$, we get using \autoref{lem:WarpingBounds}
    \begin{align*}
        \nabla_i \left[\left(g^{i j} - \frac{u^2f^i f^j}{1 + u^2|df|_g^2}\right)\frac{uf^k \Hess_{k j} (f)}{\sqrt{1 + u^2|df|_g^2}}\right] & = \left(g^{i j} - \frac{u^2f^i f^j}{1 + u^2|df|_g^2}\right)\frac{u\Hess_{ij} v}{2\sqrt{1 + u^2|df|_g^2}} \leq 0, \\
        \left|\frac{f^k f^i (\nabla_i u_k)}{(1 + u^2|df|_g^2)^{3/2}}\right| & \leq \frac{C|\Hess u|_g v}{(1 + u^2v)^{3/2}} \leq C, \\
        \left|\frac{3f^k f^i u_k u u_i|df|_g^2}{(1+u^2|df|_g^2)^{5/2}}\right| & \leq \frac{Cu |du|_g^2v^2}{(1+u^2v)^{5/2}} \leq C.
    \end{align*} 
    Here, the constant $C$ depends on $(M, g, K)$, $|u|_{C^2(\overline \Omega)}$ and $\inf_M u$. It remains to bound the following remaining two terms from the right hand side of \eqref{STAR3}
    \begin{equation*}
        E := \frac{f^k (\nabla_i f^i)u_k}{(1 + u^2|df|_g^2)^{3/2}} - \frac{u|\Hess f|_g^2}{\sqrt{1 + u^2|df|_g^2}}.
    \end{equation*}
    The first term may be estimated by noting that
    \begin{equation*}
        |\nabla_i f^i| \leq |\Hess f|_g \sqrt{n}
        \leq \kappa|\Hess f|_g^2 + \frac{n}{4\kappa}
    \end{equation*}
    where $\kappa > 0$ is arbitrary. Hence
    \begin{equation*}
        E \leq \frac{u|\Hess f|_g^2}{\sqrt{1 + u^2|df|_g^2}}\left(\frac{\kappa |df|_g|du|_g}{u(1 + u^2|df|_g^2)} - 1\right) + \frac{n|du|_g|df|_g}{4\kappa (1 + u^2|df|_g^2)^{3/2}}.
    \end{equation*}
    If $|df|_g|du|_g = 0$ at $p$, then we can bound the above by $0$, if not then one may choose $\kappa = |df|_g^{-1}|du|_g^{-1}u(1 + u^2|df|_g^2)$, to make the term in the parenthesis vanish, obtaining
    \begin{align*}
        E \leq \frac{n|du|_g^2|df|_g^2}{4u(1 + u^2|df|_g^2)^{5/2}} \leq C.
    \end{align*}
    We conclude that
    \begin{equation*}
        \epsilon v \leq C(1 + v^{1/2}),
    \end{equation*}
    at an interior maximum point $p$ of $v$ which implies that $|df|_g(p) \leq C$ where $C$ depends on $(M, g, K)$, $\epsilon$, $|u|_{C^2(\overline \Omega)}$ and $\inf_M u$. 

    \emph{Boundary $C^1$ estimate.} We will derive boundary gradient estimates using the so called barrier method as in Gilbarg and Trudinger \cite[Section 14]{GilbargTrudinger}. For this we will use \autoref{prop:BoundaryGrad}, which is proven in \autoref{appendix_appendixA}. 

    Thus, in order to prove the desired boundary gradient estimate, it suffices to provide the functions $\overline f$ and $\underline f$ as in \autoref{prop:BoundaryGrad}. Since $H_{\partial \Omega} - |\tr_{\partial \Omega} K| > 0$ we may choose $\epsilon > 0$ so that $H_{\partial \Omega} - |\tr_{\partial \Omega} K|  - \epsilon |\phi|> 0$ as well. Using the function $r = \text{dist}(\cdot, \partial \Omega)$ one may foliate a (possibly smaller) neighborhood $U$ of $\partial \Omega$ using the hypersurfaces $\mathcal E_r$ of constant $r$. If $\{x^1,\dots, x^{n-1}\}$ denote coordinates on $\partial \Omega$, then $\{x^1,\dots, x^{n-1},r\}$ are coordinates on $U \cap \overline\Omega$ and in these coordinates the metric can be written as $g = dr^2 + h_r$, where $h_r$ is the induced metric on $\mathcal E_r$. We assume without loss of generality that $U\cap \overline \Omega = \{0 \leq r < r_0\}$ where $r_0$ is chosen so small that $H_{\partial \mathcal E_r} - |\tr_{\partial \mathcal E_r} K|  - \epsilon |\phi|> 0$ for any $r \in [0,r_0)$.

    Since $r$ is a distance function we have $(h_r)^{r r} = 1$, $(h_r)^{\mu r} = 0$, $(h_r)^{\mu \nu} = g^{\mu \nu}$ and $\Gamma_{rr}^r = \Gamma_{r r}^\mu = \Gamma_{r \mu}^r = 0$. The mean curvature of the surfaces $\mathcal E_r$ can be computed as follows
    \begin{equation*}
        H_{\mathcal E_r} = g_r^{\mu \nu}(A_r)_{\mu \nu} = -g_r^{\mu\nu}\langle \nabla_\mu \partial_\nu , -\partial_r \rangle_g = g_r^{\mu \nu}\langle \Gamma_{\mu \nu}^r \partial_r, \partial_r\rangle = (g_r)^{\mu \nu}\Gamma_{\mu \nu}^r,
    \end{equation*}
    where we use Greek indices to denote coordinates on $\mathcal{E}_r$ and $A_r$ is the second fundamental form of $\mathcal E_r$. We now claim that for $B > 0$ sufficiently large, the functions $\overline f = s\phi + Br$ and $\underline f = s\phi - Br$ are boundary barriers satisfying the conditions of \autoref{prop:BoundaryGrad}. To show this, we compute
    \begin{align*}
        \mathcal J(\overline f) 
        & = \left(g_r^{\mu \nu} - \frac{s^2u^2\phi^\mu \phi^\nu}{1 + s^2u^2|d\phi|_{g_r}^2 + u^2B^2}\right)\left(\frac{us\Hess^{\partial \Omega}_{\mu \nu}(\phi) - u\Gamma_{\mu \nu}^r B + su_\mu \phi_\nu + s\phi_\mu u_\nu}{\sqrt{1 + s^2u^2|d\phi|_{g_r}^2 + u^2B^2}} - K_{\mu \nu}\right) \\
        & \quad + \left(0 - \frac{su^2B\phi^\mu}{1 + s^2u^2|d\phi|_{g_r}^2 + u^2B^2}\right)\left(\frac{0 - us\Gamma_{\mu r}^\kappa \phi_\kappa + Bu_\mu + s\phi_\mu u_r}{\sqrt{1 + s^2u^2|d\phi|_{g_r}^2 + u^2B^2}} - K_{\mu r}\right) \\
        & \quad + \left(1 - \frac{u^2 B^2}{1 + s^2u^2|d\phi|_{g_r}^2 + u^2B^2}\right)\left(\frac{2Bu_r}{\sqrt{1 + s^2u^2|d\phi|_{g_r}^2 + u^2B^2}} - K_{rr}\right) \\
        & = -g_r^{\mu \nu}\Gamma_{\mu \nu}^r - \tr_{\mathcal E_r } K + O(B^{-1}) \\
        & < - H_{\mathcal E_r} + |\tr_{\mathcal E_r} K| + CB^{-1}.
    \end{align*}
    A similar computation shows that
    \begin{equation*}
        \mathcal J(\underline f) = H_{\mathcal E_r} - \tr_{\mathcal E_r} K + O(B^{-1}) > H_{\mathcal E_r} - |\tr_{\mathcal E_r} K| - CB^{-1}.
    \end{equation*}
    Consequently, by choosing $B$ large enough depending on $C$, we have the inequalities
    \begin{equation*}
        \mathcal J(\overline f) - \epsilon \overline f < 0 \quad \text{ and } \quad \mathcal J(\underline f) - \epsilon \underline f > 0,
    \end{equation*}
    in $U \cap \overline \Omega$. The boundary condition $\underline f|_{\partial \Omega} = \overline f|_{\partial \Omega} = s\phi$ is automatically fulfilled. We also note that by choosing $B$ large enough we can ensure that
    \begin{align*}
        \underline f|_{\{r = r_0\}} 
        & \leq - \max\left\{\sup_{\Omega} \epsilon^{-1}|\tr_g(K)|, \sup_{\partial \Omega} |\phi| \right\}, \\ 
        \overline f|_{\{r = r_0\}} 
        & \geq \max\left\{\sup_{\Omega} \epsilon^{-1}|\tr_g(K)|, \sup_{\partial \Omega} |\phi| \right\}.
    \end{align*}
    Combining this with \eqref{C0Bound}, we may now apply \autoref{prop:BoundaryGrad} to find that $\sup_{\partial \Omega}|df|_g \leq C$. 

    \emph{$C^{2,\alpha}$ estimate.} In local coordinates we may write $\mathcal J(f) - \epsilon f = 0$ as a quasilinear elliptic equation
    \begin{equation}
        \label{JangQuaslinear}
        a^{ij}(x,Df)D_{ij} f + b(x,f,Df) = 0,
    \end{equation}
    where
    \begin{align*}
        a^{ij}(x,Df) & = \left(g^{ij} - \frac{u^2f^if^j}{1 + u^2|df|_g^2}\right)\frac{u}{\sqrt{1 + u^2|df|_g^2}}, \\
        b(x,f,Df) & = \left(g^{ij} - \frac{u^2f^if^j}{1 + u^2|df|_g^2}\right)\left(\frac{-u\Gamma_{ij}^k f_k + 2u_if_j}{\sqrt{1 + u^2|df|_g^2}} - sK_{ij}\right) - \epsilon f.
    \end{align*}
    We have already shown that there is a constant $C$ independent of $s$ such that
    \begin{equation*}
        \sup_{\Omega} |f| + \sup_{\Omega} |Df| \leq C.
    \end{equation*}
    In particular, this implies that the equation  \eqref{JangQuaslinear} is strictly elliptic with an ellipticity constant $\lambda_C > 0$ independent of $s \in [0,1]$. Moreover, there is a constant $\mu_C > 0$ independent of $s$ such that
    \begin{equation*}
        \sum_{i,j} (|a^{ij}(x,v)| + |D_xa^{ij}(x,v)| + |D_va^{ij}(x,v)|) + |b(x,y,v)| < \mu_C,
    \end{equation*}
    for all $x \in \overline{\Omega}$ whenever $|y|,|v| < C$. It follows by \cite[Theorem $13.7$]{GilbargTrudinger} that there exists $\beta := \beta(C,\mu_{C}/\lambda_{C}, \Omega, |\varphi|_{C^{2,\alpha}(\Omega)}) = \beta_C$ such that
    \begin{equation*}
        |Df|_{C^{0,\beta}(\Omega)} < C' = C'(C,\mu_C/\lambda_C, \Omega, |\varphi|_{C^{2,\alpha}(\Omega)}).
    \end{equation*}
    We can now treat the equation \eqref{RegularJangS} as a linear elliptic equation, namely
    \begin{equation*}
        a^{ij}D_{ij}f + b^iD_if + cf = d,   
    \end{equation*}
    where
    \begin{align*}
        a^{ij} = \frac{u\overline g^{ij}}{\sqrt{1 + u^2|df|_g^2}}, \quad
        b^k = \frac{-u\overline g^{ij}\Gamma_{ij}^k + 2\overline g^{ik}u_i}{\sqrt{1 + u^2|df|_g^2}}, \quad
        c =  - \epsilon, \quad 
        d = s\overline g^{ij}K_{ij},
    \end{align*}
    and $\overline g^{ij}$ is given in \eqref{eq:MetricComponents}. By the above, $|a^{ij}|_{C^{0,\beta}(\overline \Omega)}, |b^{i}|_{C^{0,\beta}(\overline \Omega)}, |c|_{C^{0,\beta}(\overline \Omega)} \leq C$ for some constant $C$ independent of $s$. Using standard Schauder estimates \cite[Theorem $6.6$]{GilbargTrudinger} we conclude that $|f_s|_{C^{2,\beta}(\overline \Omega)} < C$, uniformly in $s \in [0,1]$. It follows that $|f_s|_{C^{1,\alpha}(\overline \Omega)} < C$ uniformly in $s \in [0,1]$. Applying \cite[Theorem $6.6$]{GilbargTrudinger} one more time, we conclude that $|f_s|_{C^{2,\alpha}(\overline \Omega)} < C$, uniformly in $s \in [0,1]$. 
\end{proof}

\subsection{Proof of \autoref{RegularizedBVP} by the method of continuity}

\begin{proof}[Proof of \autoref{RegularizedBVP}]
    Let $S$ be the set of those $s \in [0,1]$ for which the system \eqref{RegularJangS}-\eqref{RegularJangSBoundary} has a solution $f_s \in C^{3,\alpha}(\overline \Omega)$. $S$ is nonempty since $0 \in S$. We will now show that $S$ is both open and closed which will imply that $S = [0,1]$. 

    \emph{The set $S$ is closed.} Suppose that $\{s_k\}_{k = 1}^\infty \subset S$ is a sequence converging to $s_0 \in [0,1]$. Then by the results of Section \ref{sec_Apriori} the corresponding solutions $f_{s_k}$ to \eqref{RegularJangS}-\eqref{RegularJangSBoundary} are uniformly bounded in $C^{2,\alpha}(\overline \Omega)$. By the \cite[Lemma $6.36$]{GilbargTrudinger} there is a subsequence $\{s_{k_i}\}_{i = 1}^\infty$ such that $|f_{s_{k_i}} - f|_{C^{2,\beta}(\overline \Omega)} \to 0$ as $i \to \infty$ for some $f \in C^{2,\beta}(\overline \Omega)$ and $\beta \in (0,\alpha)$. $f$ is then a solution to \eqref{RegularJangS}-\eqref{RegularJangSBoundary} with $s = s_0$. Applying standard elliptic estimates we conclude that $f \in C^{3,\alpha}(\overline \Omega)$ and hence $s_0 \in S$, as claimed. 
    
    \emph{The set $S$ is open.} This is a consequence of the implicit function theorem. Consider the $C^1$ map $T: C^{3,\alpha}(\overline \Omega) \times \mathbb{R} \to C^{1,\alpha}(\overline \Omega) \times C^{3,\alpha}(\partial \Omega) \times \mathbb{R}$ given by
    \begin{equation}
        \label{DefT}
        T(f,s) := (T_1(f,s), T_2(f,s), T_3(f,s)) := (H_g(f) - s\tr_g(K)(f) - \epsilon f,f|_{\partial \Omega} - s\varphi,s). 
    \end{equation}
    The linearization of $T$ at $(f,s)$ is the map $DT_{(f,s)}: C^{3,\alpha}(\overline \Omega) \times \mathbb{R} \to C^{1,\alpha}(\overline \Omega) \times C^{3,\alpha}(\partial \Omega) \times \mathbb{R}$ given by
    \begin{align*}
        DT_{(f,s)}(\eta,\sigma) 
        & = (D(T_1)_{(f,s)}(\eta,\sigma),D(T_2)_{(f,s)}(\eta,\sigma),D(T_3)_{(f,s)}(\eta,\sigma)) \\
        & = (D(T_1)_{(f,s)}(\eta,\sigma),\eta|_{\partial \Omega} - \sigma \varphi,\sigma).
    \end{align*}
    More specifically
    \begin{equation}
        \label{eq:LinearizedOperator}
        D(T_1)_{(f,s)}(\eta,\sigma) = G_{(f,s)}^{ij}(x)\Hess_{ij} \eta + b_{(f,s)}^k(x,f,s) \eta_k + \sigma d_{(f,s)}(x,f,s) - \epsilon \eta,
\end{equation}
    where the coefficients $G^{ij}$, $b^k$ and $d$ are given by
    \begin{align*}
        G_{(f,s)}^{ij}(x) 
        & = \frac{u\overline g^{ij}}{(1 + u^2|df|_g^2)^{1/2}}, \\
        b_{(f,s)}^k(x) 
        & = \frac{-u^2(\overline g^{ij}f^k + \overline g^{jk}f^i + \overline g^{ki}f^j)(u\Hess_{ij}f + u_if_j + f_iu_j)}{(1+u^2|df|_g^2)^{3/2}} + 
        \frac{2\overline g^{ik}u_i}{(1+u^2|df|_g^2)^{1/2}} \\
        d_{(f,s)}(x) 
        & = -\overline g^{ij}K_{ij},
    \end{align*}
    and $\overline g^{ij}$ is given in \eqref{eq:MetricComponents}. For fixed $u \in C^{2,\alpha}(\overline \Omega)$, $f \in C^{3,\alpha}(\overline \Omega)$ and $s \in [0,1]$, the above coefficients are all in $C^{1,\alpha}(\overline \Omega)$ and the differential operator in the right hand side of \eqref{eq:LinearizedOperator} is uniformly elliptic. Recalling that $\epsilon > 0$ we can use standard elliptic theory to see that for every $\psi \in C^{1,\alpha}(\overline \Omega)$, $\xi \in C^{3,\alpha}(\partial \Omega)$ and $t \in \mathbb{R}$ there is a unique pair $(\eta, \sigma) \in C^{3,\alpha}(\overline \Omega) \times \mathbb{R}$ satisfying
    \begin{equation*}
        \begin{cases}
            G_{(f,s)}^{ij}(x)\Hess_{ij} \eta + b_{(f,s)}^k(x) \eta_k + \sigma d_{(f,s)}(x) - \epsilon \eta = \psi, & \text{ in $\Omega$} \\
            \eta = \sigma \varphi + \xi, & \text{ on $\partial \Omega$} \\
            \sigma = t. & 
        \end{cases}
    \end{equation*}
    Now suppose $s_0 \in S$ and $f_{s_0} \in C^{3,\alpha}(\overline \Omega)$ is a solution to \eqref{RegularJangS}-\eqref{RegularJangSBoundary} with $s = s_0$. By the implicit function theorem there is an interval $I = (s_0 - \delta,s_0 + \delta)$ such that for all $s \in I \cap [0,1]$, \eqref{RegularJangS}-\eqref{RegularJangSBoundary} has a solution $f_s \in C^{3,\alpha}(\overline \Omega)$. In other words, $S$ is open in $[0,1]$ as claimed. 
\end{proof}


\section{Geometric solution to the Jang equation}
\label{sec_GeometricSolution}
In this section we will take the $\epsilon$-limit of a sequence of solutions to the regularized problems \eqref{RegularJang}-\eqref{RegularJangBoundary} with $\phi=0$ in \Secref{sec_RegularBVP}, thereby constructing a so called geometric solution to the generalized Jang equation. To study this geometric solution, we will need a priori estimates of the vertical part of the normal. We begin by recalling the following proposition from Han and Khuri \cite[Section $3$]{HanKhuri}, generalizing an identity by Schoen and Yau \cite[Equation $2.18$]{SchoenYau2}.\\
\begin{proposition}
    \label{PDENormalProposition}
    Let $\Gamma(f) = \{(x,f(x)) : x \in U\} \subset M\times_u \mathbb{R}$ be a graphical submanifold over the open set $U \subset M$, with downward pointing unit normal $\nu$ and with mean curvature $H$ computed as the tangential divergence of $\nu$. Assuming that $g \in C^2_\text{loc}(U)$, $f \in C^3_\text{loc}(U)$ and that $\nu$ and $H$ are extended along the $\mathbb R$ factor so that
    \begin{align*}
        \nu_{(x,f(x))} = \nu_{(x,f(x) + t)}, \quad H_{(x,f(x))} = H_{(x,f(x) + t)}, \quad \text{for all $x \in U$ and $t \in \mathbb R$},
    \end{align*}
    we then have
    \begin{equation}
        \label{PDENormalT}
        \Delta^{\overline g}\langle \partial_t, -\nu \rangle + (|A|_{\overline g}^2 + \nu(H) + \widetilde \Ric(\nu,\nu))\langle \partial_t, -\nu \rangle = 0.
    \end{equation}
\end{proposition}
We will now apply \eqref{PDENormalT} when $f$ satisfies an equation of mean curvature type, deriving a differential inequality for the vertical part of the normal, $\langle u^{-1}\partial_t, -\nu\rangle$.\\
\begin{proposition}
    \label{DifferentialInequalityGraphs}
    Suppose that $\alpha \in (0,1)$, $\tau > 0$, that $(M,g,K)$ is a $C^{2,\alpha}_\tau$ asymptotically anti-de Sitter initial data set and that $\Gamma(f) = \{(x,f(x)) : x \in U\} \subset (M\times_u \mathbb{R},\widetilde g = g + u^2df^2)$ is a graphical submanifold over the open set $U \subset M$, with downward pointing unit normal $\nu$ and mean curvature $H$. 
    
    Suppose that $u - \rho^{-1} \in C^{2,0}_0(M)$, that $f \in C^3_\text{loc}(U)$ and that there is a $C^1$ function $F: (M \times_u \mathbb{R}) \times T(M\times_u \mathbb{R}) \to \mathbb{R}$ such that $H(x,f(x)) = F(x,f(x),\nu)$ for all $x \in U$. Let
    \begin{align*}
        D & := \{(p,X) \in (M \times_u \mathbb{R}) \times T(M\times_u \mathbb{R}) : |X|_{\widetilde g} \leq 1\},
    \end{align*}
    and let \( \widetilde g_\text{Sas} \) be the Sasaki metric on the tangent bundle \( T\widetilde M \) (see \cite[Section 3]{SasakiRef}), and suppose additionally that $|\nabla^{\widetilde g_\text{Sas}} F|_{\widetilde g_\text{Sas}} \leq C_0$ on \( D \), that the map $t\mapsto F(x,t,X)$ is non-decreasing for all $x \in M$ and $X \in T\widetilde M$. Then there is a constant $C > 0$ depending only on $(M,g,K)$, $|u-\rho^{-1}|_{C^{2,0}_0(M)}$, $\min_M u$ and $C_0$ such that
    \begin{equation}
        \label{eq:DiffIneqWithoutGrad}
        \Delta^{\overline g}(\langle u^{-1}\partial_t, -\nu\rangle^{1/2}) \leq C\langle u^{-1}\partial_t, -\nu\rangle^{1/2}. 
    \end{equation}
\end{proposition}
\begin{remark}
    \label{PowerRemovesBoundaryTerm}
    Note that \eqref{eq:DiffIneqWithoutGrad} differs from the inequality in \cite[Lemma A.1]{EichmairMetzger} of Eichmair and Metzger. Working with the quantity $\langle u^{-1}\partial_t, -\nu\rangle^{1/2}$ instead of $\langle u^{-1}\partial_t, -\nu\rangle$ allows us to eliminate a gradient term, see \cite[Equation $17$]{EichmairMetzger} and \cite[Remark A.2]{EichmairMetzger}.  
\end{remark} \hfill 
\begin{remark}
    \autoref{DifferentialInequalityGraphs} also generalizes a result of Han and Khuri \cite[Theorem $3.2$]{HanKhuri}, which holds in dimension $3$. The proof of \cite[Theorem $3.2$]{HanKhuri} uses the boundedness of $|A|_{\overline g}$ and as such does not extend to higher dimensions. 
\end{remark}
\begin{proof}
    We adapt the argument of Eichmair and Metzger \cite[Lemma A.1]{EichmairMetzger} to the warped product setting. We let \( x \in U \) be fixed but arbitrary and let $x^i$, $i=1,\ldots,n$, be geodesic normal local coordinates centered at $ x \in M$ and we denote by $t$ the coordinate on $\mathbb{R}$. Further, let $p_1, \ldots p_n, p_t$ be the canonical coordinates on $T(M\times_u \mathbb{R})$ induced by $(x^1,\ldots, x^n,t)$. In the rest of this proof, all quantities defined on \( U \subset M \) are evaluated at \( x \) and all quantities defined on \( \Gamma(f) \) and \( U \times \mathbb R \) are evaluated at \( (x,f(x)) \). Our assumptions on the function \( F \) imply that
    \begin{equation}
        \label{eq:ConditionsF}
        |F_{x_i}| + |F_{p_i}| + \frac{|F_{t}|}{u} + \frac{|F_{p_t}|}{u} \leq C_0.
    \end{equation}
    The factor \( u^{-1} \) in the last two terms appears since \( |\partial_t|_{\widetilde g} = u \). Defining $w := \langle u^{-1}\partial_t,-\nu\rangle$, \eqref{eq:UnitNormalGraph} implies
    \begin{equation}
        \label{NormalComponents}
        \nu^i = uwf^i, \quad \nu^t = -\frac{w}{u}.
    \end{equation}
    Consequently, we have
    \begin{align}
        \label{ExpandMC1}
        \begin{split}
            \nu(H)
            & = \nu^i\partial_i(F(x,f(x),\nu(x)))
            = \nu^i\left(F_{x^i} + F_t f_i + F_{p_j}\frac{\partial \nu^j}{\partial x^i} + F_{p_t}\frac{\partial \nu^t}{\partial x^i}\right) \\
            & = \nu^iF_{p_t}\widetilde \nabla_i \nu^t + \nu^iF_{p_j}\widetilde \nabla_i\nu^j + F_t |df|_g^2\langle \partial_t, -\nu \rangle\\
            & + \nu^i(F_{x^i} - F_{p_j}\widetilde\Gamma_{im}^j \nu^m - 
            F_{p_j}\widetilde\Gamma_{it}^j \nu^t - 
            F_{p_t}\widetilde\Gamma_{im}^t \nu^m -  
            F_{p_t}\widetilde\Gamma_{it}^t \nu^t).
        \end{split}
    \end{align}
    The last two terms on the right-most side of \eqref{ExpandMC1} can be bounded from below as
    \begin{align*}
        F_t |df|_g^2\langle \partial_t, -\nu \rangle 
        & \geq 0,
    \end{align*}
    and recalling \eqref{eq:WarpedChristoffels}, \eqref{eq:UnitNormalGraph} and that \( \Gamma_{ij}^k = 0\) we have
    \begin{align*}
        \nu^i(F_{x^i} - F_{p_j}\widetilde\Gamma_{im}^j \nu^m - F_{p_j}\widetilde\Gamma_{it}^j \nu^t - F_{p_t}\widetilde\Gamma_{it}^t \nu^t -  F_{p_t}\widetilde\Gamma_{it}^j \nu^t)
        & = \nu^i\biggl(F_{x^i} -   
        F_{p_t}\frac{u_i}{u}\nu^t\biggr) \\
        & = \nu^iF_{x^i} +   
        w\frac{F_{p_t}}{u}\frac{\nu^iu_i}{u} \\
        & \geq -C,
    \end{align*}
    where the constant $C$ above depends only on the geometry of $(M,g,K)$, \( u \) and the constant $C_0$. As a consequence of these inequalities, \eqref{ExpandMC1} yields
    \begin{equation}
        \label{ExpandMC2}
        \nu(H) \geq \nu^iF_{p_t}\widetilde \nabla_i \nu^t + \nu^iF_{p_j}\widetilde \nabla_i\nu^j - C.
    \end{equation}
    A computation using the identity $w = (1 + u^2|df|_g^2)^{-1/2}$ shows that
    \begin{align}
        \label{GradientUW}
        \begin{split}
            (uw)_i
            & = \frac{u_i}{(1 + u^2|df|_g^2)^{1/2}} - \frac{u^2u_i|df|_g^2 + u^3f^k\nabla_if_k}{(1 + u^2|df|_g^2)^{3/2}} \\
            & = \frac{u_i}{(1 + u^2|df|_g^2)^{1/2}}\left(1 -\frac{u^2|df|_g^2}{1 + u^2|df|_g^2}\right) - \frac{u^3f^k\nabla_if_k}{(1 + u^2|df|_g^2)^{3/2}} \\
            & = w^3u_i - u^3w^3f^k\nabla_if_k.
        \end{split}
    \end{align}
    Further, \eqref{NormalComponents} and \eqref{GradientUW} imply
    \begin{align*}
        & \quad uw\nu^i\widetilde \nabla_i\nu^j \\
        & = u^2w^2f^i\left(uw\nabla_i f^j -(uw)^3f^jf^k\nabla_if_k + w^3f^ju_i\right) && \text{ by \eqref{NormalComponents} and \eqref{GradientUW}} \\
        & = u^3w^3f^i\left(g^{jk} - u^2w^2f^jf^k\right)\nabla_i f_k + u^2w^5f^if^ju_i \\
        & = u^3w^3f^i\overline g^{jk} + u^2w^5f^if^ju_i && \text{ by $\overline g^{ij} = g^{ij} - u^2w^2f^if^j$} \\
        & = \overline g^{jk}(-(uw)_k + w^3u_k) + u^2w^5f^if^ju_i && \text{ by \eqref{GradientUW}} \\
        & = -\overline g^{ij}(uw)_i + \overline g^{ij}w^3u_i + u^2w^2f^if^jw^3u_i \\
        & = -\overline g^{ij}(uw)_i + g^{jk}w^3u_k && \text{ by $\overline g^{ij} = g^{ij} - u^2w^2f^if^j$.} 
    \end{align*}
    After multiplying the equality resulting from the above calculation by $F_{p_j}$ we have
    \begin{equation}
        \label{DangerousTerm}
        uw\nu^iF_{p_j}\widetilde \nabla_i\nu^j = -\overline g^{ij}(uw)_iF_{p_j} + g^{jk}w^3F_{p_j}u_k.
    \end{equation}
    Due to \autoref{lem:WarpingBounds} we have $|du|_g, |\Hess u|_g \leq Cu$ for some constant $C$ depending on $(M,g,K)$ and $u$. By definition we have $0\leq w \leq 1$. These bounds imply $\left|w^3g^{jk}F_{p_j}u_k\right| \leq Cuw$ and combining this with \eqref{ExpandMC2} and \eqref{DangerousTerm} it follows that
    \begin{equation}
        \label{InequalityNuH}
        -uw\nu(H) \leq uwF_{p_t}\nu^i\widetilde \nabla_i (-\nu^t) + \overline g^{ij}(uw)_iF_{p_j} + Cuw.
    \end{equation}
    Due to \eqref{NormalComponents} we have
    \begin{equation*}
        \nu^i\widetilde \nabla_i (-\nu^t)
        = \frac{\nu^iw_i}{u} - \frac{w\nu^i u_i}{u^2}.
    \end{equation*}
    This equality combined with \eqref{InequalityNuH} yields
    \begin{equation}
        \label{InequalityNuH*}
        -uw\nu(H) \leq w F_{p_t} \nu^iw_i - w^2F_{p_t}\frac{\nu^i u_i}{u} + \overline g^{ij}(uw)_iF_{p_j} + Cuw.
    \end{equation}
    Consequently, \eqref{PDENormalT} implies
    \begin{equation*}
        \Delta^{\overline g}(uw) = (\underbrace{-\widetilde{\Ric}(\nu,\nu)}_{\leq C} \underbrace{-|A|_{\overline g}^2}_{\leq 0} - \nu(H))uw \leq Cuw - uw\nu(H),
    \end{equation*}
    which combined with \eqref{InequalityNuH*} yields, after rearrangement, 
    \begin{equation}
        \label{DifferentialInequalityUW}
        \Delta^{\overline g}w  \leq \left(C - w\frac{F_{p_t}}{u}\frac{\nu^i u_i}{u} + \overline g^{ij}F_{p_j}\frac{u_i}{u} - \frac{\Delta^{\overline g} u}{u}\right)w + \overline g^{ji}\left(F_{p_j} - \frac{2u_j}{u}\right)w_i + w \frac{F_{p_t}}{u} \nu^iw_i.
    \end{equation}
    Recalling that $|du|_g, |\Hess u|_g \leq Cu$, equation \eqref{eq:LaplacianDifferenceGraph} implies
    \begin{equation*}
        \left|\frac{\Delta^{\overline g} u}{u}\right|
        = \left|\frac{\Delta u}{u} -\frac{u(\Hess u)(\nabla f, \nabla f)}{1 + u^2|df|_g^2} + \frac{|du|_g^2|df|_g^2}{1 + u^2|df|_g^2} -\frac{\langle df,du\rangle H_{\Gamma(f)}}{\sqrt{1 + u^2|df|_g^2}}\right| \leq C,
    \end{equation*}
    hence
    \begin{equation}
        \label{eq:FactorDegree0}
        \left|C - w\frac{F_{p_t}}{u}\frac{\nu^i u_i}{u} + \overline g^{ij}F_{p_j}\frac{u_i}{u} - \frac{\Delta^{\overline g} u}{u}\right| \leq C.
    \end{equation}
    Furthermore, a calculation using the identity \( g^{ij} = \overline g^{ij} + \nu^i\nu^j \) (see \eqref{eq:MetricComponents}) shows that
    \begin{align*}
        \nu^i w_i 
        = g^{ij}(g_{jk}\nu^k)w_i 
        & = \overline g^{ij}(g_{jk}\nu^k)w_i + \nu^i \nu^j (g_{jk}\nu^k) w_i \\
        & = \overline g^{ij}(g_{jk}\nu^k)w_i + (1 - w^2)\nu^i w_i,
    \end{align*}
    which can be rearranged into
    \begin{equation}
        \label{eq:FactorDegree1}
        \nu^i w_j = w^{-2}\overline g^{ij}(g_{jk}\nu^k)w_i.
    \end{equation}
    We may now conclude by \eqref{eq:FactorDegree0}, \eqref{eq:FactorDegree1} and \eqref{DifferentialInequalityUW} that
    \begin{equation}
        \label{eq:HarnackAlmost}
        w^{-1}\Delta^{\overline g}w  \leq C 
        + \overline g^{ij}\left(F_{p_j} - \frac{2u_j}{u} + \frac{F_{p_t}}{uw}(g_{jk}\nu^k)\right)\frac{w_i}{w}.
    \end{equation}
    Defining $X_i := F_{p_i} - \frac{2u_i}{u} + \frac{F_{p_t}}{uw}(g_{ik}\nu^k)$, then using our assumptions on the function \( F \), we have by the Cauchy-Schwarz inequality
    \begin{equation}
        \label{eq:BarCS}
        \overline g^{ij}X_j\frac{w_i}{w}
        \leq \bigl(\overline g^{ij}X_iX_j\bigr)^{1/2}\biggl(\overline g^{kl}\frac{w_l}{w}\frac{w_k}{w}\biggr)^{1/2} 
        = \bigl(\overline g^{ij}X_iX_j\bigr)^{1/2}|w^{-1}dw|_{\overline g}.
    \end{equation}
    Due to the equality
    \begin{align*}
        \overline g^{ij}(g_{ik}\nu^k)(g_{jl}\nu^l)
        & = (g^{ij} - \nu^i\nu^j)(g_{ik}\nu^k)(g_{jl}\nu^l) 
        = g^{ij}g_{ik}g_{jl}\nu^l\nu^k - (\nu^ig_{ik}\nu^k) (\nu^kg_{jl}\nu^l) \\
        & = (1 - w^2) - (1-w^2)^2 \\
        & = w^2(1-w^2),
    \end{align*}
    we conclude by the triangle inequality that 
    \begin{equation}
        \label{eq:BarTriangle}
        (\overline g^{ij}X_iX_j)^{1/2} 
        \leq (\overline g^{ij}F_{p_i}F_{p_j})^{1/2} 
        + 2\biggl(\overline g^{ij}\frac{u_i}{u}\frac{u_j}{u}\biggr)^{1/2} 
        + \frac{|F_{p_t}|}{uw}(\overline g^{ij}(g_{ik}\nu^k)(g_{jl}\nu^l))^{1/2}
        \leq C.
    \end{equation}
    All in all, the inequalities \eqref{eq:BarCS} and \eqref{eq:BarTriangle} can be used in \eqref{eq:HarnackAlmost} to conclude that
    \begin{equation*}
        w^{-1}\Delta^{\overline g}w  \leq C 
        + C\frac{|dw|_{\overline g}}{w}.
    \end{equation*}
    An application of the inequality $2ab\leq a^2 + b^2$ now yields
    \begin{equation*}
        w^{-1}\Delta^{\overline g}w \leq C + \frac{|dw|_{\overline g}^2}{2w^2}.
    \end{equation*}
    This in turn implies
    \begin{equation*}
        w^{-1/2}\Delta^{\overline g} w^{1/2}
        = \frac{\Delta^{\overline g} w}{2w} - \frac{|dw|_{\overline g}^2}{4w^2} \leq C,
    \end{equation*}
    and \eqref{eq:DiffIneqWithoutGrad} follows.
\end{proof}

\begin{proposition}
    \label{GradientFromOscillation}
    Suppose that $\alpha \in (0,1)$, $\tau > 0$ are real numbers and that $(M,g,K)$ is a $C^{2,\alpha}_\tau$ asymptotically anti-de Sitter initial data set. Suppose moreover that $f \in C^3_\text{loc}(B_R(x_0))$ satisfies $\mathcal J(f) = \epsilon f$ for some $\epsilon \in [0,1)$ on $B_R(x_0) \subset M$, that $u - \rho^{-1} \in C^{2,0}_0(M)$ and that $\sup_{B_R(x_0)}|uf| \leq T$ for some $T > 0$. Then there is a constant $C > 0$ depending only on $(M,g,K)$, $|u - \rho^{-1}|_{C^{2,0}_0(M)}$, $\inf_M u$, $R$ and $T$ such that
    \begin{align*}
        |(udf)(x_0)|_g \leq C. 
    \end{align*}
\end{proposition}
\begin{proof}
    The argument we present is inspired by Eichmair \cite[Lemma 2.1]{EichmairPlateau}, who in turn adapts an argument of Spruck \cite{Spruck}. Throughout this proof, the constant $C > 0$ may change from line to line but depends only on $(M,g,K)$, $|u|_{C^2(B_R(x_0))}$, $\inf_{B_R(x_0)}u$, $R$ and $T$. \autoref{DifferentialInequalityGraphs} implies that $w := \langle u^{-1}\partial_t,-\nu \rangle = (1 + u^2 |df|_g^2)^{-1/2}$ satisfies
    \begin{equation}
        \label{HarnackIneq}
        \Delta^{\overline g} w^{1/2} \leq Cw^{1/2}.
    \end{equation}
    We also note that the assumption $\mathcal J(f) = \epsilon f$ implies
    \begin{equation*}
        H = \tr_{\Gamma(f)} \overline K + \epsilon f = \tr_{\Gamma(f)} K + \frac{\langle \nabla u, \nu\rangle_{\widetilde g}}{u} +  \epsilon f, 
    \end{equation*}
    which combined with \autoref{lem:WarpingBounds} and the assumption $\sup_{B_R(x_0)}|uf| \leq T$ implies
    \begin{equation}
        \label{eq:BoundedMeanCurvature}
        |H| \leq C,
    \end{equation}
    where the constant $C$ depends on $(M,g,K)$, $|u - \rho^{-1}|_{C^{2,0}_0(M)}$, $\inf_M u$, $T$ and $\epsilon$. We let $L > 0$ be a constant to be specified and we define the functions
    \begin{align*}
        \eta := e^{L\phi} - 1,\quad
        \phi(x) := -u(x_0)f(x_0) + u(x)f(x) + R - \frac{T + R}{R^2}d(x_0,x)^2.
    \end{align*}
    The set $\Omega := \{x \in B_R(x_0) : \phi(x) > 0\}$ is open and contains $x_0$. We have $\phi \equiv 0$ on $\partial \Omega$ and hence also $\eta = 0$ on $\partial \Omega$. At an interior maximum point $p \in \Omega$ of $w^{-1/2}\eta$ we have
    \begin{align}
        \label{eq-weta_grad}
        0 = d(w^{-1/2}\eta) & = \eta \overline\nabla (w^{-1/2}) + w^{-1/2}\overline\nabla \eta \\
        \label{eq-weta_Laplace}
        0 \geq w^{1/2}\overline\Delta (w^{-1/2}\eta) & = \overline\Delta \eta + 2w^{1/2}\langle \overline\nabla \eta, \overline\nabla w^{-1/2}\rangle_{\overline g} + \eta w^{1/2} \overline\Delta (w^{-1/2}).
    \end{align}
    The equation \eqref{eq-weta_grad} implies
    \begin{equation*}
        2w^{1/2}\langle \overline\nabla \eta, \overline\nabla(w^{-1/2})\rangle
        = -2w\eta |\overline\nabla(w^{-1/2})|_{\overline g}^2,
    \end{equation*}
    and by \eqref{HarnackIneq} we also have
    \begin{equation*}
        w^{1/2} \overline\Delta (w^{-1/2})
        = -w^{-1/2}\overline \Delta w^{1/2} + 2w|\overline \nabla w^{-1/2}|_{\overline g}^2 \geq -C + 2w|\overline \nabla w^{-1/2}|_{\overline g}^2.
    \end{equation*}
    In conclusion, \eqref{eq-weta_Laplace} implies that at an interior maximum point $p$ of $w^{-1/2}\eta$ we have
    \begin{equation}
        \label{SimplifiedMaxima}
        C\eta \geq \overline\Delta \eta.
    \end{equation}
    Next, we note that by the definitions of $\eta$ and $\phi$, we have the identities
    \begin{align}
        \label{LaplaceEta}
        \overline\Delta \eta & = Le^{L\phi}\overline \Delta \phi + L^2 e^{L\phi}|\overline\nabla \phi|_{\overline g}^2,\\
        \label{DerivativePhi}
        \overline \nabla \phi & = u\overline \nabla f + f\overline \nabla u - \frac{T + R}{R^2}\overline \nabla \psi, \\
        \label{LaplacePhi}
        \overline \Delta \phi & = u\overline \Delta f + 2\langle \overline \nabla f, \overline \nabla u\rangle_{\overline g} + f \overline \Delta u- \frac{T + R}{R^2}\overline \Delta \psi,
    \end{align}
    where $\psi := d(x_0,\cdot)^2$. Combining \eqref{SimplifiedMaxima} with \eqref{LaplaceEta} and \eqref{LaplacePhi} we find that at $p$
    \begin{equation}
        \label{FromMaxPrinciple}
        Ce^{L\phi}
        \geq Le^{L\phi}\left(u\overline \Delta f + 2\langle \overline \nabla f, \overline \nabla u\rangle_{\overline g} + f \overline \Delta u- \frac{T + R}{R^2}\overline \Delta \psi\right) + L^2e^{L\phi}|\overline \nabla \phi|_{\overline g}^2.
    \end{equation}
    Next, we would like to estimate the terms in the parentheses in the right hand side of \eqref{FromMaxPrinciple}. First we note that by \eqref{eq:LaplacianDifferenceGraph}
    \begin{equation*}
        \overline \Delta \psi = \overline g^{ij} \Hess_{ij} \psi + \frac{u^2|df|_g^2\langle u^{-1}du,d\psi \rangle_g}{1 + u^2|df|_g^2} - \frac{uH\langle df,d\psi\rangle}{\sqrt{1 + u^2|df|_g^2}}.
    \end{equation*}
    The terms above can be bounded as follows
    \begin{align*}
        |\overline g^{ij} \Hess_{ij} \psi| & \leq C && \text{ by the Hessian comparison theorem} \\
        \left|\frac{u^2|df|_g^2\langle u^{-1}du,d\psi \rangle_g}{1 + u^2|df|_g^2}\right| & \leq \frac{u^2|df|_g^2}{1 + u^2|df|_g^2} \leq C && \text{ by $|d\psi|_g = 1$ and  \autoref{lem:WarpingBounds}} \\
        \left|\frac{uH\langle df,d\psi\rangle}{\sqrt{1 + u^2|df|_g^2}}\right| & \leq \frac{u|df|_g}{\sqrt{1 + u^2|df|_g^2}} \leq C && \text{ by $|d\psi|_g = 1$ and  \autoref{lem:WarpingBounds}}.
    \end{align*}
    Thus it follows that:
    \begin{equation*}
        \left|\frac{T + R}{R^2}\overline \Delta \psi\right| \leq C.
    \end{equation*}
    Further, by the definition of $\overline g$ we have
    \begin{equation*}
        2\langle \overline \nabla f, \overline \nabla u\rangle_{\overline g} 
        = \frac{2\langle df,du\rangle_g}{1 + u^2|df|_g^2}.
    \end{equation*}
    Using \eqref{eq:MeanCurvatureGraph} we find
    \begin{align*}
        u\overline g^{ij}\Hess_{i j}(f) 
        & = H\sqrt{1+u^2 |df|_g^2} - \overline g^{ij}(u_i f_j + f_i u_j + u^2 \langle du,df\rangle_g f_i f_j) \\
        & = H\sqrt{1+u^2 |df|_g^2} - \frac{(2 + u^2|df|_g^2)\langle du,df\rangle_g}{1 + u^2|df|_g^2},
    \end{align*}
    which combined with \eqref{eq:LaplacianDifferenceGraph} results in
    \begin{align*}
        u\overline \Delta f 
        & = H\sqrt{1+u^2 |df|_g^2} - \frac{(2 + u^2|df|_g^2)\langle du,df\rangle_g}{1 + u^2|df|_g^2} + \frac{u^2|df|_g^2\langle d u, d f\rangle_g}{1 + u^2|df|_g^2} - \frac{u^2|df|_g^2 H}{\sqrt{1 + u^2|df|_g^2}} \\
        & = \frac{H}{\sqrt{1 + u^2|df|_g^2}} - \frac{2\langle du,df\rangle_g }{1 + u^2|df|_g^2}.
    \end{align*}
    In the view of \eqref{eq:BoundedMeanCurvature} we have $\left|u\overline \Delta f + 2\langle \overline \nabla f, \overline \nabla u\rangle_{\overline g} \right| \leq |H| \leq C$. Finally, we need to bound the term
    \begin{equation*}
        f\overline \Delta u
        = f\Delta u -\frac{fu^2(\Hess u)(\nabla f, \nabla f)}{1 + u^2|df|_g^2} + \frac{fu|du|_g^2|df|_g^2}{1 + u^2|df|_g^2} -\frac{\langle df,du\rangle_g fuH}{\sqrt{1 + u^2|df|_g^2}}.
    \end{equation*}
    For this, we use \autoref{lem:WarpingBounds} which implies the bounds $|du|_g, |\nabla^2 u|_g \leq Cu$ for some constant $C$ that depends on $(M,g,K)$, $|u - \rho^{-1}|_{C^{2,0}_0(M)}$ and $\inf_M u$ and recall our assumption $|uf| \leq T$ on $B_R(x_0)$. Hence
    \begin{equation*}
        |f\Delta u| 
            \leq C, \quad
        \left|\frac{fu^2(\Hess u)(\nabla f, \nabla f)}{1 + u^2|df|_g^2}\right| 
            \leq C, \quad
        \frac{uf|du|_g^2|df|_g^2}{1 + u^2|df|_g^2} 
            \leq C, \quad
        \frac{\langle df,du\rangle_g ufH}{\sqrt{1 + u^2|df|_g^2}} 
            \leq C.
    \end{equation*}
    This implies: $|f\overline{\Delta} u| \leq C$. 

    In summary, choosing $L$ to be large enough with respect to $C$, \eqref{FromMaxPrinciple} we obtain
    \begin{equation}
        \label{eq:PhiBound}
        |\overline \nabla \phi|_{\overline g}^2 \leq \frac{C}{L}.
    \end{equation}
    On the other hand, by \eqref{DerivativePhi} we have
    \begin{equation*}
        |\overline \nabla \phi|_{\overline g}^2 
        = \left|udf + fdu - \frac{T + R}{R^2}d\psi\right|_g^2 - \frac{\big\langle udf,udf + fdu - \frac{T + R}{R^2}d\psi \big\rangle_g^2}{1 + u^2|df|_g^2}.
    \end{equation*}
    Defining $v := fdu - \frac{T + R}{R^2}d\psi$, by \autoref{lem:WarpingBounds} we have $|v|_g\leq C$, thus
    \begin{equation*}
        |\overline \nabla \phi|_{\overline g}^2 
        = |udf + v|_g^2 - \frac{\langle udf,udf + v\rangle_g^2 }{1 + u^2|df|_g^2}
        \geq |udf + v|_g^2 - \frac{u^2|df|_g^2|udf + v|_g^2}{1 + u^2|df|_g^2} 
    = \frac{|udf + v|_g^2}{1 + u^2|df|_g^2}.
    \end{equation*}
    Using the inequality $|X + Y|^2 \geq \frac{1}{2}|X|^2 - |Y|^2$ and the bound $|v|_g \leq C$ to get
    \begin{equation*}
        |\overline \nabla \phi|_{\overline g}^2 \geq \frac{u^2|df|_g^2 - C}{2(1 + u^2|df|_g^2)}. 
    \end{equation*}
    Choosing $L>4C$ in \eqref{eq:PhiBound} we get 
    \begin{equation*}
        \frac{u^2|df|_g^2 - C}{2(1 + u^2|df|_g^2)} \leq \frac{1}{4}. 
    \end{equation*}
    All in all, we conclude that $u|df|_g \leq C$ at a local maximum point $p$ of the function
    \begin{equation*}
        w^{-1/2}\eta = (e^{L\phi} - 1)(1 + u^2|df|_g^2)^{1/4}.
    \end{equation*}
    Noting that $\phi \leq 3T + 2R \leq C$ on $B_R(x_0)$ we get
    \begin{equation*}
        \sup_{x \in B_R(x_0)} w^{-1/2}\eta \leq C\cdot (e^{LC} - 1).
    \end{equation*}
    In particular, we have  
    \begin{equation*}
        w(x_0)^{-1/2}\eta(x_0) \leq C(e^{LC} - 1).    
    \end{equation*}
    If we now assume that $L \geq 2/R$ we have $\eta(x_0) = e^{LR} - 1 \geq 1$ which implies
    \begin{equation*}
        w^{-1/2}(x_0) \leq C(e^{LC} - 1)\eta(x_0)^{-1} \leq C(e^{LC} - 1).
    \end{equation*}
    After fixing any $L$ satisfying the above requirements, we have the desired result.
\end{proof}

We are now ready to construct a so called geometric solution to the generalized Jang equation, a complete embedded hypersurface $\Gamma \subset M\times_u \mathbb R$ having prescribed mean curvature $H_\Gamma = \tr_\Gamma\overline K$. We assume that $(u - \rho^{-1}) \in C^{3,\alpha}_0(M)$ and we choose an exhaustion of $M$ by precompact sets $\Omega_m$ such that $\partial \Omega_m = \{\rho = \rho_m\}$, where $\rho_m \to 0$. For every sufficiently small $\epsilon_m > 0$, \autoref{RegularizedBVP} implies the existence of a function $f_m \in C^{3,\alpha}(\overline \Omega_m)$ satisfying
\begin{align*}
    \begin{cases}
        \mathcal J(f_m) = \epsilon_m f_m, & \text{ in $\Omega_m$,} \\
        f_m = 0, & \text{ on $\partial \Omega_m$.}
    \end{cases}
\end{align*}
The next step is to obtain the limit of the graphs $\Gamma(f_m) \subset \widetilde M$ when $\epsilon_m \to 0$. In dimension $3$, one can derive uniform bounds on the second fundamental form, which in turn implies a precompactness result for $\Gamma(f_m)$. However, these estimates are not available in higher dimensions. Instead, we will use the \autoref{prop:LimitSubmanifolds} proven in \autoref{appendix_LimitManifold} by adapting the results of Eichmair \cite[Appendix A]{EichmairPlateau} to the warped product setting. As a result we obtain the following.\\

\begin{restatable}[Existence of geometric solutions]{theorem}{JangGeometric}
    \label{thm:JangGeometric}
    Suppose that $\alpha \in (0,1)$, $\tau > n/2$, that $(M,g,K)$ is a $C^{2,\alpha}_\tau$ asymptotically anti-de Sitter initial data set and that $(u - \rho^{-1})\in C^{2,\alpha}_0(M)$. Suppose further that $\{f_m\}_{m = 1}^\infty \subset C^{3,\alpha}_\text{loc}$ are solutions to the boundary value problem \eqref{RegularJang}-\eqref{RegularJangBoundary} with $\epsilon_m \to 0$ on precompact domains $\Omega_m$ such that $\bigcup_{m = 0}^\infty \Omega_m = M$. 
    
    Then there exists an open set $E \subset M \times_u \mathbb R$ and an oriented $C^{3,\alpha}$ submanifold $\Gamma = \partial E$ of $\widetilde M$ with the unit normal pointing out of $E$ denoted by $\nu$ such that each connected component of $\Gamma$ is either a graph or a cylinder and the mean curvature of $\Gamma$ computed as the tangential divergence of $\nu$ is given by
    \begin{align*}
        H = \tr_{\Gamma} \overline K,
    \end{align*}
    where $\overline{K}$ is the extension of $K$ as in \eqref{eq:ExtensionK}. Let $f: U \to \mathbb R$ be the function such that the graphical part of $\Gamma$ is $\Gamma(f)$. Then one of the connected components of $U$ is of the form $M\setminus K_0$ for some compact set $K_0 \subset M$ and the restriction of $f$ to this connected component satisfies 
    \begin{equation}
        \label{EndAsymptotics}
        f \in C^{3,\alpha}_{\tau + 1}(M\setminus K_0). 
    \end{equation}
    The boundary of $U$ decomposes as $\partial U = \partial^+ U \cup \partial^- U$ where $f(x) \to \infty$ (respectively $-\infty$) as $x \to \partial^+ U$ (respectively $\partial^- U$). Each connected component of $\partial^+ U$ and $\partial^- U$ is a $C^{3,\alpha}$ submanifold of $M$ and its mean curvature, $H^{\partial^\pm U}$, computed as the tangential divergence of the unit normal pointing out of $U$ satisfies:
    \begin{equation}
        \label{MeanCurvatureMOTS}
        H^{\partial^\pm U} = \pm \tr_{\partial^\pm U} K. 
    \end{equation}
   Finally, there is a $T > 0$ such that $\Gamma(f) \cap \{|t| > T\}$ can be written as a graph over $\partial^\pm U \times \mathbb R$. 
\end{restatable}

\begin{proof}
    The mean curvature of the graph $\Gamma(f_m)$ of $f_m: \Omega_m \to \mathbb R$ satisfying \eqref{RegularJang}-\eqref{RegularJangBoundary} equals
    \begin{equation}
        \label{JangMeanCurvatures}
        H_m
        = \epsilon_m f_m + \left(g^{ij} - \frac{u^2f_m^if_m^j}{1 + u^2|df_m|_g^2}\right) K_{ij} + \frac{\langle \nu_m, \nabla u\rangle}{u}\left(1 - \langle -\nu_m, u^{-1}\partial_t \rangle^2\right)
    \end{equation}
    Thus the hypothesis of \autoref{prop:LimitSubmanifolds} is satisfied with $F_m$ and $F$ given by
    \begin{align*}
        F_m(x,t,X) & = \tr_{\widetilde g}K - K(X,X) + \frac{\langle X, \nabla u\rangle}{u}\left(1 - \langle -X, u^{-1}\partial_t \rangle^2\right) + \epsilon_m t, \\
        F(x,t,X) & = \tr_{\widetilde g}K - K(X,X) + \frac{\langle X, \nabla u\rangle}{u}\left(1 - \langle -X, u^{-1}\partial_t \rangle^2\right).
    \end{align*}
    Thus \autoref{prop:LimitSubmanifolds} implies that there is a subsequence of $\{f_m\}_m$, denoted by the same notation, an open set $E \subset M\times_u \mathbb R$ whose boundary $\Gamma = \partial E$ is a $C^{3,\alpha}$ submanifold of $M\times_u \mathbb R$ such that the $\Gamma(f_m)$ converge locally to $\Gamma$ in $C_\text{loc}^{3,\alpha}$ as graphs over the tangent spaces of $\Gamma$. In addition, we have the convergence of the indicator functions
    \begin{equation*}
        \chi_{E_m} \to \chi_{E} \quad \text{in} \quad \text{BV}_{\text{loc}}(M\times_u \mathbb R).
    \end{equation*}
    Let $\nu$ be the unit normal of $\Gamma$ pointing out of $E$. Since the normals $\nu_m$ of the $\Gamma(f_m)$ are downward pointing we find that $\langle u^{-1}\partial_t,-\nu_m \rangle > 0$ for all $m$ and hence $w = \langle u^{-1}\partial_t, -\nu\rangle \geq 0$ on $\Gamma$. Due to the local $C^{3,\alpha}$ convergence of $\Gamma(f_m)$ to $\Gamma$ in \autoref{DifferentialInequalityGraphs}, we find that
    \begin{equation*}
        \Delta^{\overline g} w^{1/2} - Cw^{1/2} \leq 0,
    \end{equation*}
    on $\Gamma$ for some $C > 0$. By the strong maximum principle (cf. \cite[Theorem $3.5$]{GilbargTrudinger}) we conclude that either $w > 0$ or $w \equiv 0$ on each connected component of $\Gamma$. In the first case the component is graphical and in the second case it is cylindrical.

    Let $f: U \to \mathbb R$ be the graphing function of the graphical part of $\Gamma$. The barriers constructed in \Secref{sec_Barriers} satisfy $f_- < 0 < f_+$ and $\mathcal J(f_+) < 0 < \mathcal J(f_-)$, which implies $\mathcal J(f_+) < \epsilon_mf_+$ and $\mathcal J(f_-) > \epsilon_m f_-$. \autoref{lem:ComparisonJang} now implies that outside of a compact set $K_0$ we have: $f_- \leq f_m \leq f_+$ for all $m$. The estimates $|f_\pm| \leq C\rho^{\tau + 1}$ imply $|f_m| \leq C\rho^{\tau + 1}$, which then implies $|f| \leq C\rho^{\tau+1}$ on $M\setminus K_0$. \autoref{GradientFromOscillation} yields $u|df|_g \leq C$ hence $f$ satisfies a uniformly elliptic quasilinear PDE, $u^{-1}\mathcal J(f) = 0$, outside of $K_0$. 

    For any $B_R(x) \subset M\setminus K_0$, we can apply \cite[Theorem $13.6$]{GilbargTrudinger} (see also the proof of \autoref{prop:AprioriEstimates}) to conclude that $|f|_{C^{1,\gamma}(B_r(x))} \leq C$, where $r < R$ is a smaller radius depending only on $(M,g,K,R,C)$ and $\gamma \in (0,1)$. We now treat $u^{-1}\mathcal{J}(f) = 0$ as a linear uniformly elliptic PDE with coefficients bounded in $C^{0,\gamma}(B_r(x))$ and apply standard Schauder estimates \cite[Theorem $6.2$]{GilbargTrudinger} to obtain $|f|_{C^{2,\gamma}(B_r(x))} \leq C\rho^{\tau + 1}$. A standard bootstrap gives us the inequality: $|f|_{C^{k,\alpha}(B_r(x))} \leq C\rho^{\tau + 1}$, for all $B_r(x) \subset M\setminus K_0$ with $r < R$ and $C = C(M,g,K,R)$, establishing \eqref{EndAsymptotics}. 

    Since $\Gamma(f) = \partial E$ is boundaryless it follows that $|f(x)| \to \infty$ as $x\to \partial U$. Thus we have
    \begin{equation*}
        \partial U = \partial^+ U \cup \partial^- U \quad \text{ where } \quad \partial U^\pm := \{x_0 \in \partial U: \lim_{x\to x_0}f(x) = \pm \infty\}.
    \end{equation*}
    Clearly, for $t \in \mathbb R$, the mean curvature of the graph $\Gamma(f - t) \subset M \times_u \mathbb R$ is independent of $t$ and is given by
    \begin{align*}
        F(x,X) := \tr_{\widetilde g}K - K(X,X) + \frac{\langle X, \nabla u\rangle}{u}\left(1 - \langle -X, u^{-1}\partial_t \rangle^2\right).
    \end{align*}
    Letting $t\to \infty$ and applying \autoref{prop:LimitSubmanifolds} again, up to passing to a subsequence, we obtain a $C^{3,\alpha}$ surface $G = \partial W$ in the limit, where $W \subset M \times_u \mathbb R$ is open. The convergence of the indicator functions in \autoref{prop:LimitSubmanifolds} implies $W = \pi_M(W) \times \mathbb R$. By construction we must have $\partial^+ U = \partial (\pi_M(W))$ and hence $G = \partial^+ U \times \mathbb R$. The cylinder $G$ is $C^{3,\alpha}$ and hence so is $\partial^+ U$. Letting $t \to -\infty$, we conclude that $\partial^- U$ is $C^{3,\alpha}$ as well.

    We will now show that $\Gamma(f) \cap \{|t| > T\}$ is a graph over $\partial^\pm U \times \mathbb R$ for $T$ large enough. To this end, let $N$ be a $C^1$ unit normal vector field defined in a neighborhood $\Omega$ of $\partial^+ U \subset M$ and extend $N$ to all of $\Omega \times \mathbb R$ by requiring that $N_{(x,t)} = N_{(x,t+s)}$ for all $t,s \in \mathbb R$. Suppose on the contrary that for any $T > 0$, the set $\Gamma(f) \cap \{t > T\}$ cannot be written as a graph over $\partial^+ U \times \mathbb R$. In this case, there is a sequence of points $\{(x_l,f(x_l))\}\subset \Gamma$ such that
    \begin{align*}
        \lim_{l \to \infty} \text{dist}(x_l,\partial^+ U) = 0, \quad 
        \lim_{l \to \infty} f(x_l) = \infty, \quad 
        \langle \nu, N \rangle_g|_{(x_l,f(x_l))} = 0.
    \end{align*}
    Passing to a subsequence we can assume that $x_l \to x_0 \in \partial^+ U$ and that the submanifolds $\Gamma(f_l - f(x_l)) \cap \{-1 < t < 1\}$ $C^{3,\alpha}$ converge to $\partial^+ U \times (-1,1)$. Let $\nu^+$ denote the unit normal of $\partial^+ U$ pointing out of $W$, we then find
    \begin{align*}
        \pm 1 = \langle \nu^+,N\rangle|_{(x_0,0)} = \lim_{l \to \infty} \langle \nu, N \rangle|_{(x_l,f(x_l))} = \lim_{l \to \infty} 0 = 0,
    \end{align*}
    a contradiction. Repeating the same argument for $t \to -\infty$, we conclude that for some $T > 0$, $\Gamma(f) \cap \{t > |T|\}$ is a graph over $\partial^\pm U \times \mathbb R$. 

    Lastly we compute the mean curvature of $\partial^\pm U \subset M$, the tangential divergence of the normal $\nu$ pointing out of $U$. In the view of the above argument,  \autoref{prop:LimitSubmanifolds} implies $H^{\partial^\pm U \times \mathbb R} = \tr_{\partial^\pm U \times \mathbb R}\overline K$. At the same time, we note that \eqref{eq:PresecribedMeanCurvature} and \eqref{eq:WarpedChristoffels} yield
    \begin{align*}
        \tr_{\partial^\pm U\times \mathbb R}\overline K
        & = \tr_{\partial^\pm U\times \mathbb R} K + \frac{\langle \nu, \nabla u\rangle}{u}\left(1 - \langle -\nu, u^{-1}\partial_t \rangle^2\right) 
        = \tr_{\partial^\pm U} K + \frac{\langle \nu, \nabla u\rangle}{u}, \\
        H^{\partial^\pm U \times \mathbb R} - (\pm H^{\partial^\pm U})
        & = -g^{tt}\langle \nabla_{\partial_t} \nu, \partial_t\rangle 
        = u^{-2}\langle \nu, \nabla_{\partial_t}\partial_t \rangle
        = u^{-2}\langle\nu,u\nabla u\rangle 
        = \frac{\langle \nu,\nabla u\rangle}{u},
    \end{align*}
    and hence $\pm H^{\partial^\pm U} = \tr_{\partial^\pm U} K$.
\end{proof}

\begin{remark}
    \label{rem:MOTSMITS}
    In the proof of the above proposition we have not discussed the asymptotics of the asymptotically cylindrical ends. These asymptotics remain the same no matter which initial data we use, be it asymptotically hyperboloidal, asymptotically Euclidean or asymptotically anti-de Sitter. The reader is referred to Metzger \cite{MetzgerBlowUp}, Han and Khuri \cite{HanKhuri} and Yu \cite{YuBlowup} for results in this direction that all apply in our case, under some additional assumptions on the set $\partial^\pm U$ and/or the warping factor $u$ in a neighborhood of $\partial^\pm U$.  
\end{remark}


\section{A Jang equation reduction argument for asymptotically anti-de Sitter initial data sets}
\label{sec_MassChange}
In this section we discuss applications of the generalized Jang equation in the context of the positive mass theorem for asymptotically anti-de Sitter initial data sets. We start by recalling the definition of mass that is relevant here, see Michel \cite[Section $4.3$]{Michel} and Chruściel and Herzlich \cite{ChruscielHerzlich}.\\
\begin{definition}
    Suppose that $\alpha \in (0,1)$, $\tau > n/2$, that $(M,g)$ is a $C^{2,\alpha}_\tau$ asymptotically hyperbolic manifold (possibly with boundary) as in \autoref{def:AsympHypManifolds} such that
    \begin{equation}
        \label{eq:IntegrableScalarCurvature}
        \frac{\Scal^g + n(n-1)}{\rho} \in L^1(M,g).
    \end{equation}
    Letting $\{x^i\}$ be the restrictions of Cartesian coordinates of $\mathbb R^n$ to $S^{n-1}$ we define
    \begin{equation}
        \label{KernelFunctions}
        \mathcal N := \operatorname{span}\left\{V_0 := \frac{1-\rho}{\rho}, \quad V_1 := \frac{x^1}{\rho}, \quad \dots \quad V_n := \frac{x^n}{\rho}\right\}.
    \end{equation}
    Defining $e:= g - b$ and $c_n^{-1} = 2\omega_{n-1}(n-1)$, then for $V \in \mathcal N$ we have
    \begin{equation}
        \label{MassFunctionalDefinition}
        \M_{g}(V)
        = c_n\lim_{\rho \to 0}\int_{\{|x| = \rho\}} [V(\div^{b} e - d(\tr_{b} e)) - e(\nabla^{b} V,\cdot) + (\tr_{b} e)dV](\nu)d\mu^b,
    \end{equation} 
    is a well defined linear map called the \emph{mass functional} of $(M,g)$. 
\end{definition}
Recall that in \Secref{sec_GeometricSolution} we constructed a geometric solution to the Jang equation as a submanifold of the warped product $M \times_u \mathbb{R}$ with prescribed mean curvature. We also noted that this geometric solution has a $C^{3,\alpha}_\tau$ asymptotically hyperbolic component $\Gamma(f)$, which is given as the graph of a function $f \in C^{3,\alpha}_{\tau + 1}(U)$ defined outside of a compact subset of $M$.\\

\begin{theorem}[Mass change and Jang reduction]
    \label{MassChangeJang}
    Suppose that $\alpha \in (0,1)$, $\tau > n/2$, that $(M,g,K)$ is a $C^{2,\alpha}_{\tau}$ asymptotically anti-de Sitter initial data set with $g$ satisfying \eqref{eq:IntegrableScalarCurvature} and that $u - \rho^{-1} \in C^{2,\alpha}_0(M)$. Suppose further that $f \in C^{3,\alpha}_{\tau+1}(U)$ is a solution to the Jang equation on an open set $U \subset M$ such that $M\setminus U$ is compact and $\partial U$ is $C^{3,\alpha}_\text{loc}$. Then the mass functional of $(\Gamma(f),\overline g = g + u^2df^2)$ is well defined and coincides with the one of $(M,g)$. 
\end{theorem}
\begin{proof}
    The identity \eqref{eq:SchoenYauIdentity} implies
    \begin{equation*}
        \Scal^{\overline g} + n(n-1) = 2(\mu - J(w)) + |A - \overline K|_{\Gamma(f)}|_{\overline g}^2 + 2|q|_{\overline g}^2 - \frac{2}{u}\div_{\overline g}(uq).
    \end{equation*}
    That $(\Scal^g + n(n-1))\rho^{-1} \in L^1(M,g)$, $f \in C^{k+1,\alpha}_{\tau + 1}(U)$, $K\in C^{k-1,\alpha}_{\tau}(M)$ and the identities \eqref{eq:LocalEnergy}, \eqref{eq:LocalCurrent}, \eqref{eq:ExtensionK} and \eqref{eq:DefSchoenYau} combined with the above imply
    \begin{equation*}
        \frac{\Scal^{\overline g} + n(n-1)}{\rho} \in L^1(M,\overline g).
    \end{equation*}
    Moreover we note that $\overline g - b = u^2df^2 + g - b$ where $u^2df^2 \in C^{k,\alpha}_{2\tau}(U)$ and $g - b \in C^{k,\alpha}_{\tau}(U)$. It follows that $(\Gamma(f),g + u^2df^2)$ is a $C^{k,\alpha}_\tau$ asymptotically hyperbolic manifold satisfying \eqref{eq:IntegrableScalarCurvature} and so its mass functional is well defined. That the mass functionals of $(\Gamma(f),g + u^2df^2)$ and $(M,g)$ agree follows from a straightforward computation using \eqref{MassFunctionalDefinition}. 
\end{proof}
Next, we study how the mass functional changes under conformal changes of the metric. In what follows we let
\begin{equation}
    \label{eq:defKappa}
    \kappa := \frac{4}{n-2}.
\end{equation}.

\begin{proposition}
    \label{SimplifiedDifferenceConformalProp}
    Suppose that $\alpha \in (0,1)$, $\tau > n/2$, that $(M,g)$ is a $C^{2,\alpha}_\tau$ asymptotically hyperbolic manifold satisfying \eqref{eq:IntegrableScalarCurvature}. Let $\theta \in C^{k,\alpha}_\tau(M)$ and $\widehat g:= e^{\kappa \theta}g$ be such that
    \begin{equation*}
        (\Scal^{\widehat g} + n(n+1))\rho^{-1} \in L^1(M,\widehat g).
    \end{equation*}
    Then
    \begin{equation}
        \label{SimplifiedDifferenceConformal}
        \M_{\widehat g}(V) - \M_{ g}(V) = \frac{\kappa}{2\omega_{n-1}}\lim_{\rho \to 0}\int_{\{|x| = \rho\}} \bigg(\theta dV - V d\theta\bigg)(\nu) d\mu^b.
    \end{equation}
\end{proposition}

\begin{proof}
    Noting that
    \begin{equation*}
        \widehat e := \widehat g - b = \widehat g -  g +  g - b = (e^{\kappa\theta} - 1) g +  g - b,
    \end{equation*}
    we see that $(M,\widehat g)$ is a $C^{k,\alpha}_\tau$ asymptotically hyperbolic manifold as well and so the mass functional of $(M,\widehat g)$ is well defined. From \eqref{MassFunctionalDefinition} and the definition of $\widehat g$ it follows that
    \begin{equation}
        \label{DifferenceConformal}
        \M_{\widehat g}(V) - \M_{ g}(V) = c_n\lim_{\rho \to 0}\int_{\{|x| = \rho\}} [V(\div^b \widetilde e - d\tr^b \widetilde e) + (\tr^b \widetilde e) dV - \widetilde e(\nabla^b V,\cdot)](\nu) d\mu^b,
    \end{equation}
    for $V\in \mathcal N$ and $\widetilde e := (e^{\kappa\theta} - 1) g$. Recalling that $e =  g - b$, we can also check that
    \begin{align*}
        V(\div^b \widetilde e)(\nu) 
        &  = \kappa e^{\kappa\theta}V e(\nabla^b \theta,\nu) + \kappa e^{\kappa\theta}Vd\theta(\nu) + (e^{\kappa\theta} - 1)V\div^b  e(\nu), \\
        -V(d\tr^b \widetilde e)(\nu) 
        & = -V (n + \tr^b  e) \kappa e^{\kappa \theta} d\theta(\nu) - V(e^{\kappa \theta} - 1)(d\tr^b  e)(\nu), \\
        (\tr^b \widetilde e) dV(\nu)
        & = (e^{\kappa \theta} - 1)\tr^b(  e)dV(\nu) + n(e^{\kappa\theta} - 1)dV(\nu), \\
        -\widetilde e(\nabla^b V,\nu) 
        & = -(e^{\kappa\theta} - 1)  e(\nabla^b V,\nu) - (e^{\kappa\theta} - 1)dV(\nu).
    \end{align*}
    These identities together with \eqref{DifferenceConformal} imply
    \begin{align*}
        & \M_{\widehat g}(V) - \M_{g}(V) \\
        & = c_n\lim_{\rho \to 0}\int_{\{|x| = \rho\}} \bigg((e^{\kappa\theta} -1)[V(\div^b  e - d\tr^b  e) + (\tr^b  e) dV -  e(\nabla^b V,\cdot)](\nu) \\
        & \quad + \kappa e^{\kappa \theta}V [ e(\nabla^b \theta,\nu) - (\tr^b  e) d\theta(\nu)] + (n-1)[(e^{\kappa \theta} - 1)dV - \kappa Ve^{\kappa \theta}d\theta](\nu)\bigg)d\mu^b.
    \end{align*}
    Recalling that $|\theta| + |\nabla^b \theta|_b \leq C\rho^{\tau}$, $| e|_b + |\nabla^b  e|_b \leq C\rho^{\tau}$, we obtain
    \begin{equation*}
        \M_{\widehat g}(V) - \M_{ g}(V) 
        = (n-1)c_n\lim_{\rho \to 0}\int_{\{|x| = \rho\}} \bigg((e^{\kappa\theta} - 1)dV - \kappa Ve^{\kappa\theta} d\theta\bigg)(\nu) d\mu^b.
    \end{equation*}
    The estimates $e^{\kappa \theta} -1 = \kappa\theta + O(\rho^{2\tau}) = O(\rho^{\tau})$ and $d\theta(\nu) = O(\rho^\tau)$ now imply \eqref{SimplifiedDifferenceConformal}.
\end{proof}
We conclude this section with a theorem that highlights a potential application of the results of this paper in the context of the positive mass conjecture for asymptotically anti-de Sitter initial data sets.\\

\begin{restatable}{theorem}{ChangeOfMassFinal}
\label{ChangeOfMassFinal}
    Suppose that $\alpha \in (0,1)$, $\tau > n/2$, that $(M,g,K)$ is a complete $C^{2,\alpha}_{\tau}$ asymptotically anti-de Sitter initial data set satisfying the dominant energy condition and that $(M,g)$ satisfies \eqref{eq:IntegrableScalarCurvature}. Suppose furthermore that there exists functions $f \in C^{3,\alpha}_{\tau + 1}(M)$ and $u \in C^{2,\alpha}_{\text{loc}}(M)$ such that $u$ is positive, $u - V_0 \in C^{2,\alpha}_{\tau - 1}(M)$ and $(f,u)$ is an entire solution (i.e. defined on all of $M$) to the coupled system
    \begin{equation}
        \label{eq:KeySystem}
        \begin{cases}
            \mathcal J(f) & = 0, \\
            \Delta^{\overline g}u & = nu,
        \end{cases}
    \end{equation}
    where $\overline g = g + u^2df^2$. Then the energy momentum vector of $(M,g,K)$ is future causal,
    \begin{equation*}
        \mathcal M_g(V_0) \geq \sqrt{\sum_{i = 1}^n \mathcal M_g(V_i)^2}.
    \end{equation*}
    If $\mathcal M_g(V_0) = 0$, then $(M,g)$ embeds as a spacelike graphical slice into the anti-de Sitter spacetime with second fundamental form $K$. 
\end{restatable}

\begin{proof}
    Suppose the assumptions of the theorem are satisfied. By Gicquaud \cite[Theorem $3.3$]{Gicquaud} there is a function $\theta \in C^{k,\alpha}_\tau(M)$ satisfying the Yamabe equation
    \begin{equation}
        \label{Yamabe}
        -\kappa(n-1)\left(|d\theta|_{\overline g}^2 + \Delta^{\overline g} \theta\right) + \Scal^{\overline g} =  -n(n-1)e^{\kappa \theta},
    \end{equation}
    so that $(M,\widehat g = e^{\kappa \theta}\overline g)$ is a $C^{k,\alpha}_{\tau}$ asymptotically hyperbolic manifold with $\Scal^{\widehat g} \equiv -n(n-1)$. Thus $(M,\widehat g)$ has a well defined mass functional. \autoref{SimplifiedDifferenceConformalProp} applied with $\overline g$ in the place of $g$ combined with the estimates $|\theta dV_0 - V_0 d\theta|_b = O(\rho^{\tau - 1})$ and $|d\mu^b - d\mu^{\overline g}| = O(\rho^{\tau})$ implies:
    \begin{equation*}
        \delta \mathcal M:= \frac{2\omega_{n-1}}{\kappa}[\M_{\widehat g}(V_0) - \M_{\overline g}(V_0)] = \lim_{\rho \to 0}\int_{\{|x| = \rho\}} \bigg(\theta dV_0 - V_0 d\theta\bigg)(\overline \nu) d\mu^{\overline g}.
    \end{equation*}
    Let $\eta \in C^{k,\alpha}_\text{loc}(M)$ be any function such that $\eta - 1 \in C^{k,\alpha}_\tau(M)$. Since $u - V_0 \in C^{k,\alpha}_{\tau-1}(M)$ and $\theta \in C^{k,\alpha}_\tau(M)$ we have
    \begin{align*}
        \theta dV_0 - V_0d\theta 
        & = \theta \underbrace{d(V_0 - u)}_{\in C^{k-1,\alpha}_{\tau-1}(M)} - \underbrace{(V_0 - u)}_{\in C^{k,\alpha}_{\tau-1}(M)}\theta + (1-\eta^2) \cdot \underbrace{(\theta du - ud\theta)}_{\in C^{k-1,\alpha}_{\tau-1}(M)} + \eta^2(\theta du - ud\theta) \\
        & = (\eta \theta)d(\eta u) - (\eta u)d(\eta \theta) + O(\rho^{2\tau-1}).
    \end{align*}
    As a consequence, using the divergence theorem, we can rewrite $\delta \mathcal M$ as
    \begin{align*}
        \delta \mathcal M 
        & = \lim_{\rho \to 0}\int_{\{|x| = \rho\}} [(\eta\theta) d(\eta u) - (\eta u)d(\eta \theta) ](\overline \nu) d\mu^{\overline g}
        = \int_{M} \eta \theta \Delta^{\overline g} (\eta u) - \eta u \Delta^{\overline g} (\eta \theta) dV^{\overline g} \\
        & = \int_{M} \bigg(\eta^2 (\theta \Delta^{\overline g} u - u \Delta^{\overline g} \theta) + \langle d (\eta^2), \theta d u - u d\theta \rangle\bigg) dV^{\overline g}.
    \end{align*}
    Next we note the following inequality which follows from the generalized Schoen and Yau identity and the dominant energy condition, see \eqref{eq:SchoenYauIdentity} and \autoref{def:DEC}:
    \begin{equation*}
        -\Scal^{\overline g} \leq n(n-1) - 2|q|_{\overline g}^2 + \frac{2}{u}\div_{\overline g}(uq).
    \end{equation*}
    Combining the above with \eqref{Yamabe} we obtain
    \begin{equation}
        \label{DeltaTheta}
        -\Delta^{\overline g} \theta = |d \theta|_{\overline g}^2 - \frac{ne^{\kappa \theta}}{\kappa} - \frac{\Scal^{\overline g}}{\kappa(n-1)}
        \leq |d \theta|_{\overline g}^2 -\frac{n}{\kappa}(e^{\kappa \theta} - 1) - \frac{2\left(|q|_{\overline g}^2 - \frac{1}{u}\div_{\overline g}(uq)\right)}{\kappa (n-1)}.
    \end{equation}
    In the view of the assumption $\Delta^{\overline g} u = nu$ and the inequality \eqref{DeltaTheta} we arrive at the estimate
    \begin{align*}
        \delta \mathcal M
        & \leq \int_M n\eta^2u\left(\theta - \frac{e^{\kappa \theta} - 1}{\kappa}\right) dV^{\overline g} - \frac{2}{\kappa(n-1)} \int_M\eta^2 (u|q|_{\overline g}^2 -\div_{\overline g}(uq)) dV^{\overline g} \\
        & \quad + \int_M \bigg(\eta^2 u|d \theta|_{\overline g}^2 + \langle d (\eta^2), \theta du - u d\theta \rangle_{\overline g}\bigg) dV^{\overline g}.
    \end{align*}
    Since $1 + x \leq e^x$ for all $x \in \mathbb{R}$, the first integral above is non-positive. Furthermore, the divergence theorem implies
    \begin{equation*}
        \int_M \eta^2 \div_{\overline g}(uq) dV^{\overline g} 
        = - \int_M u q(\overline \nabla \eta^2) dV^{\overline g}.
    \end{equation*}
    Note that there are no boundary terms in the above formula since \eqref{eq:DefSchoenYau},  \eqref{eq:SecFundFormGraph} and the assumptions $u - \rho^{-1} \in C^{k,\alpha}_0(M)$, $f \in C^{k+1,\alpha}_{\tau + 1}(M)$ and $K \in C^{k-1,\alpha}_\tau(M)$ together imply  $|u\eta^2q(\overline \nu)| = O(\rho^{2\tau-1}) = o(\rho^{n-1})$. In conclusion, we have the upper bound
    \begin{equation*}
        \delta \mathcal M
        \leq \int_M \bigg(\eta^2 u|d \theta|_{\overline g}^2 + \langle d(\eta^2), \theta du - u d\theta \rangle_{\overline g}\bigg) - \frac{2}{\kappa (n-1)}\bigg(\eta^2u |q|_{\overline g}^2 + uq(\overline \nabla \eta^2)\bigg)dV^{\overline g}.
    \end{equation*}
    We now let $\eta = e^{2\theta}$, then
    \begin{equation}
        \label{eq:ChosenEta}
        \delta \mathcal M 
        \leq 
        \int_M\bigg(ue^{4\theta}|d \theta|_{\overline g}^2 + 4e^{4\theta}\langle d\theta, \theta du - ud\theta \rangle_{\overline g}\bigg) - \frac{2ue^{4\theta}}{\kappa(n-1)}\bigg(|q|_{\overline g}^2 + 4q(\overline \nabla \theta)\bigg)dV^{\overline g}.
    \end{equation}
    In turn, the Cauchy-Schwarz inequality implies: $|4q(\overline \nabla \theta)| \leq |q|_{\overline g}^2 + 4|d\theta|_{\overline g}^2$, hence
    \begin{equation}
        \label{BoundI2}
        \int_M \frac{-2ue^{4\theta}}{\kappa(n-1)}\bigg(u |q|_{\overline g}^2 + 4q(\overline \nabla \theta)\bigg)dV^{\overline g}
        \leq \int_M \frac{2(n-2)}{n-1}ue^{4\theta}|d\theta|_{\overline g}^2 dV^{\overline g}.
    \end{equation}
    Moreover, we have
    \begin{equation}
        \label{eq:GradientTerms}
        4e^{4\theta}\langle d\theta, \theta du - ud\theta \rangle_{\overline g} 
        = 4\theta e^{4\theta} \langle d\theta, du \rangle_{\overline g} - 4u e^{4\theta}|d\theta|_{\overline g}^2.
    \end{equation}
    Combining \eqref{eq:ChosenEta}, \eqref{BoundI2} and \eqref{eq:GradientTerms} we obtain
    \begin{equation}
        \label{eq:UsedInRigidity}
        \delta \mathcal M \leq \int_M ue^{4\theta}|d\theta|_{\overline g}^2\underbrace{\left(1 - 4 + \frac{2(n-2)}{(n-1)}\right)}_{<0} dV^{\overline g} + \int_M 4\theta e^{4\theta} \langle d\theta, du \rangle_{\overline g} dV^{\overline g}.
    \end{equation}
    The first integral on the right hand side is non-positive. As for the second integral, we note that the function $\gamma(x) := xe^{4x} - e^{4x}/4 + 1/4$ satisfies
    \begin{equation*}
    \lim_{x \to -\infty} \gamma(x) = \frac{1}{4}, \quad \lim_{x \to \infty} \gamma(x) = \infty, \quad \gamma'(x) = 4xe^{4x}, \quad \gamma(0) = 0.
    \end{equation*}
    It is easy to check that this implies that $\gamma(x) \geq 0$ for all $x \in \mathbb R$ and that $\gamma(x) = O(x^2)$ for small $x$. Hence $\gamma(\theta) = O(\rho^{2\tau})$ and consequently, the divergence theorem implies
    \begin{equation}
        \label{eq:DummyRHS}
        \begin{aligned}
        \int_M 4\theta e^{4\theta}\langle d\theta, du\rangle_{\overline g} dV^{\overline g}
        & = \int_M\langle d(\gamma(\theta)), du\rangle_{\overline g} dV^{\overline g}
        - \int_M \gamma(\theta) (\Delta^{\overline g} u) dV^{\overline g} \\
        & = -n\int_M u\gamma (\theta) dV^{\overline g} \leq 0.
        \end{aligned}
    \end{equation}
    Again, there are no boundary terms appearing when applying the divergence theorem due to the estimate $|\gamma(\theta)du|_{\overline g} = O(\rho^{2\tau-1}) = o(\rho^{n-1})$. It follows that $\delta \mathcal M \leq 0$ and hence $\mathcal M_{\widehat g}(V_0) \leq \mathcal M_{\overline g}(V_0) = \mathcal M_{g}(V_0)$. An application of the standard positive mass theorem for asymptotically hyperbolic manifolds, see e.g. \cite[Theorem $1$]{ChruscielGalloway}, to $(M,\widehat g)$ gives $\mathcal M_{\widehat g}(V_0) \geq 0$ and hence $\mathcal M_g(V_0) \geq 0$ as well. 

    Next suppose that $E := \mathcal M_g(V_0)^2 < \sum_{i = 1}^n \mathcal M_g(V_i)^2$ and define the vector
    \begin{equation*}
        P := (\mathcal M_g(V_1), \dots, \mathcal M_g(V_n)),
    \end{equation*}
    so that $E^2 < |P|^2$. Since every Lorentz boost (see \cite[Section 2]{CortierDahlGicquaud} for details) $q: \mathbb H^n \to \mathbb H^n$ is an isometry, we find that for any chart at infinity $\Psi: M \setminus K_0 \to \mathbb H^n\setminus \overline B_{R_0}$, the map $\Psi' = q \circ \Psi$ is another a chart at infinity, after replacing $K_0$ with a larger compact set if necessary. If $q$ is a Lorentz boost by the angle $\eta$ in the direction $P$, then since the mass functional changes equivariantly under isometries of the background manifold $(\mathbb H^n,b)$ (see e.g. \cite{Michel}), we have
    \begin{equation*}
        \mathcal M'_g(V_0) = \frac{E - \eta |P|}{\sqrt{1 - \eta^2}},
    \end{equation*}
    where $\mathcal M'_g$ is the mass functional with respect to the chart $\Psi'$. In particular, choosing $\eta \in (|E|/|P|,1)$ gives $\mathcal M'_g(V_0) < 0$, a contradiction. Thus we have $\mathcal M_g(V_0)^2 \geq \sum_{i = 1}^n \mathcal M_g(V_i)^2$ as desired.

    Lastly suppose that $\mathcal M_g(V_0) = 0$, in which case we must have $\delta \mathcal M = 0$. By the calculation in $\eqref{eq:DummyRHS}$, the second term in the right hand side of \eqref{eq:UsedInRigidity} is non-positive and so
    \begin{equation*}
        0 = \delta \mathcal M \leq \int_M ue^{4\theta}|d\theta|_{\overline g}^2\underbrace{\left(1 - 4 + \frac{2(n-2)}{(n-1)}\right)}_{ < 0} dV^{\overline g} \leq 0.
    \end{equation*}
    This combined with $\theta \in C^{k,\alpha}_\tau(M)$ implies that $d\theta \equiv 0$ so that $\theta \equiv 0$ and $\Scal^{\overline g} \equiv -n(n-1)$, which combined with \cite[Theorem $3$]{HuangLeeMartin} implies that $(M,\overline g)$ is isometric to the hyperbolic space $(\mathbb H^n, b)$.

    The equality $\Scal^{\overline g} \equiv -n(n-1)$ combined with the generalized Schoen-Yau identity (see \eqref{eq:SchoenYauIdentity}) and the dominant energy condition implies
    \begin{equation}
    \label{eq:FromSYidentity}
     u|A - \overline K|_{\overline g}^2 \leq 2u(\mu - J(w)) + u|A - \overline K|_{\overline g}^2 = - 2u|q|_{\overline g}^2 + 2\div_{\overline g}(uq).
    \end{equation}
    Integrating over $\{|x| \leq \rho\}$ and letting $\rho \to 0$, we may argue as before to find
    \begin{align*}
        0\leq \int_M u|A - \overline K|_{\overline g}^2 dV^{\overline g} 
        &\leq \int_M - u|q|_{\overline g}^2 + \div_{\overline g}(uq) dV^{\overline g} \leq \int_{M} \div_{\overline g}(uq) dV^{\overline g} \\
        & = \lim_{\rho \to 0} \int_{\{|x| = \rho\}} uq(\overline \nu) d\mu^{\overline g} = 0.
    \end{align*}
    Since $u > 0$ it follows that $A = \overline K$. The calculations in \cite[Appendix B]{BrayKhuri} apply to any warped product, which implies that $(M,g)$ embeds into the warped product space-time $(\mathbb R \times \mathbb H^n, - u^2dt^2 + b)$ via the map $x \mapsto (f(x),x)$ with second fundamental form $K$. Since $u$ satisfies $\Delta^b u = nu$ and the asymptotic condition $u - V_0 \in C^{k,\alpha}_\tau(M)$ it follows that $u = V_0$ and so $(\mathbb R \times \mathbb H^n,- u^2dt^2 + b)$ is the anti-de Sitter spacetime. In summary, $(M,g)$ can be embedded as a space-like graphical slice of anti de-Sitter space-time with second fundamental form $K$. 
\end{proof}

\begin{remark}
    We do not expect that every asymptotically anti-de Sitter initial data set $(M,g,K)$ admits a globally defined solution $(f,u)$ of the system \eqref{eq:KeySystem}, as solutions of the generalized Jang equation may blow-up on marginally outer and inner trapped surfaces, see \Secref{sec_GeometricSolution}. Presumably, this issue can be dealt with by using the results Han and Khuri \cite{HanKhuri} and Yu \cite{YuBlowup} regarding the asymptotics of solutions near the boundary of their blow up sets (see also \autoref{rem:MOTSMITS}). We note that one is confronted with the same problem in the reduction argument proposed by Cha and Khuri in \cite{ChaKhuri18}. 
\end{remark}

\appendix


\section{Boundary gradient estimates.}
\label{appendix_appendixA}
\begin{proposition}
    \label{prop:BoundaryGrad}
    Let $U$ be a neighborhood of $\partial \Omega$, let  $\overline f, \underline f \in C^2(U\cap \Omega) \cap C^1(U \cap \overline \Omega)$ be such that
    \begin{align*}
        H_{g}(\overline f) - s\tr_{g}(K)(\overline f) - \epsilon \overline f & < 0, \\
        H_{g}(\underline f) - s\tr_{g}(K)(\underline f) - \epsilon\underline f & > 0,
    \end{align*}
    in $U\cap \Omega$ and $\underline f = \overline f = s\phi$ on $\partial \Omega$. If $f$ is a solution to \eqref{RegularJangS}-\eqref{RegularJangSBoundary} and satisfies
    \begin{equation*}
        \underline f \leq f \leq \overline f \quad \text{ on } \quad\partial U \cap \Omega,
    \end{equation*}
    then we have the bound
    \begin{equation*}
        \sup_{\partial \Omega}|df|_g \leq \max(|d\underline f|_g, |d\overline f|_g).
    \end{equation*}
\end{proposition}

\begin{proof}
    Since $\overline{\Omega \cap U}$ is compact, $\overline f - f$ attains its minimum at some point $p$ in $\overline{\Omega \cap U}$. If $p \in \partial (\Omega \cap U) = \partial \Omega \cup (\partial U \cap \Omega)$, then in the view of
    \begin{equation*}
        f = \overline f, \text{ on $\partial \Omega$} \quad \text{ and } \quad
        f \leq \overline f, \text{ on $\partial U \cap \Omega$},
    \end{equation*}
    it follows that $f(p) \leq \overline f(p)$. If $p$ lies in the interior of $\Omega \cap U$, then at $p$
    \begin{equation*}
        \nabla(\overline f - f) = 0, \quad \Hess(\overline f - f) \text{ is non-negative definite at $p$}.
    \end{equation*}
    Hence
    \begin{align*}
        0 & > H_g(\overline f) - s\tr_g(K)(\overline f) - \epsilon \overline f - H_g(f) + s\tr_g(K)(f) + \epsilon f \\
        & = \underbrace{\left(g^{ij} - \frac{u^2f^if^j}{1 + u^2|df|_g^2}\right)\left(\frac{u\Hess_{ij} \overline f - u\Hess_{ij}f}{\sqrt{1 + u^2|df|_g^2}}\right)}_{\geq 0} + \underbrace{s\left(g^{ij} - \frac{u^2f^if^j}{1 + u^2|df|_g^2}\right)(K_{ij} - K_{ij})}_{ = 0} \\
        & + \epsilon(f - \overline f) \geq \epsilon(f - \overline f),
    \end{align*}
    at $p$. It follows that: $\overline f(p) \geq f(p)$. Hence $\overline f \geq f$ in $U \cap \Omega$ and by a similar argument $\underline f \leq f$ in $U \cap \Omega$. Since $\underline f(p_0) = f(p_0) = \overline f(p_0)$ for all $p_0 \in \partial \Omega$ it follows that for all $p \in U \cap \Omega$
    \begin{equation*}
        \frac{\underline f(p) - \underline f(p_0)}{d(p_0,p)} \leq \frac{f(p) - f(p_0)}{d(p_0,p)} \leq \frac{\overline f(p) - \overline f(p_0)}{d(p_0,p)}.
    \end{equation*}
    Passing to local coordinates and comparing partial derivatives we get the desired bound. 
\end{proof}

\newpage
\section{Results on convergence of graphs of bounded mean curvature}
\label{appendix_LimitManifold}
The aim of this appendix it to prove the following proposition.\\

\begin{restatable}[Limits of graphs in warped product spaces]{proposition}{LimitsSubmanifolds}
    \label{prop:LimitSubmanifolds}
    Let $(M^n,g)$ be a Riemannian manifold with $g \in C^{2,\alpha}_\text{loc}(M)$, $3 \leq n \leq 7$, $\alpha \in (0,1)$ and let $\widetilde M := (M\times \mathbb R, \widetilde g := g + u^2dt^2)$ be a warped product space where $u \in C_\text{loc}^{2,\alpha}(M)$ is positive. 

    Suppose that $\Omega \subset M$ is open, $f_m: \Omega \to \mathbb R$ is a sequence of functions in $C_\text{loc}^{3,\alpha}(\Omega)$ such that the mean curvature of each graph $\Gamma(f_m) := \{(x,f_m(x)), x \in \Omega\}$ computed as the tangential divergence of the downward pointing unit normal $\nu_m$ is equal to $H_{\Gamma(f_m)}(x,t) = F_m(x,t,\nu_m)$. Moreover, suppose that the functions $F_m: \widetilde M \times T\widetilde M \to \mathbb R$ are in $C^{1,\alpha}_{\text{loc}}(\widetilde M \times T\widetilde M)$ and satisfy $|F_m|_{C^1(\widetilde M \times T\widetilde M)} \leq C$, for some constant $C > 0$ and that there is a function $F \in C^{1,\alpha}_\text{loc}(\widetilde M \times T\widetilde M)$ such that $F_m \to F$ in $ C^{1,\alpha}_{loc}(\widetilde M \times T\widetilde M)$ as $m \to \infty$.
    
    Then there exists a subsequence of $\{f_m\}_{m = 1}^\infty$, denoted using the same subscript $m$, and a $C^{3,\alpha}$ submanifold $\Gamma \subset \widetilde M$ given by $\Gamma = \partial E$ for an open set $E \subset \Omega \times \mathbb R$, such that the manifolds $\Gamma(f_m)$ converge to $\Gamma$ in $C_\text{loc}^{3,\alpha}$ as graphs over the tangent spaces of $\Gamma$. Letting $E_m := \{(x,y): x \in \Omega, y > f_m(x)\}$ we have $\chi_{E_m} \to \chi_{E}$ in $\text{BV}_{\text{loc}}(\Omega \times \mathbb R)$. Moreover, the mean curvature of $\Gamma$ with respect to the normal pointing out of $E$ is equal to $H_\Gamma(p) = F(p,\nu)$. 
\end{restatable} 

We will prove \autoref{prop:LimitSubmanifolds} using methods from geometric measure theory. The reader is referred to Simon \cite{SimonGMT} for the general theory of currents, varifolds and sets of locally finite perimeter. In particular, we will be making use of the compactness and regularity theory of so called $C$-almost minimizing currents as presented by Eichmair in
\cite[Appendix A]{EichmairPlateau}, see also Duzaar and Steffen \cite{DuzaarSteffen}. Adapting \cite[Definition A.$1$]{EichmairPlateau} to our setting, we make the following definition, using the notations of \cite[Chapter $7$, Section $5$]{SimonGMT}, see also \cite[Section $37$]{SimonGMTOld}.\\

\begin{definition}
    \label{CalmostMinimizer}
    Suppose that $n \geq 1$ and $l \geq 1$ are integers, $U \subset \mathbb R^{n+l}$ is an open set, $N \subset \mathbb R^{n+l}$ is an oriented and embedded $(n+1)$-dimensional $C^2$ submanifold of $\mathbb R^{n+l}$ and $T \in D_n(U)$ is an integer multiplicity current with $\text{spt}(T) \subset N \cap U$. Assume further that $T = \partial \LB E\RB$ for some $\mathcal H^{n+1}$ measurable subset $E \subset N \cap U$ of locally finite perimeter in $N$. 

    We say that $T$ is \emph{$C$-almost minimizing in $N \cap U$} if for any open set $W \ssubset U$ and any integer multiplicity current $X \in D_{n+1}(U)$ compactly supported in $N \cap W$ we have:
    \begin{equation}
        \label{CalmostIneq}
        M_W(T) \leq M_W(T + \partial X) + CM_W(X).
    \end{equation}
\end{definition} 

In the next lemma we show that a graph $\Gamma = \{(x,f(x)), x \in \Omega\} \subset (M \times \mathbb R, h)$ of bounded mean curvature $|H_\Gamma| \leq C$ has the $C$-almost minimizing property whenever the coordinate vector field $\partial_t$ along the $\mathbb R$ factor is a Killing vector field of the metric $h$. This holds in particular for the aforementioned warped product metric $h = g + u^2dt^2$, but also for metrics of the form:
\begin{equation*}
    h = g + w \otimes dt + dt \otimes w + u^2 dt^2,
\end{equation*}
where $u$ is a function and $w$ is a $1$-form on $M$, that are relevant in the context of reduction arguments for the mass-angular momentum and the mass-angular momentum-charge inequalities, see Cha and Khuri \cite{ChaKhuri15} and \cite{ChaKhuri18} and Cha, Khuri and Sakovich \cite{ChaKhuriSakovich}.\\

\begin{lemma}
    \label{lem:CMinExistence}
    Let $(M\times \mathbb R,h)$ be an oriented Riemannian manifold with $h \in C^{2}_\text{loc}(M\times \mathbb R)$ and let $t$ be the coordinate on $\mathbb R$. Suppose that the vector field $\partial_t$ is a Killing vector field for the metric $h$, that $\Omega \subset M$ is an open set and that $f: \Omega \to \mathbb R$ is a $C^2$ function whose graph $\Gamma(f) = \{(x,f(x)): x \in \Omega\} \subset M\times \mathbb R$ has uniformly bounded mean curvature $|H_{\Gamma(f)}| \leq C$.

    For any $C^2$ isometry $\psi: (M \times \mathbb R,h) \hookrightarrow (\mathbb R^{n+l},\delta)$ and any open set $U \subset \mathbb R^{n+l}$ such that $\psi(\Omega \times \mathbb R) = \psi(M \times \mathbb R) \cap U$, the current $\LB\psi(\Gamma(f))\RB \in D_n(\mathbb R^{n+l})$ is $C$-almost minimizing in $\psi(M \times \mathbb R) \cap U$.
\end{lemma}

\begin{proof}
    We begin by defining a one parameter family of maps $\varphi: (M \times \mathbb R) \times \mathbb R \to M \times \mathbb R$ by
    \begin{equation*}
        \varphi(x,t,s) = (x,t+s) \quad \text{for $x \in M$ and $t,s \in \mathbb R$}.
    \end{equation*}
    The map $\varphi$ is the flow generated by the vector field $\partial_t$ and since $\partial_t$ is a Killing vector field of $h$ we conclude that each map $\varphi_s(x,t) = \varphi(x,t,s)$ is an isometry of $(M\times \mathbb R,h)$. Letting $\nu$ be the downward pointing unit normal to $\Gamma(f) \subset \Omega \times \mathbb R$, we extend it to vector field on $\Omega \times \mathbb R$ defined by
    \begin{equation}
        \label{eq:extensionNu}
        \nu(x,t+s) = d\varphi_s(\nu(x,t)) \quad \text{for all $x \in \Omega$ and $t,s \in \mathbb R$}. 
    \end{equation}
    This extension implies that $\nu(x,f(x) + s)$ is the downward pointing unit normal of the vertically shifted graph $\Gamma(f + s)$. To see this, let $\gamma: (-\epsilon, \epsilon) \to M \times \mathbb R$ be a curve satisfying $\gamma(\tau) \in \Gamma(f+s)$ for all $\tau \in (-\epsilon, \epsilon)$. We want to show that $\langle \gamma'(0),\nu \rangle_h = 0$. By definition we have
    \begin{equation*}
        \varphi_s^{-1}(\gamma(\tau)) \in \Gamma(f) \quad \text{for all $\tau \in (-\epsilon,\epsilon)$}. 
    \end{equation*}
    Moreover, we have
    \begin{align*}
        \langle \gamma'(0), \nu(x,f(x) + s) \rangle_h 
        & = \langle d\varphi_s^{-1}(\gamma'(0)), d\varphi_s^{-1}(\nu(x,f(x) + s)) \rangle_h \\
        & = \langle (\varphi_s^{-1} \circ \gamma)'(0), \nu(x,f(x)) \rangle_{h} \\
        & = 0,
    \end{align*}
    where we used the fact that $\varphi_s$ is an isometry in the first line and \eqref{eq:extensionNu} and the chain rule in the second line. The last line follows from the fact that $\nu$ is normal to $\Gamma(f)$ and that $(\varphi_s^{-1} \circ \gamma)'(0) \in T_{(x,f(x))}\Gamma(f)$. A similar computation shows that
    \begin{equation*}
        \langle \partial_t, \nu\rangle_h\bigg|_{(x,f(x) + s)} = \langle \partial_t, \nu\rangle_h\bigg|_{(x,f(x))} < 0 \quad \text{for all $x\in \Omega$ and $s \in \mathbb R$},
    \end{equation*}
    which shows that $\nu(x,t)$ always points downwards. Furthermore, we note that the mean curvature of each shifted graph $\Gamma(f + s)$ is the same as that of $\Gamma(f)$, namely
    \begin{equation*}
        H_{\Gamma(f + s)}(x,f(x)+s) = H_{\Gamma(f)}(x,f(x)) \quad \text{for all $x\in \Omega$ and $s \in \mathbb R$}. 
    \end{equation*}
    To see this, fix $x \in \Omega$ and let $\{e_1,\dots, e_n\}$ be an orthonormal frame of $\Gamma(f)$ at $(x,f(x)) \in \Gamma(f)$. We extend each $e_i$ to $\{x\} \times \mathbb R$ by setting
    \begin{equation}
        \label{eq:ExtensionFrame}
        e_i(x,t+s) = d\varphi_s(e_i(x,t)) \quad \text{for all $t,s \in \mathbb R$ and $i \in \{1,\dots, n\}$}
    \end{equation}
    and we set $e_0 := \nu$. Denoting $p := (x,f(x))$ and $p_s := (x,f(x) + s)$ in order to simplify notation, we obtain
    \begin{align*}
        H_{\Gamma(f+s)}(p_s)
        & = \sum_{i = 1}^n \langle e_i(p_s),(\nabla^h_{e_i(p_s)}\nu)(p_s) \rangle_h && \text{by definition}\\
        & = \sum_{i = 1}^n \langle d\varphi_s(e_i(p)),(\nabla^h_{d\varphi_s(e_i(p))}d\varphi_s(\nu))(p_s) \rangle_h && \text{by \eqref{eq:extensionNu} and \eqref{eq:ExtensionFrame}} \\
        & = \sum_{i = 1}^n \langle d\varphi_s(e_i(p)),d\varphi_s(\nabla^h_{e_i(p)}\nu)(p_s) \rangle_h && \parbox{15em}{since $\varphi_s$ is an isometry}\\
        & = \sum_{i = 1}^n \langle e_i(p),(\nabla^{h}_{e_i(p)}\nu)(p) \rangle_h && \parbox{15em}{since $\varphi_s$ is an isometry} \\
        & = H_{\Gamma(f)}(p) && \text{by definition.}
    \end{align*}
    We can thus extend $H_{\Gamma(f)}$ to a function $H: \Omega \times \mathbb R \to \mathbb R$ such that
    \begin{equation*}
        H(x,f(x) + s) = H_{\Gamma(f + s)}(x,f(x) + s) = H_{\Gamma(f)}(x,f(x)) \quad \text{for $x \in \Omega$ and $s \in \mathbb R$}. 
    \end{equation*}
    The proof of the fact that $\LB \psi(\Gamma(f))\RB$ has the $C$-almost minimizing property is very similar to that of Eichmair \cite[Example A.$1$]{EichmairPlateau}. We define the $n$-form $\sigma := dV^h \mres \nu := dV^h(\nu, \cdot, \dots, \cdot)$. For all $s \in \mathbb R$, $\sigma$ is the Riemannian volume form of $\Gamma(f + s)$ and consequently on $\Omega \times \mathbb R$ we have
    \begin{align}
        \label{eq:Sigma1}
        |\sigma(v_1, \dots, v_n)| & \leq 1 && \parbox{15em}{for any orthonormal collection of \\ vectorfields $\{v_1,\dots, v_n\}$ on $\Omega \times \mathbb R$,} \\
        \label{eq:Sigma2}
        d\sigma = (\div \nu)dV^h
        & = H dV^h && \text{on $\Omega \times \mathbb R$}
    \end{align}
    We now set $N := \psi(M \times \mathbb R)$ and let $U \subset \mathbb R^{n+l}$ be any open set such that $\psi(\Omega \times \mathbb R) = \psi(M\times \mathbb R) \cap U$. In this way $N \subset \mathbb R^{n+l}$ is an oriented and embedded $C^2$ submanifold and letting $E := \{(x,y) \in \Omega \times \mathbb R: y > f(x)\}$ it follows that $\psi(E) \subset N \cap U$ is $\mathcal H^{n+1}$-measurable and that $\psi(E)$ has locally finite perimeter in $N$. Setting
    \begin{equation*}
        T := \LB\psi(\Gamma(f))\RB = \LB\partial\psi(E) \RB = \partial\LB \psi(E) \RB, 
    \end{equation*}
    we see that $T \in D_n(U)$ is of integer multiplicity and that $\text{spt}(T) \subset N\cap U$. 

    Now let $W$ and $X$ be as in \autoref{CalmostMinimizer}, so that $W \ssubset U$ is a compactly contained open set and $X \in D_{n+1}(U)$ is of integer multiplicity and has compact support in $N \cap W$. We denote by $D(W,X)$ the collection of $C^2$ functions on $U$ with compact support in $W$ that are identically equal to $1$ in a neighborhood of $\text{spt}(X)$. Then we have
    \begin{align*}
        M_W(\partial\LB\psi(E)\RB + \partial X) 
        & = \sup_{\substack{w \in D^n(W) \\ |w| \leq 1}} (\partial \LB\psi(E)\RB + \partial X)(w) 
        \geq \sup_{\substack{\phi \in D(W,X) \\ |\phi| \leq 1}} (\partial \LB\psi(E)\RB + \partial X)(\phi \psi_*\sigma) \\
        & \geq \sup_{\substack{\phi \in D(W,X) \\ |\phi| \leq 1}} \partial \LB\psi(E)\RB(\phi \psi_*\sigma) - \sup_{\substack{\phi \in D(W,X) \\ |\phi| \leq 1}}(-\partial X)(\phi \psi_*\sigma).
    \end{align*} 
    As for the first term in the last line, we note that
    \begin{align*}
        \sup_{\substack{\phi \in D(W,X) \\ |\phi| \leq 1}} \partial \LB\psi(E)\RB(\phi \cdot \psi_*\sigma) 
        = \sup_{\substack{\phi \in D(W,X) \\ |\phi| \leq 1}} \int_{\partial\psi(E)} \phi\psi_*\sigma 
        = \int_{\partial\psi(E)} \psi_*\sigma = M_W(\partial [\psi(E)]).
    \end{align*}
    Furthermore, since $\phi \equiv 1$ in a neighborhood of $\text{spt}(X)$ and $d\sigma = HdV^h$, the second term can be estimated as follows
    \begin{align*}
        \sup_{\substack{\phi \in D^n(W,X) \\ |\phi| \leq 1}}(-\partial X)(\phi\psi_*\sigma) 
        = \sup_{\substack{\phi \in D^n(W,X) \\ |\phi| \leq 1}} (-X)(\phi\psi_*d\sigma) 
        \leq C\sup_{\substack{\phi \in D^n(W,X) \\ |\phi| \leq 1}} |X(\phi\psi_*dV^g)|
        \leq C M_W(X).
    \end{align*}
    In summary, we have shown that $M_W(\partial \LB\psi(E)\RB + \partial X) \geq M_W(\partial \LB\psi(E)\RB) - C M_W(X)$, as claimed.
\end{proof}
We conclude by outlining the proof of \autoref{prop:LimitSubmanifolds}, for which we will heavily rely on the results of \cite[Appendix A]{EichmairPlateau}. These results are formulated for currents of codimension $1$ in $\mathbb R^{n+1}$ however as explained in \cite[Remark A$.3$]{EichmairPlateau}, they carry over to Riemannian manifolds by isometrically embedding them into a higher dimensional Euclidean space. This is very similar to the related theory for area minimizing currents in \cite[Chapter $7$, Section $5$]{SimonGMT}, see also \cite[Section $37$]{SimonGMTOld}.

\begin{proof}[Proof of \autoref{prop:LimitSubmanifolds}]
    By the Nash embedding theorem, see \cite[Chapter $3.1$, Imbedding Theorem]{Gromov86} and \cite[Section $1.2$, footnote $7$]{GromovOverview}, there is a $C^{2,\alpha}$ isometry $\psi: M\times_u \mathbb R \to \mathbb R^{n+l}$ and we may choose an open set $U \subset \mathbb R^{n+l}$ such that $\psi(\Omega \times \mathbb R) = \psi(M \times \mathbb R) \cap U$. Recalling the notation $\widetilde M = M\times \mathbb R$, we define 
    \begin{equation}
        \label{eq:CminSetup}
        N := \psi(\widetilde M), \quad E_i := \{(x,y): y > f_i(x), x \in \Omega\} \subset \widetilde M, \quad \overline E_i := \psi(E_i) \subset N \subset \mathbb R^{n+l}.
    \end{equation}
    \autoref{lem:CMinExistence} implies that the currents $\partial \LB \overline E_i\RB$ are $C$-almost minimizing in $N \cap U$. The $C$-almost minimizing property can be used to show the locally uniform bound
    \begin{equation*}
        \sup_i M_W(\LB \overline E_i\RB) + \sup_i M_W(\partial\LB \overline E_i\RB) \leq C_W < \infty.
    \end{equation*}
    The compactness theorem for integer multiplicity currents with locally bounded mass (see \cite[Chapter $6$, Theorem $3.11$]{SimonGMT}) implies that upon passing to a subsequence we have $\partial\LB \overline E_i\RB \rightharpoonup T$ for some $T \in D_n(U)$ of integer multiplicity. Since each $\partial\LB \overline E_i\RB$ is supported in $N \cap U$, so is $T$. By \cite[Lemma A$.2$]{EichmairPlateau} it follows that there is a relatively open set $\overline E \subset N$ of locally finite perimeter in $N$ such that $T = \partial \LB \overline E\RB$ is $C$-almost minimizing in $N \cap U$ and furthermore, after passing to another subsequence, we have $\partial \LB\overline E_i\RB \rightharpoonup \partial \LB \overline E\RB$, $\chi_{\overline E_i} \to \chi_{\overline E}$ in $\text{BV}_\text{loc}(N)$ and the underlying varifolds of $\partial \LB \overline E_i \RB$ converge to the underlying varifold of $T = \partial\LB \overline E\RB$ in the sense of Radon measures.

    By \cite[Theorem A$.1$]{EichmairPlateau} we know that $\partial \overline E \subset \mathbb R^{n+l}$ is $C^{1,\gamma}_\text{loc}$ for some $\gamma \in (0,1)$, except when $n = 7$ in which case $\partial \overline E$ may have isolated singularities. Let $p$ be an isolated singularity and write $\partial \overline E$ as a graph over its tangent space close to $p$ with $\varphi: B_r(0)\to N$ the graphing function. The completion in $C^0(T_p(\partial \overline E))$ of the family $\mathcal F = \{\varphi_{x,\lambda}(y) = \varphi(x+\lambda y)\}$ can be checked to satisfy \cite[Assumption A.$1$]{SimonGMT}, \cite[Assumption A.$2$]{SimonGMT} and \cite[Assumption A.$3$]{SimonGMT}. Since the singularity of $\varphi$ is isolated it follows that we may let $\phi = \lim_{k \to \infty}\varphi_{x_0,\lambda_k} \in \mathcal F$ for some $\lambda_k \to 0$ so that $\phi$ graphs a minimal surface contained in $T_pN$ with codomension $1$ in $T_pN$ and with an isolated singularity at the origin. After identifying $T_pN$ with $\mathbb R^{8}$ we obtain a graphical, conical minimal surface in $\mathbb R^8$ with an isolated singularity at the origin. The rest of the argument can now be carried out as in Eichmair \cite[Remark $4.1$]{EichmairPlateau} to show that the singular point $p$ did not exist in the first place. Hence $\partial \overline E$ is $C^{1,\gamma}_\text{loc}$.

    Fixing $p \in \partial \overline E$, the approximate monotonicity formula in \cite[Chapter $4$, Theorem $3.17$]{SimonGMT} can be used together with the uniform bound on $\sup_i|F_i| \leq C < \infty$ to show that there exists an index $i(p) > 0$ such that the conditions of Allard's regularity theorem \cite[Chapter $5$, Theorem $5.2$]{SimonGMT} hold with all constants independent of $i \geq i(p)$. Hence, there is a fixed radius $r = r(p) > 0$ such that for $i \geq i(p)$, $\partial \overline E \cap B_r(p)$ and each $\partial \overline E_i \cap B_r(p)$ can be written as a  $C^{1,\gamma}$ graph over the tangent space $T_p(\partial \overline E)$. Moreover, $C^{1,\gamma}$ norms of the graphing functions are uniformly bounded by a constant independent of $i$. 

    Since $\psi^{-1}$ is $C^2$ we may define $E := \psi^{-1}(\overline E)$ and by the above we know that near each point $p \in \partial E$, both $\partial E$ and $\partial E_i$, for all $i$ large enough, can be represented as graphs of functions with uniformly bounded $C^{1,\gamma}$ norms over the tangent space $T_p(\partial E)$. These facts combined with standard elliptic estimates for the prescribed mean curvature equations, cf. \cite[Proposition $5.4$]{SakovichJang} and \cite[Appendix B]{CederbaumSakovich}, imply the desired $C^{3,\alpha}_\text{loc}$ regularity and convergence of the graphing functions.
\end{proof}







\newpage
\bibliographystyle{abbrv}
\bibliography{ref}

@article {AgostinianiMazzieriOronzio,
    AUTHOR = {Agostiniani, Virginia and Mazzieri, Lorenzo and Oronzio, Francesca},
     TITLE = {A {G}reen's function proof of the positive mass theorem},
   JOURNAL = {Comm. Math. Phys.},
  FJOURNAL = {Communications in Mathematical Physics},
    VOLUME = {405},
      YEAR = {2024},
    NUMBER = {2},
     PAGES = {Paper No. 54, 23},
      ISSN = {0010-3616,1432-0916},
   MRCLASS = {53C20 (31C12)},
  MRNUMBER = {4707045},
       DOI = {10.1007/s00220-024-04941-8},
       URL = {https://doi.org/10.1007/s00220-024-04941-8},
}

@article {AnderssonCaiGalloway,
    AUTHOR = {Andersson, Lars and Cai, Mingliang and Galloway, Gregory J.},
     TITLE = {Rigidity and positivity of mass for asymptotically hyperbolic manifolds},
   JOURNAL = {Ann. Henri Poincar\'{e}},
  FJOURNAL = {Annales Henri Poincar\'{e}. A Journal of Theoretical and Mathematical Physics},
    VOLUME = {9},
      YEAR = {2008},
    NUMBER = {1},
     PAGES = {1--33},
      ISSN = {1424-0637,1424-0661},
   MRCLASS = {53C24 (53C20 53C21 53C80)},
  MRNUMBER = {2389888},
MRREVIEWER = {Alan\ D.\ Rendall},
       DOI = {10.1007/s00023-007-0348-2},
       URL = {https://doi.org/10.1007/s00023-007-0348-2},
}

@incollection {AnderssonEichmairMetzger,
    AUTHOR = {Andersson, Lars and Eichmair, Michael and Metzger, Jan},
    TITLE = {Jang's equation and its applications to marginally trapped surfaces},
    BOOKTITLE = {Complex analysis and dynamical systems {IV}. {P}art 2},
    SERIES = {Contemporary Mathematics},
    VOLUME = {554},
     PAGES = {13--45},
 PUBLISHER = {American Mathematical Sociecy, Providence, RI},
      YEAR = {2011},
      ISBN = {978-0-8218-5197-5},
   MRCLASS = {53C21 (35R01 53A10 53C80)},
  MRNUMBER = {2884392},
MRREVIEWER = {Gabjin\ Yun},
       DOI = {10.1090/conm/554/10958},
       URL = {https://doi.org/10.1090/conm/554/10958},
}

@article {BrayKhuri,
    AUTHOR = {Bray, Hubert L. and Khuri, Marcus A.},
     TITLE = {A {J}ang equation approach to the {P}enrose inequality},
   JOURNAL = {Discrete Contin. Dyn. Syst.},
  FJOURNAL = {Discrete and Continuous Dynamical Systems. Series A},
    VOLUME = {27},
      YEAR = {2010},
    NUMBER = {2},
     PAGES = {741--766},
      ISSN = {1078-0947,1553-5231},
   MRCLASS = {53C21 (35Q76 53C80 83C40 83C57 83C75)},
  MRNUMBER = {2600688},
MRREVIEWER = {Jos\'{e}\ Nat\'{a}rio},
       DOI = {10.3934/dcds.2010.27.741},
       URL = {https://doi.org/10.3934/dcds.2010.27.741},
}

@article {BrayKhuriKazarasStern,
    AUTHOR = {Bray, Hubert L. and Kazaras, Demetre P. and Khuri, Marcus A. and Stern, Daniel L.},
     TITLE = {Harmonic functions and the mass of 3-dimensional asymptotically flat {R}iemannian manifolds},
   JOURNAL = {J. Geom. Anal.},
  FJOURNAL = {Journal of Geometric Analysis},
    VOLUME = {32},
      YEAR = {2022},
    NUMBER = {6},
     PAGES = {Paper No. 184, 29},
      ISSN = {1050-6926,1559-002X},
   MRCLASS = {53C20 (53C21 58J90 58Z05 83C40 83C57)},
  MRNUMBER = {4411747},
MRREVIEWER = {Alfonso\ Garc\'{\i}a-Parrado},
       DOI = {10.1007/s12220-022-00924-0},
       URL = {https://doi.org/10.1007/s12220-022-00924-0},
}

@article {CederbaumSakovich,
    AUTHOR = {Cederbaum, Carla and Sakovich, Anna},
     TITLE = {On center of mass and foliations by constant spacetime mean curvature surfaces for isolated systems in general relativity},
   JOURNAL = {Calc. Var. Partial Differential Equations},
  FJOURNAL = {Calculus of Variations and Partial Differential Equations},
    VOLUME = {60},
      YEAR = {2021},
    NUMBER = {6},
     PAGES = {Paper No. 214, 57},
      ISSN = {0944-2669,1432-0835},
   MRCLASS = {53C21 (58J37 83C05)},
  MRNUMBER = {4305436},
MRREVIEWER = {Ettore\ Minguzzi},
       DOI = {10.1007/s00526-021-02060-z},
       URL = {https://doi.org/10.1007/s00526-021-02060-z},
}

@article {ChaKhuri15,
    AUTHOR = {Cha, Ye Sle and Khuri, Marcus A.},
     TITLE = {Deformations of charged axially symmetric initial data and the mass-angular momentum-charge inequality},
   JOURNAL = {Ann. Henri Poincar\'e},
  FJOURNAL = {Annales Henri Poincar\'e. A Journal of Theoretical and Mathematical Physics},
    VOLUME = {16},
      YEAR = {2015},
    NUMBER = {12},
     PAGES = {2881--2918},
      ISSN = {1424-0637,1424-0661},
   MRCLASS = {83C22},
  MRNUMBER = {3416872},
MRREVIEWER = {Yong\ Wang},
       DOI = {10.1007/s00023-014-0378-5},
       URL = {https://doi.org/10.1007/s00023-014-0378-5},
}

@article {ChaKhuri18,
    AUTHOR = {Cha, Ye Sle and Khuri, Marcus},
     TITLE = {Transformations of asymptotically {A}d{S} hyperbolic initial data and associated geometric inequalities},
   JOURNAL = {Gen. Relativity Gravitation},
  FJOURNAL = {General Relativity and Gravitation},
    VOLUME = {50},
      YEAR = {2018},
    NUMBER = {1},
     PAGES = {Paper No. 3, 48},
      ISSN = {0001-7701,1572-9532},
   MRCLASS = {83C22},
  MRNUMBER = {3736107},
MRREVIEWER = {Yong\ Wang},
       DOI = {10.1007/s10714-017-2323-7},
       URL = {https://doi.org/10.1007/s10714-017-2323-7},
}

@article {ChaKhuriSakovich,
    AUTHOR = {Cha, Ye Sle and Khuri, Marcus and Sakovich, Anna},
     TITLE = {Reduction arguments for geometric inequalities associated with asymptotically hyperboloidal slices},
   JOURNAL = {Classical and Quantum Gravity},
  FJOURNAL = {Classical and Quantum Gravity},
    VOLUME = {33},
      YEAR = {2016},
    NUMBER = {3},
     PAGES = {035009, 33},
      ISSN = {0264-9381,1361-6382},
   MRCLASS = {83C05 (53C80 83C22 83C57)},
  MRNUMBER = {3529561},
MRREVIEWER = {Janusz\ Garecki},
       DOI = {10.1088/0264-9381/33/3/035009},
       URL = {https://doi.org/10.1088/0264-9381/33/3/035009},
}

@misc{ChruscielDelayHyp,
      title={The hyperbolic positive energy theorem}, 
      author={Piotr T. Chruściel and Erwann Delay},
      year={2019},
      eprint={1901.05263},
      archivePrefix={arXiv},
      primaryClass={math.DG},
      note = {https://arxiv.org/abs/1901.05263}
}

@article {ChruscielGalloway,
    AUTHOR = {Chru\'{s}ciel, Piotr T. and Galloway, Gregory J.},
     TITLE = {Positive mass theorems for asymptotically hyperbolic
              {R}iemannian manifolds with boundary},
   JOURNAL = {Classical and Quantum Gravity},
    VOLUME = {38},
      YEAR = {2021},
    NUMBER = {23},
     PAGES = {Paper No. 237001, 6},
      ISSN = {0264-9381,1361-6382},
   MRCLASS = {83C40 (53C20 83C05 83C30)},
  MRNUMBER = {4353388},
MRREVIEWER = {Veselin\ T.\ Videv},
       DOI = {10.1088/1361-6382/ac1fd1},
       URL = {https://doi.org/10.1088/1361-6382/ac1fd1},
}

@article {ChruscielGallowayNguyenPaetz,
    AUTHOR = {Chru\'{s}ciel, Piotr T. and Galloway, Gregory J. and Nguyen,
              Luc and Paetz, Tim-Torben},
     TITLE = {On the mass aspect function and positive energy theorems for
              asymptotically hyperbolic manifolds},
   JOURNAL = {Classical Quantum Gravity},
  FJOURNAL = {Classical and Quantum Gravity},
    VOLUME = {35},
      YEAR = {2018},
    NUMBER = {11},
     PAGES = {115015, 38},
      ISSN = {0264-9381,1361-6382},
   MRCLASS = {53C20 (53C21)},
MRREVIEWER = {Man\ Chun\ Leung},
       DOI = {10.1088/1361-6382/aabed1},
       URL = {https://doi.org/10.1088/1361-6382/aabed1},
}

@article {ChruscielHerzlich,
    AUTHOR = {Chru\'{s}ciel, Piotr T. and Herzlich, Marc},
     TITLE = {The mass of asymptotically hyperbolic {R}iemannian manifolds},
   JOURNAL = {Pacific J. Math.},
  FJOURNAL = {Pacific Journal of Mathematics},
    VOLUME = {212},
      YEAR = {2003},
    NUMBER = {2},
     PAGES = {231--264},
      ISSN = {0030-8730,1945-5844},
   MRCLASS = {53C21 (83C05 83C30)},
  MRNUMBER = {2038048},
MRREVIEWER = {Simonetta\ Frittelli},
       DOI = {10.2140/pjm.2003.212.231},
       URL = {https://doi.org/10.2140/pjm.2003.212.231},
}

@article {ChruscielJezierskiLeski,
    AUTHOR = {Chru\'{s}ciel, Piotr T. and Jezierski, Jacek and \L\c{e}ski, Szymon},
     TITLE = {The {T}rautman-{B}ondi mass of hyperboloidal initial data sets},
   JOURNAL = {Adv. Theor. Math. Phys.},
  FJOURNAL = {Advances in Theoretical and Mathematical Physics},
    VOLUME = {8},
      YEAR = {2004},
    NUMBER = {1},
     PAGES = {83--139},
      ISSN = {1095-0761,1095-0753},
   MRCLASS = {83C30 (83C05)},
  MRNUMBER = {2086675},
MRREVIEWER = {Alexander\ N.\ Petrov},
       URL = {http://projecteuclid.org/euclid.atmp/1091475314},
}

@article {ChruscielNagy,
    AUTHOR = {Chru\'{s}ciel, Piotr T. and Nagy, Gabriel},
     TITLE = {The mass of spacelike hypersurfaces in asymptotically anti-de {S}itter space-times},
   JOURNAL = {Adv. Theor. Math. Phys.},
  FJOURNAL = {Advances in Theoretical and Mathematical Physics},
    VOLUME = {5},
      YEAR = {2001},
    NUMBER = {4},
     PAGES = {697--754},
      ISSN = {1095-0761,1095-0753},
   MRCLASS = {53C50 (53C40 83C05 83C30)},
  MRNUMBER = {1926293},
MRREVIEWER = {Lars\ \AA ke\ Andersson},
       DOI = {10.4310/ATMP.2001.v5.n4.a3},
       URL = {https://doi.org/10.4310/ATMP.2001.v5.n4.a3},
}

@misc{CortierDahlGicquaud,
  doi = {10.48550/ARXIV.1603.07952},
  note = {https://arxiv.org/abs/1603.07952},
  author = {Cortier, Julien and Dahl, Mattias and Gicquaud, Romain},
  title = {Mass-like invariants for asymptotically hyperbolic metrics},
  publisher = {arXiv},
  year = {2016},
  copyright = {arXiv.org perpetual, non-exclusive license}
}

@article {DuzaarSteffen,
    AUTHOR = {Duzaar, Frank and Steffen, Klaus},
     TITLE = {{$\lambda$} minimizing currents},
   JOURNAL = {Manuscripta Math.},
  FJOURNAL = {Manuscripta Mathematica},
    VOLUME = {80},
      YEAR = {1993},
    NUMBER = {4},
     PAGES = {403--447},
      ISSN = {0025-2611,1432-1785},
   MRCLASS = {49Q15},
  MRNUMBER = {1243155},
       DOI = {10.1007/BF03026561},
       URL = {https://doi.org/10.1007/BF03026561},
}

@article {EichmairReduction,
    AUTHOR = {Eichmair, Michael},
     TITLE = {The {J}ang equation reduction of the spacetime positive energy theorem in dimensions less than eight},
   JOURNAL = {Comm. Math. Phys.},
  FJOURNAL = {Communications in Mathematical Physics},
    VOLUME = {319},
      YEAR = {2013},
    NUMBER = {3},
     PAGES = {575--593},
      ISSN = {0010-3616,1432-0916},
   MRCLASS = {58J99 (53A99 53C42 58A25)},
  MRNUMBER = {3040369},
MRREVIEWER = {Mircea\ Crasmareanu},
       DOI = {10.1007/s00220-013-1700-7},
       URL = {https://doi.org/10.1007/s00220-013-1700-7},
}

@article {EichmairHuangLeeSchoen,
    AUTHOR = {Eichmair, Michael and Huang, Lan-Hsuan and Lee, Dan A. and Schoen, Richard},
     TITLE = {The spacetime positive mass theorem in dimensions less than eight},
   JOURNAL = {J. Eur. Math. Soc. (JEMS)},
  FJOURNAL = {Journal of the European Mathematical Society (JEMS)},
    VOLUME = {18},
      YEAR = {2016},
    NUMBER = {1},
     PAGES = {83--121},
      ISSN = {1435-9855,1435-9863},
   MRCLASS = {53C21 (53C20 83C05)},
  MRNUMBER = {3438380},
MRREVIEWER = {Seungsu\ Hwang},
       DOI = {10.4171/JEMS/584},
       URL = {https://doi.org/10.4171/JEMS/584},
}

@article {EichmairPlateau,
    AUTHOR = {Eichmair, Michael},
     TITLE = {The {P}lateau problem for marginally outer trapped surfaces},
   JOURNAL = {J. Differential Geom.},
  FJOURNAL = {Journal of Differential Geometry},
    VOLUME = {83},
      YEAR = {2009},
    NUMBER = {3},
     PAGES = {551--583},
      ISSN = {0022-040X,1945-743X},
   MRCLASS = {53C21 (49Q05 53A10 53C50)},
  MRNUMBER = {2581357},
MRREVIEWER = {Gabjin\ Yun},
       URL = {http://projecteuclid.org/euclid.jdg/1264601035},
}

@article {EichmairMetzger,
    AUTHOR = {Eichmair, Michael and Metzger, Jan},
     TITLE = {Jenkins-{S}errin-type results for the {J}ang equation},
   JOURNAL = {J. Differential Geom.},
  FJOURNAL = {Journal of Differential Geometry},
    VOLUME = {102},
      YEAR = {2016},
    NUMBER = {2},
     PAGES = {207--242},
      ISSN = {0022-040X,1945-743X},
   MRCLASS = {53C42 (83C05)},
  MRNUMBER = {3454546},
MRREVIEWER = {Xin\ Zhou},
       URL = {http://projecteuclid.org/euclid.jdg/1453910454},
}

@article {Gicquaud,
    AUTHOR = {Gicquaud, Romain},
     TITLE = {De l'\'{e}quation de prescription de courbure scalaire aux \'{e}quations de contrainte en relativit\'{e} g\'{e}n\'{e}rale sur une vari\'{e}t\'{e} asymptotiquement hyperbolique},
   JOURNAL = {J. Math. Pures Appl. (9)},
  FJOURNAL = {Journal de Math\'{e}matiques Pures et Appliqu\'{e}es. Neuvi\`eme S\'{e}rie},
    VOLUME = {94},
      YEAR = {2010},
    NUMBER = {2},
     PAGES = {200--227},
      ISSN = {0021-7824,1776-3371},
   MRCLASS = {53C21 (35J60 35J61 35R01 58J60)},
  MRNUMBER = {2665418},
MRREVIEWER = {David\ L.\ Finn},
       DOI = {10.1016/j.matpur.2010.03.011},
       URL = {https://doi.org/10.1016/j.matpur.2010.03.011},
}

@book {GilbargTrudinger,
    AUTHOR = {Gilbarg, David and Trudinger, Neil S.},
     TITLE = {Elliptic partial differential equations of second order},
    SERIES = {Classics in Mathematics},
      NOTE = {Reprint of the 1998 edition},
 PUBLISHER = {Springer-Verlag, Berlin},
      YEAR = {2001},
     PAGES = {xiv+517},
      ISBN = {3-540-41160-7},
   MRCLASS = {35-02 (35Jxx)},
  MRNUMBER = {1814364},
}

@book {Gromov86,
    AUTHOR = {Gromov, Mikhael},
     TITLE = {Partial differential relations},
    SERIES = {Ergebnisse der Mathematik und ihrer Grenzgebiete (3) [Results in Mathematics and Related Areas (3)]},
    VOLUME = {9},
 PUBLISHER = {Springer-Verlag, Berlin},
      YEAR = {1986},
     PAGES = {x+363},
      ISBN = {3-540-12177-3},
   MRCLASS = {58G99 (35A99 35B99 53C42 58-02)},
  MRNUMBER = {864505},
MRREVIEWER = {Hung-Hsi\ Wu},
       DOI = {10.1007/978-3-662-02267-2},
       URL = {https://doi.org/10.1007/978-3-662-02267-2},
}

@article {GromovOverview,
    AUTHOR = {Gromov, Misha},
     TITLE = {Geometric, algebraic, and analytic descendants of {N}ash isometric embedding theorems},
   JOURNAL = {Bull. Amer. Math. Soc. (N.S.)},
  FJOURNAL = {American Mathematical Society. Bulletin. New Series},
    VOLUME = {54},
      YEAR = {2017},
    NUMBER = {2},
     PAGES = {173--245},
      ISSN = {0273-0979,1088-9485},
   MRCLASS = {58D10 (58-02 58Jxx)},
  MRNUMBER = {3619725},
MRREVIEWER = {Fr\'ed\'eric\ Robert},
       DOI = {10.1090/bull/1551},
       URL = {https://doi.org/10.1090/bull/1551},
}

@article {HanKhuri,
    AUTHOR = {Han, Qing and Khuri, Marcus},
     TITLE = {Existence and blow-up behavior for solutions of the
              generalized {J}ang equation},
   JOURNAL = {Comm. Partial Differential Equations},
  FJOURNAL = {Communications in Partial Differential Equations},
    VOLUME = {38},
      YEAR = {2013},
    NUMBER = {12},
     PAGES = {2199--2237},
      ISSN = {0360-5302,1532-4133},
   MRCLASS = {35Q75 (35B44 35B45 35J70 58J32)},
  MRNUMBER = {3169775},
MRREVIEWER = {Willie\ W.\ Wong},
       DOI = {10.1080/03605302.2013.837919},
       URL = {https://doi.org/10.1080/03605302.2013.837919},
}

@article {HirschKazarasKhuri,
    AUTHOR = {Hirsch, Sven and Kazaras, Demetre and Khuri, Marcus},
     TITLE = {Spacetime harmonic functions and the mass of 3-dimensional asymptotically flat initial data for the {E}instein equations},
   JOURNAL = {J. Differential Geom.},
  FJOURNAL = {Journal of Differential Geometry},
    VOLUME = {122},
      YEAR = {2022},
    NUMBER = {2},
     PAGES = {223--258},
      ISSN = {0022-040X,1945-743X},
   MRCLASS = {53C20 (53C21 58E20 83C40)},
  MRNUMBER = {4516940},
       DOI = {10.4310/jdg/1669998184},
       URL = {https://doi.org/10.4310/jdg/1669998184},
}

@article {HuangLeeMartin,
    AUTHOR = {Huang, Lan-Hsuan and Jang, Hyun Chul and Martin, Daniel},
     TITLE = {Mass rigidity for hyperbolic manifolds},
   JOURNAL = {Comm. Math. Phys.},
  FJOURNAL = {Communications in Mathematical Physics},
    VOLUME = {376},
      YEAR = {2020},
    NUMBER = {3},
     PAGES = {2329--2349},
      ISSN = {0010-3616,1432-0916},
   MRCLASS = {53C20 (53C24)},
  MRNUMBER = {4104551},
MRREVIEWER = {Theodoros\ Vlachos},
       DOI = {10.1007/s00220-019-03623-0},
       URL = {https://doi.org/10.1007/s00220-019-03623-0},
}

@article {Jang,
    AUTHOR = {Jang, Pong Soo},
     TITLE = {On the positivity of energy in general relativity},
   JOURNAL = {J. Math. Phys.},
  FJOURNAL = {Journal of Mathematical Physics},
    VOLUME = {19},
      YEAR = {1978},
    NUMBER = {5},
     PAGES = {1152--1155},
      ISSN = {0022-2488,1089-7658},
   MRCLASS = {83C20 (58D99)},
  MRNUMBER = {488515},
MRREVIEWER = {W.\ Israel},
       DOI = {10.1063/1.523776},
       URL = {https://doi.org/10.1063/1.523776},
}

@misc{LesourdUngerYau,
      title={The {P}ositive {M}ass {T}heorem with {A}rbitrary {E}nds}, 
      note={https://arxiv.org/abs/2103.02744},
      author={Martin Lesourd and Ryan Unger and Shing-Tung Yau},
      year={2021},
      eprint={2103.02744},
      archivePrefix={arXiv},
      primaryClass={math.DG}
}

@misc{Lundberg,
      title={On {J}ang's equation and the {P}ositive {M}ass {T}heorem for asymptotically hyperbolic initial data sets with dimensions above three and below eight}, 
      author={Lundberg, David},
      year={2023},
      eprint={2309.11330},
      archivePrefix={arXiv},
      primaryClass={math.DG},
      note = {https://arxiv.org/abs/2309.11330}
}

@article {MalecMurchadha2004,
    AUTHOR = {Malec, Edward and \'{O} Murchadha, Niall},
     TITLE = {The {J}ang equation, apparent horizons and the {P}enrose inequality},
   JOURNAL = {Classical Quantum Gravity},
  FJOURNAL = {Classical and Quantum Gravity},
    VOLUME = {21},
      YEAR = {2004},
    NUMBER = {24},
     PAGES = {5777--5787},
      ISSN = {0264-9381,1361-6382},
   MRCLASS = {83C30 (83C05)},
  MRNUMBER = {2107339},
MRREVIEWER = {Yakov\ Itin},
       DOI = {10.1088/0264-9381/21/24/007},
       URL = {https://doi.org/10.1088/0264-9381/21/24/007},
}

@article {MetzgerBlowUp,
    AUTHOR = {Metzger, Jan},
     TITLE = {Blowup of {J}ang's equation at outermost marginally trapped surfaces},
   JOURNAL = {Comm. Math. Phys.},
  FJOURNAL = {Communications in Mathematical Physics},
    VOLUME = {294},
      YEAR = {2010},
    NUMBER = {1},
     PAGES = {61--72},
      ISSN = {0010-3616,1432-0916},
   MRCLASS = {53C50 (53C80 83C05)},
  MRNUMBER = {2575475},
MRREVIEWER = {Robert\ J.\ Low},
       DOI = {10.1007/s00220-009-0934-x},
       URL = {https://doi.org/10.1007/s00220-009-0934-x},
}

@article {Michel,
    AUTHOR = {Michel, Benoît.},
     TITLE = {Geometric invariance of mass-like asymptotic invariants},
   JOURNAL = {J. Math. Phys.},
  FJOURNAL = {Journal of Mathematical Physics},
    VOLUME = {52},
      YEAR = {2011},
    NUMBER = {5},
     PAGES = {052504, 14},
      ISSN = {0022-2488,1089-7658},
   MRCLASS = {83C30 (58D15)},
  MRNUMBER = {2839077},
MRREVIEWER = {Hans-Peter\ K\"{u}nzle},
       DOI = {10.1063/1.3579137},
       URL = {https://doi.org/10.1063/1.3579137},
}

@article {SakovichJang,
    AUTHOR = {Sakovich, Anna},
     TITLE = {The {J}ang equation and the positive mass theorem in the
              asymptotically hyperbolic setting},
   JOURNAL = {Comm. Math. Phys.},
  FJOURNAL = {Communications in Mathematical Physics},
    VOLUME = {386},
      YEAR = {2021},
    NUMBER = {2},
     PAGES = {903--973},
      ISSN = {0010-3616,1432-0916},
   MRCLASS = {53C20},
  MRNUMBER = {4294283},
MRREVIEWER = {Xiao\ Zhang},
       DOI = {10.1007/s00220-021-04083-1},
       URL = {https://doi.org/10.1007/s00220-021-04083-1},
}

@article{SimonGMT,
  title={Introduction to geometric measure theory},
  author={Simon, Leon},
  journal={Tsinghua Lectures},
  volume={2},
  number={2},
  pages={3--1},
  year={2014}
}

@article{SasakiRef,
    author = {Shigeo Sasaki},
    title = {{On the differential geometry of tangent bundles of Riemannian manifolds}},
    volume = {10},
    journal = {Tohoku Mathematical Journal},
    number = {3},
    publisher = {Tohoku University, Mathematical Institute},
    pages = {338 -- 354},
    year = {1958},
    doi = {10.2748/tmj/1178244668},
    URL = {https://doi.org/10.2748/tmj/1178244668}
}

@book {SimonGMTOld,
    AUTHOR = {Simon, Leon},
     TITLE = {Lectures on geometric measure theory},
    SERIES = {Proceedings of the Centre for Mathematical Analysis, Australian National University},
    VOLUME = {3},
 PUBLISHER = {Australian National University, Centre for Mathematical Analysis, Canberra},
      YEAR = {1983},
     PAGES = {vii+272},
      ISBN = {0-86784-429-9},
   MRCLASS = {49-01 (28A75 49F20)},
  MRNUMBER = {756417},
MRREVIEWER = {J.\ S.\ Joel},
}

@article {Spruck,
    AUTHOR = {Spruck, Joel},
     TITLE = {Interior gradient estimates and existence theorems for constant mean curvature graphs in {$M^n\times\mathbb R$}},
   JOURNAL = {Pure Appl. Math. Q.},
  FJOURNAL = {Pure and Applied Mathematics Quarterly},
    VOLUME = {3},
      YEAR = {2007},
    NUMBER = {3},
     PAGES = {785--800},
      ISSN = {1558-8599,1558-8602},
   MRCLASS = {58E12 (53C42)},
  MRNUMBER = {2351645},
MRREVIEWER = {Jos\'{e}\ Antonio\ G\'{a}lvez},
       DOI = {10.4310/PAMQ.2007.v3.n3.a6},
       URL = {https://doi.org/10.4310/PAMQ.2007.v3.n3.a6},
}

@article {SchoenYau1,
    AUTHOR = {Schoen, Richard and Yau, Shing Tung},
     TITLE = {On the proof of the positive mass conjecture in general relativity},
   JOURNAL = {Comm. Math. Phys.},
  FJOURNAL = {Communications in Mathematical Physics},
    VOLUME = {65},
      YEAR = {1979},
    NUMBER = {1},
     PAGES = {45--76},
      ISSN = {0010-3616,1432-0916},
   MRCLASS = {83C99 (53C20 58E20)},
  MRNUMBER = {526976},
MRREVIEWER = {J.\ L.\ Kazdan},
       URL = {http://projecteuclid.org/euclid.cmp/1103904790},
}

@article {SchoenYau2,
    AUTHOR = {Schoen, Richard and Yau, Shing Tung},
     TITLE = {Proof of the positive mass theorem. {II}},
   JOURNAL = {Comm. Math. Phys.},
  FJOURNAL = {Communications in Mathematical Physics},
    VOLUME = {79},
      YEAR = {1981},
    NUMBER = {2},
     PAGES = {231--260},
      ISSN = {0010-3616,1432-0916},
   MRCLASS = {83C99 (35J60 53A10 53C50 58G40 83C05)},
  MRNUMBER = {612249},
MRREVIEWER = {J.\ L.\ Kazdan},
       URL = {http://projecteuclid.org/euclid.cmp/1103908964},
}

@article {Wang,
    AUTHOR = {Wang, Xiaodong},
     TITLE = {The mass of asymptotically hyperbolic manifolds},
   JOURNAL = {J. Differential Geom.},
  FJOURNAL = {Journal of Differential Geometry},
    VOLUME = {57},
      YEAR = {2001},
    NUMBER = {2},
     PAGES = {273--299},
      ISSN = {0022-040X,1945-743X},
   MRCLASS = {53C20 (53C21)},
  MRNUMBER = {1879228},
MRREVIEWER = {Lars\ \AA ke\ Andersson},
       URL = {http://projecteuclid.org/euclid.jdg/1090348112},
}

@article {Witten,
    AUTHOR = {Witten, Edward},
     TITLE = {A new proof of the positive energy theorem},
   JOURNAL = {Comm. Math. Phys.},
  FJOURNAL = {Communications in Mathematical Physics},
    VOLUME = {80},
      YEAR = {1981},
    NUMBER = {3},
     PAGES = {381--402},
      ISSN = {0010-3616,1432-0916},
   MRCLASS = {83C99 (58G25)},
  MRNUMBER = {626707},
MRREVIEWER = {Andrzej\ Trautman},
       URL = {http://projecteuclid.org/euclid.cmp/1103919981},
}

@misc {YuBlowup,
    doi = {10.7916/D8-AVNQ-G588},
    url = {https://academiccommons.columbia.edu/doi/10.7916/d8-avnq-g588},
    author = {Yu, Wenhua},
    title = {Blowup rate control for solution of Jang's equation and its application on Penrose inequality},
    publisher = {Columbia University},
    year = {2019},
    note = {https://arxiv.org/abs/1906.08841}
}

@article {Zhang,
    AUTHOR = {Zhang, Xiao},
     TITLE = {A definition of total energy-momenta and the positive mass theorem on asymptotically hyperbolic 3-manifolds. {I}},
   JOURNAL = {Comm. Math. Phys.},
  FJOURNAL = {Communications in Mathematical Physics},
    VOLUME = {249},
      YEAR = {2004},
    NUMBER = {3},
     PAGES = {529--548},
      ISSN = {0010-3616,1432-0916},
   MRCLASS = {83C30 (83C05)},
  MRNUMBER = {2084006},
       DOI = {10.1007/s00220-004-1056-0},
       URL = {https://doi.org/10.1007/s00220-004-1056-0},
}

\end{document}